\documentclass[11pt,leqno]{amsart}
\usepackage{amsthm,amsfonts,amssymb,amsmath,oldgerm,upgreek}
\numberwithin{equation}{section}
\usepackage{fullpage}
\usepackage{dsfont}
\usepackage{amsmath}
\usepackage{amsfonts}
\usepackage{amssymb}
\usepackage{setspace}
\usepackage{mathrsfs}
\usepackage{fancyhdr}
\usepackage{graphicx}
\usepackage[colorlinks = true, linkcolor=blue,
filecolor=darkgreen,
urlcolor=blue,
citecolor=red,pdftex]{hyperref}
\usepackage{psfrag}
\usepackage{listings}
\usepackage{bbm}

\usepackage[top=30mm,bottom=30mm,left=25mm,right=25mm,a4paper]{geometry}

\def\eps{\varepsilon }

\newcommand{\Id}{{\rm Id}}

\newcommand\br{\begin{remark}}
\newcommand\er{\end{remark}}
\newcommand\bp{\begin{pmatrix}}
\newcommand\ep{\end{pmatrix}}
\newcommand{\be}{\begin{equation}}
\newcommand{\ee}{\end{equation}}
\newcommand{\ba}[1]{\begin{array}{#1}}
\newcommand{\ea}{\end{array}}

\newcommand{\beg}{\begin{example}}
\newcommand{\eeg}{\end{exaplem}}
\newcommand{\bpr}{\begin{proposition}}
\newcommand{\epr}{\end{proposition}}
\newcommand{\bt}{\begin{theorem}}
\newcommand{\et}{\end{theorem}}
\newcommand{\bc}{\begin{corollary}}
\newcommand{\ec}{\end{corollary}}
\newcommand{\bl}{\begin{lemma}}
\newcommand{\el}{\end{lemma}}
\newcommand{\bd}{\begin{definition}}
\newcommand{\ed}{\end{definition}}
\newcommand{\brs}{\begin{remarks}}
\newcommand{\ers}{\end{remarks}}

\newtheorem{theorem}{Theorem}[section]
\newtheorem{proposition}[theorem]{Proposition}
\newtheorem{corollary}[theorem]{Corollary}
\newtheorem{lemma}[theorem]{Lemma}
\theoremstyle{remark}
\newtheorem{remark}[theorem]{Remark}
\theoremstyle{definition}
\newtheorem{definition}[theorem]{Definition}
\newtheorem{assumption}[theorem]{Assumption}

\newtheorem{example}[theorem]{Example}

\def\eps {\varepsilon}

\def\part{\partial}

\newcommand{\CC}{{\mathbb C}}

\newcommand{\MM}{{\mathbb M}}
\newcommand{\NN}{{\mathbb N}}

\newcommand{\PP}{{\mathbb P}}
\newcommand{\QQ}{{\mathbb Q}}
\newcommand{\RR}{{\mathbb R}}
\newcommand{\TT}{{\mathbb T}}

\newcommand{\ZZ}{{\mathbb Z}}

\newcommand\cA{{\mathcal A}}
\newcommand\cB{{\mathcal B}}

\newcommand\cF{{\mathcal F}}
\newcommand\cG{{\mathcal G}}

\newcommand\cI{{\mathcal I}}

\newcommand\cK{{\mathcal K}}
\newcommand\cL{{\mathcal L}}

\newcommand\cO{{\mathcal O}}

\newcommand\cQ{{\mathcal Q}}
\newcommand\cR{{\mathcal R}}

\newcommand\cT{{\mathcal T}}
\newcommand\cU{{\mathcal U}}
\newcommand\cV{{\mathcal V}}
\newcommand\cW{{\mathcal W}}

\newcommand\cY{{\mathcal Y}}

\newcommand{\myunderbar}[1]{\underline{#1\mkern-4mu}\mkern4mu }

\newcommand\rwidehat[1]{
  \savestack{\tmpbox}{\stretchto{
      \scaleto{
        \scalerel*[\widthof{\ensuremath{#1}}]{\kern.1pt\mathchar"0362\kern.1pt}
                 {\rule{0ex}{\textheight}}
                }{\textheight} 
    }{2.3ex}}
  \ensurestackMath{\stackon[-6.9pt]{#1}{\tmpbox}}
}\parskip 1ex

\newcommand{\mypar}{{\mkern3mu\vphantom{\perp}\vrule depth 0pt\mkern2mu\vrule depth 0pt\mkern3mu}}

\newcommand{\PG}{\mathscr{P}}

\usepackage{ulem}
\usepackage{cancel}

\title[The Fast Limit Model of the 
Euler--Maxwell--Two--Fluid system]{
 The Fast Limit Model associated with\\
the Euler--Maxwell--Two--Fluid system
}
\author[N. Besse]{Nicolas Besse}
\address[Nicolas Besse]{Observatoire de la C\^ote d'Azur,
               Bd de l'Observatoire CS 34229,
               06304 Nice Cedex 4, France}
\author[C. Cheverry]{Christophe Cheverry}
\address[Christophe Cheverry]{Institut Math\'ematique de Rennes,
Campus de Beaulieu, 263 avenue du G\'en\'eral Leclerc CS 74205
35042 Rennes Cedex\\FRANCE}

\begin{document}

\begin{abstract} The filtering method \cite{Sch05} applied at the level of the Euler--Maxwell--Two--Fluid system produces a Fast Limit Model (FLM) which captures up to the electron depth  essential features of plasma dynamics. In the case of prepared data, as recently proved in \cite{BC25p}, the discussion reduces to the eXtended MagnetoHydroDynamic (XMHD) framework of physicists \cite{AKY15,DML16,LMM15}, which involves the density $ \varrho $, the velocity $ u $ and the magnetic field $ B $ as state variables. By contrast, for unprepared data, an electric field $ E $ is created by  resonances, and it participates to the time evolution. It turns out that FLM is a well-posed system on $ (\varrho,u,E,B) $, extending XMHD, and implying a mechanism of interactions between $ (\varrho,u,B) $ and $ E $ which can convert a part of the energy carried by $ (\varrho,u,B) $ into electric energy.
\end{abstract}

\maketitle

{\small \parskip=1pt
\setcounter{tocdepth}{1}
\tableofcontents
}


\section{Introduction}
The three-dimensional Euler--Maxwell--Two--Fluid system, EMTF 
in abbreviated form, is a fundamental description in plasma physics that treats electrons and ions as two interpenetrating interacting fluids. EMTF falls within the scope of quasilinear symmetric systems, and therefore it is locally well-posed. From there, a lot of efforts have been done to investigate the long time existence, mainly with the help of dissipation or relaxation terms  \cite{DLZ12, ISYG18, XXK13}, and lately under smallness and irrotational conditions \cite{GIP16}. Without such assumptions, but  starting from perturbations (of small size $ \eps $ with $ 0 < \eps \ll 1$) around a constant state, an effective approach to the long time analysis of the fourth state of matter is provided by the filtering theory \cite{Sch05}. As noted in \cite{BC25p}, the two-fluid low Mach number limit of EMTF leads to a {\it Fast Limit Model} (that will be called {\it FL Model} or just {\it FLM}). The aim of the present paper is to establish preliminary results about the Fast Limit Model related to EMTF.

The FL Model can be deduced from the asymptotic analysis (under spatial periodic boundary conditions) of the forthcoming equation  \eqref{TFEM_cU}, where the penalized term $ \cL $ is a linear (partial differential) operator given by \eqref{DefL}. In doing so, some insights have been previously obtained by working under very restrictive assumptions, see for instance \cite{LPX15,LPW20,PL24,PXZ25} and the references therein. These preliminary steps bypass a number of crucial two-fluid effects and the significant impacts of the numerous interactions occurring between the rapid oscillations (of the type $ e^{{\rm i} \lambda t/\eps} $ with $ \lambda \in \RR^* $) generated by $ \cL $. 

By contrast, the filtering approach allows to progress in full generality, leading to a well-posed  Fast Limit Model.
Note that such FL Models have been thoroughly investigated in the context of incompressible limits of Euler (or Navier--Stokes)  systems \cite{DG99,Gal98,Gal03,LM98,Mas01,MS00,MS01,MS03,Sch94, Sch05,Sch07}. They have also been considered  for ideal MHD equations \cite{BC26,JJLX14,JSX19}. But, despite the key importance of EMTF in plasma physics, the FL Model associated with EMTF does not seem to have been considered before\footnote{In the present case, FLM takes the form of \eqref{modueffective} together with \eqref{expresscAj}, \eqref{expresscF}, \eqref{expresscA} and \eqref{characP}.}. This gap is probably due, among other things, to the technical complexity of this issue.

In the case of prepared data, which are special solutions to \eqref{TFEM_cU} satisfying $ \cL \,  \cU_\varepsilon = \cO(\eps) $, FLM reduces to 
the {\it Slow Limit Model} (that will be called {\it SL Model} or just {\it SLM}).  As proved in \cite{BC25p}, SLM\footnote{By construction, SLM  is made of \eqref{modu0effective} together with \eqref{characPe}.} is equivalent to incompressible eXtended MagnetoHydroDynamics (XMHD), a system of PDEs which can be written in terms of the state variables $ (\varrho,u,B) $. As stated in \cite{BC25}, 
the Cauchy problem for  compressible and incompressible XMHD is 
locally well-posed in Sobolev spaces. 
The XMHD framework was  first introduced in 1959 by L\"ust \cite{Lus59} via heuristic considerations. It was then developped to resolve issues arising from small-scale plasma physics (at electron and ion skin depths), in connection with the challenges of magnetic and inertial fusion. It is today intensively studied by physicists. It is investigated  through Hamiltonian methods \cite{AKY15,Bur17,DML16,KLMWW14,KM14,LMM15,LMM16},  through numerical simulations \cite{AGMDG14, AMDG14,ADG16,GTAM17,SM11}, and by means of  spectral or equilibria stability considerations \cite{AL24,ALM16,AY16,KTM20,LMM17,SMA24,Zhu17}. 
XMHD  is now recognized by the experts in the field as a more accurate alternative to standard MHD. From a mathematical perspective, the interest of \cite{BC25p} is to lay the foundation of XMHD on the basis of the first principles of two--fluid conducting fluids. In a word, XMHD is physically consistent and precise (it is able to grasp Hall and electron inertial effects) and mathematically canonical. 

At the same time, an outcome of the filtering method \cite{BC25p,Sch05} is to reveal 
the existence of a system encompassing XMHD, namely FLM. The FL Model is  more general in scope because it provides a record of the (long time impacts of the) interactions between fast waves. Theoretically, we find a two-level system, with one (XMHD) embedded within the other (FLM).
In practice, however, the distinction between the prepared und unprepared situations is not straightforward. On the one hand, under certain 
conditions\footnote{This may occur for example under non-periodic boundary conditions, in the presence of dispersion, or when some randomness is included.}, the fast singular limit of EMTF can reduce to SLM (and the use of FLM is not required). On the other hand, 
even in the periodic case, resonance phenomena are not guaranteed to manifest. Our primary objective is to prove that this latter point does not hold; in fact, the opposite is true.

\begin{theorem}\label{thmain}[A meaningful distinction between SLM and FLM] FLM can differ from SLM. In particular, FLM involves an electric field $ E $ created by resonances, while SLM lacks it.
\end{theorem}

Retain that FLM is a non-trivial extension of SLM. The fact remains that FLM is the most comprehensive  model, even if it is a much more complex system to map out. Another outcome, which stems from the above  separation of FLM from SLM, is that reserves of fluid and magnetic energy can accelerate or decelerate the charged particles. This is because the fast oscillating waves, which are generated at the level of \eqref{TFEM_cU}, are organized in such a way to produce an {\it electric field} $ E $, which is completely absent from SLM\footnote{Recall that an electric field $ E_X $ is hidden in the XMHD construction where it can be deduced from $ (\varrho,u,B) $ through the generalized Ohm's law, see Appendix of \cite{BC25p}. However, contrary to $ E $, the electric field $ E_X $ of XMHD is not created by the resonant interaction of fast waves. It does not interfere (directly) with the XMHD time evolution, and it is handled in XMHD as a corrector.}. This furnishes a new  mechanism by which some power stored inside $ (\varrho,u,B) $ can be converted into electric energy (related to $ E $). In this process, the defect of neutrality (the plasma is  quasineutral) and the two-fluids aspects play a dominant role. 

To study FLM, it might be tempting to apply Schochet's method \cite{Sch94,Sch05,Sch07}, see especially \cite{Sch07} for guidelines, which consists of decomposing the solution,  say $U (\tau,t,x)$, of FLM into a slow part $U_{\rm slow}(t,x) $ and a fast oscillating part $U_{\rm fast}(\tau,t,x) $, where $ \tau= t/\eps $. The relevance of this method for Euler (or Navier--Stokes) equations lies in two points. The first one is that the slow part $U_{\rm slow}$ satisfies some SL Model\footnote{This is related  to the fact that compressible Euler equations involve only fast longitudinal (acoustic) waves.}, and therefore it is not coupled to $U_{\rm fast}$. The second point is that the time-averaging operator
$ \MM_\tau $ given by \eqref{time-averaging}
 can somehow be eliminated from the FL Model \eqref{modu0} thanks to the introduction of a scalar potential (which satisfies a wave equation in the fast time variable $ \tau  $ and from which the fast oscillating part $U_{\rm fast}$ can be recovered).
    Nevertheless, for EMTF, due to the presence of fast transverse electromagnetic waves\footnote{The inventory of transverse and longitudinal waves can be found respectively in Paragraphs \ref{para1} and \ref{para2}.}, the preceding decoupling between $U_{\rm slow}$ and $U_{\rm fast}$ does not occur. Furthermore, the elimination of the time-averaging operator is far from being  straightforward\footnote{It would involve four scalar potentials, $\phi_{\perp 1}$, $\phi_{\perp 2}$, $\phi_{\mypar}^l$ and $\phi_{\mypar}^r$  respectively associated with the eigenvectors  $r_{\perp 1}$, $r_{\perp 2}$, $r_{\mypar}^l$ and $r_{\mypar}^r$ (see Section~\ref{subsubsec:compzeknot0}). The potentials  $\phi_{\perp 1}$ and $\phi_{\perp 2}$ (resp. $\phi_{\mypar}^l$ and $\phi_{\mypar}^r$) satisfy (resp.  perturbed by a second-order linear differential operator) Klein--Gordon equations whose symbols have a spatial part given by the square of the eigenvalue $\lambda_{\perp k}$ (resp. $\lambda_{\mypar k}^l$ and $\lambda_{\mypar k}^r$).}. It turns out that the preceding simplifications are  not available in the EMTF context. Still, it is possible to 
progress through the refined tools of nonlinear geometric optics. 

The preliminary step is to complete (on the Fourier side) the  spectral decomposition of $\cL$. This is already a fully separate matter, which appears implicitly in physics textbooks \cite{GP04} (but only in a partial way). 
In Section \ref{sec:specdecompo}, by exploiting symmetry and algebraic properties of $ \cL $, we are able to furnish a comprehensive treatment of all eigenvalues and all  eigenspaces of $\cL$. This precise description can serve different objectives, including the  classification of EMTF dispersion laws and of EMTF propagating waves. It is also a prerequisite for investigating the structure of resonances and transparencies, in the spirit of \cite{Met09,Rau12}.
This program is common and quite ancient in WKB analysis. Many theoretical results are available concerning symmetrizable quasilinear systems of conservation laws \cite{Met09,Rau12}. However, their implementation is often difficult when applied to real-world scenarios. As it happens, we aim at examining what can be done in the context of EMTF.
This setting is of course very specific but, on the other hand, it is strongly motivated by physics. It is inherited from preceding advances in this direction \cite{BC25,BC25p}. Now, it is quite complicated to obtain a full description of resonances and to check that transparencies do occur (or not). For this reason, an exhaustive study seems for the moment out of reach\footnote{Even in the more accessible context of Euler equations, the analysis of transparencies has been carried out in a number of stages \cite{Gal98,Mas01,Sch05}, and it is not completely finalized.}. 

In the present framework, the two solutions $ \PG \cU_0 $ and $ \PP_e \, \cU_0 $ of respectively 
FLM-\eqref{modueffective} and SLM-\eqref{characPe} are chosen periodic with respect to the spatial variable $ x $. In particular, we can compute according to \eqref{spamean} their corresponding (spatial) mean values $ \langle \PG \cU_0 \rangle $ and $ \langle \PP_e \, \cU_0 \rangle $. The strategy to prove Theorem \ref{thmain} is just to compare the time evolution of $ \langle \PG \cU_0 \rangle $ and $ \langle \PP_e \, \cU_0 \rangle $, and to show that they diverge. In this way, we can focus on the zero mode. It follows that the number of resonances to consider is strongly reduced. It suffices to examine the set $ R $ of resonances defined in \eqref{defredresonances}. Then, in Section \ref{sec:resonances},  exploiting the spectral analysis of $ \cL $, a dissection of the implied {\it resonant interactions} between oscillatory wave trains becomes available. 

That is not the end of the story. Indeed,  either all or part of resonances could be ineffective due to cancellations induced by the special configurations of the nonlinearities involved inside  EMTF. This is the well-known phenomenon of {\it transparency} \cite{CGM03}. In Sections \ref{transparenciesst} and \ref{transQL}, we examine what happens respectively at the level of the semilinear and quasilinear parts of FLM. Following these guidelines,  some precious (partial) information\footnote{The formulation of Theorem \ref{thmain} is deliberately kept a bit conceptual, to stay at a high level. In the sequel, the interplay between $ (\varrho,u,B) $ and $ E $ will be fully detailed, with explicit equations.} can be extracted. The semilinear term is not completely transparent (leading to Theorem \ref{thmain}), while the quasilinear term is.


\section{Background and notations}
While fixing some notation, we present  below the penalized system issued from EMTF by investigating near a backgroung state a regime of small amplitude $ \eps $ (with $ 0 < \eps \ll 1 $) and long time evolution ($ \tau \sim 1/ \eps $ while $ t = \eps \, \tau \sim 1 $). The reader is referred to \cite{BC25p}-Subsection 2.1 for a more step-by-step approach to the forthcoming material. In suitable unknowns
\[
\cU_{\varepsilon}={}^t(\varrho_{e\varepsilon},\varrho_{i\varepsilon},{}^tu_{e\varepsilon},{}^tu_{i\varepsilon},{}^tE_{\varepsilon},{}^tB_{\varepsilon}) = \cU_0 + \cO(\eps) \,, \qquad \cU_{\varepsilon} \in \RR \times \RR \times \RR^3 \times \RR^3\times \RR^3\times \RR^3 \, , 
\]
the penalized equations under consideration  reads
\begin{equation}
  \label{TFEM_cU}
\qquad  \partial_t \cU_\varepsilon - \frac{1}{\varepsilon} \, \cL \,\cU_\varepsilon =
   \sum_{j=1}^3 \cA_j(\eps,\cU_\eps) \, \part_{x_j} \cU_\eps + \cF(\eps,\cU_\eps)\,, \quad\quad  \cG \,\cU_\eps = \cO(\eps)\,, \quad\quad {\cU_\eps}_{|_{t=0}}=\cU_\eps^0\,,
\end{equation}
where $\cL$ is the ($ 14 \times 14 $ matrix-valued) linear differential operator 
\begin{equation}
  \label{DefL}
  \cL:=
    \, \left( \begin{array}{cccccc}
   0 & 0 & -\underline{a}_e \, \nabla \cdot  & 0 & 0 & 0 \\
   0 &  0 & 0 & -\underline{a}_i \, \nabla \cdot & 0 & 0 \\
   -\underline{a}_e \, \nabla & 0 & 0 & 0 & -{\rm I}_3\sqrt{\underline{n}/m_e} & 0 \\
   0 & -\underline{a}_i \, \nabla  & 0 & 0  & \ \ {\rm I}_3\sqrt{\underline{n}/m_i} & 0 \\
    0 & 0 & {\rm I}_3\sqrt{\underline{n}/m_e} & - {\rm I}_3\sqrt{\underline{n}/m_i} & 0  & \nabla \times\\
    0 & 0 & 0 & 0 & -\nabla \times  & 0
  \end{array}  \right) ,
\end{equation}
and where $\cG$ stands for the ($ 2 \times 14 $ matrix-valued) linear differential operator which is inherited from Gauss's constraints, namely
\begin{equation}
  \label{DefG}
\qquad  \cG \,\cU_\eps  = \cO(\eps) , \qquad \cG:=
  \left( \begin{array}{cccccc}
   0 & 0 & 0 & 0 & 0 & \nabla \cdot  \\
 \displaystyle \sqrt{\underline{n}/\underline{p}'_e} & \displaystyle - \sqrt{\underline{n}/\underline{p}'_i}  & 0 & 0 & \nabla \cdot  & 0
  \end{array}  \right) . 
\end{equation}
In \eqref{TFEM_cU}, the three expressions 
$\cA_j(\eps,\cU)$ represent matrices of size $ 14 \times 14 $ depending smoothly on $ \eps \in [0,1] $ and $ \cU \in \RR^{14} $ according to
\begin{equation}
  \label{expresscAj}
\qquad \cA_j(\eps,\cU) :=-\left( \begin{array}{cccccc}
  \displaystyle \frac{u_{e}^j}{\sqrt{{\underline n} m_e}} & 0 & \myunderbar{\mathscr R}_\eps (a_e) \, {}^t e^j & 0 & 0 & 0 \\
  0 &  \displaystyle \frac{u_{i}^j}{\sqrt{{\underline n} m_i}} & 0 & \myunderbar{\mathscr R}_\eps (a_i) \, {}^t e^j & 0 & 0 \\
  \myunderbar{\mathscr R}_\eps (a_e) \, e^j & 0 &  \displaystyle \frac{u_{e}^j}{\sqrt{{\underline n} m_e}} \, e^j\otimes e^j & 0 & 0 & 0 \\
  0 & \myunderbar{\mathscr R}_\eps (a_i) \, e^j & 0 & \displaystyle \frac{u_{i}^j}{\sqrt{{\underline n} m_i}} \, e^j\otimes e^j & 0 & 0 \\
  0 & 0 & 0 & 0 & 0  & 0 \\
  0 & 0 & 0 & 0 & 0  & 0 
\end{array}  \right) .
\end{equation}
In \eqref{TFEM_cU}, the source term $ \cF ={}^t (\cF_{\varrho_e},\cF_{\varrho_i},{}^t\cF_{u_e},{}^t\cF_{u_i},{}^t\cF_E,{}^t\cF_B) $ is specified by
\begin{equation}
  \label{expresscF}
  \qquad \left\{\begin{array}{l}
    \cF_{\varrho_e}=\cF_{\varrho_i}=\cF_B=0\,, \medskip\\
    \cF_{u_e} (\eps,\cU_\eps) := - \, u_{e\varepsilon}\times B_{\varepsilon}/{m_e} \, , \smallskip \\
    \cF_{u_i} (\eps,\cU_\eps) := + \, u_{i\varepsilon}\times B_{\varepsilon}/{m_i} \, , \smallskip \\
    \displaystyle \cF_E (\eps,\cU_\eps) := \, \myunderbar{\mathscr R} (g_e^{-1})\Big(\eps, \frac{ \varrho_{e\varepsilon}}{\sqrt{{\underline n} m_e}} \Big) \ \frac{u_e}{ \sqrt{{\underline n} m_e} } - \myunderbar{\mathscr R} (
    g_i^{-1}) \Big(\eps,\frac{ \varrho_{i\varepsilon}}{\sqrt{{\underline n} m_i}} \Big) \ \frac{u_i}{ \sqrt{{\underline n} m_i} } \, .
  \end{array}\right.
\end{equation}
Above, the vector $e^j$ with $ j \in \{1,2,3\} $ denotes the $j$-th vector of the canonical basis of $ \RR^3 $. The positive real numbers $ m_e $ and $ m_i $ (with $ 0 < m_e < m_i $) stand for the masses of electrons and ions. We work away away from vacuum, near equilibrium densities $ \underline{\rm n}_e= \underline{\rm n}_i=\underline{n} \in \RR_+^* $. The underlined letters indicate that we look at the proximity of the constant solution $ {}^t (\underline{n},\underline{n},0,0,0,0) $.

With this in mind, given a fixed position $ \underline{q} \in \RR $ and a function $ h \in \mathscr C^\infty (\RR;\RR) $, 
we consider the function $ \mathscr R (h,\underline{q}): [0,\eps_0] \times \RR \rightarrow \RR $ determined by
\begin{eqnarray*}
\mathscr R (h,\underline{q}) (\eps,q) & := & \frac{1}{\varepsilon}\, \big(h(\underline{q} +\varepsilon q) - h(\underline{q})\big) = \int_0^1 h'(\underline{q} +\varepsilon q \, r) \, q \ dr \\
&\,=&  h'(\underline{q}) \, q \, +\,  \cO(\eps) \, , \quad \mathscr R (h,\underline{q}) (0,q) = h'(\underline{q}) \, q \, .
\end{eqnarray*}
When $ \underline{q} = 0 $, we just use $ \myunderbar{\mathscr R} (h) := \mathscr R (h,0) $. Then, we set 
\[ \myunderbar{\mathscr R}_\eps (a_s) \equiv \myunderbar{\mathscr R} (a_s)(\eps,q_{s\eps}) \, , \qquad q_{s\eps} := \varrho_{s\varepsilon} / \sqrt{\underline{n} \, m_s} \, ,  \]
which allows to interpret $ \myunderbar{\mathscr R}_\eps (a_s) $ as a function of the unknown state variable $ \varrho_{s\varepsilon} $. 
This explains some notation inside $ \cA_j $ and $ \cF_E $, where $ h $ is replaced by $ a_s $ and $ g_s^{-1} $. In particular, retain that
\begin{equation}
  \label{retainforbaras}
    \myunderbar{\mathscr R}_0 (a_s) = \frac{a'_s(0)}{\sqrt{\underline{n} \, m_s}} \ \varrho_{s\varepsilon} \, , \qquad \forall s \in \{e,i \} \, . 
\end{equation}
The origin of $ a_s $ and $ g_s $ is as follows. We start with pressure laws $ p_s : \RR_+ \rightarrow \RR_+^* $ with $  s\in\{e,i\} $ that are strictly increasing. Let us define $ \underline{p}_s' := p_s'(\underline{n}) $. Then, instead of working with  the original densities $ {\rm n}_s$, and in order to symetrize the original EMTF system, we use
\begin{equation*}
  \label{implens}
        {\rm q}_s := g_s({\rm n}_s)\,, \qquad g_s({\rm n}_s):=\frac{1}{\sqrt{m_s}}\int_{\underline{n}}^{{\rm n}_s} \frac{\sqrt{p_s'(\theta)}}{\theta}\ d\theta\,,
        \qquad s\in\{e,i\}\, . 
\end{equation*}
 This change of density variables leads to two scalar functions $ a_s : \RR \rightarrow \RR_+ $ given by
\begin{equation*}
  \label{changedensity}
  \qquad \qquad \ a_s({\rm q}_s) :=g_s^{-1}({\rm q}_s) \, g_s'\circ g_s^{-1}({\rm q}_s)=g_s'({\rm n}_s) \, {\rm n}_s= \sqrt{p_s'({\rm n}_s) /m_s}\,, \qquad s\in\{e,i\}\,,
\end{equation*}
whose values at the image  $ g_s (\underline{n}) = 0 $ of the equilibrium position $ \underline{n}\neq0 $ are such that
\begin{equation}
  \label{defunderlinea}
  \underline{a}_s := a_s(\underline{\rm q}_s) =a_s(0) =  \sqrt{\underline{p}_s'/m_s} >0\,, \qquad s\in\{e,i\}\,.
\end{equation}
We will also exploit the notation
\begin{equation}
  \label{defunderlineb}
  \underline{b}_s :=\sqrt{\underline{n}/m_s} >0 \,, \qquad s\in\{e,i\}\,;  \qquad \underline{b}=\sqrt{\underline{b}_e^2+\underline{b}_i^2}\,.
\end{equation}
Note that $ \cL $ is only parametrized by $ \underline{a}_s $ and $ \underline{b}_s $ with $ s\in\{e,i\} $.
At this stage, all the ingredients implied at the level of  \eqref{TFEM_cU}-\eqref{DefL}-\eqref{DefG}-\eqref{expresscAj}-\eqref{expresscF} have been identified. These equations serve as our starting point. We refer the reader to \cite{BC25p,GP04} for a detailed account of their origin. For later use, following \cite{BC25p}, it is important to recall the transformation between the original unsymmetrized variables $U_0= {}^t(n_{e0}, n_{i0}, {}^t v_{e0},  {}^t v_{i0},{}^t E_{0},{}^t B_{0})$ of the EMTF system in \cite{BC25p} and the convenient symmetrized ones for the asymptotic analysis, i.e. $\cU_0\,$:
\begin{equation*}
  \label{passagecUU0}
  \begin{array}{rl}\cU_0 \! \! \! & = \cA_0^{1/2} \,  U_0 = {}^t(\varrho_{e0},\varrho_{i0},{}^tu_{e0},{}^tu_{i0},{}^tE_{0},{}^tB_{0}) \smallskip \\
    \ & = {}^t \Bigl(\sqrt{\underline{p}'_e/\underline n} \ n_{e0},\sqrt{\underline{p}'_i/\underline n} \ n_{i0},\sqrt{\underline n \, m_e} \ {}^t v_{e0}, \sqrt{\underline n \, m_i} \ {}^t v_{i0},{}^t E_{0},{}^t B_{0} \Bigr) \, .  \end{array}
\end{equation*}
where
\begin{equation*}
  \cA_0:=\left( \begin{array}{cccccc}
    \underline{n} \, m_e & 0 & 0 & 0 & 0 & 0 \\
    0 & \underline{n} \, m_i & 0 & 0 & 0 & 0 \\
    0 & 0 & \underline{n}\, m_e \, {\rm I}_{3\times 3} & 0 & 0 & 0 \\
    0 & 0 & 0 & \underline{n} \, m_i \,  {\rm I}_{3\times 3} & 0 & 0 \\
    0 & 0 & 0 & 0 &  {\rm I}_{3\times 3}  & 0 \\
    0 & 0 & 0 & 0 & 0  &  {\rm I}_{3\times 3} 
  \end{array}  \right)\,.  
\end{equation*}

Now, we explain briefly how FLM and SLM can be derived from EMTF, that is from \eqref{TFEM_cU}. Given some function $ \psi(\tau,t,x) $ which is almost periodic with respect to $ \tau $, we can compute
\begin{equation}\label{time-averaging}
\MM_\tau \bigl( \psi(\tau,t,x) \bigr) := \lim_{\cT \rightarrow + \infty} \frac {1}{\cT}\int_0^\cT  \psi (\tau,t,x) \ d \tau \,.
\end{equation}
The map $\MM_\tau$ is a time-averaging (or mean value) operator. 
From \cite{BC25p,Sch05}, there exists a profile 
\[ \cU_0=\cU_0(t,x) \in L^\infty([0,T];H^{s}) \cap  W^{1,\infty}([0,T];H^{s-1}) \, , \qquad s>5/2 \, ,\] such that
\begin{equation*}
  \label{difflim}
  \forall \sigma <  s \,, \qquad \lim_{\eps \rightarrow 0 +} \| \cU_\eps (t,x) - e^{t/\eps} \, \cU_0 (t,x) \|_{\mathscr{C}([0,T];H^{\sigma})} = 0\, .
\end{equation*}
Here $ H^s \equiv H^s (\TT^3) $ is the space of periodic functions (equipped with the standard Sobolev norm). 
It turns out that the profile $ \cU_0 $ is the unique solution of the modulation equation 
\begin{equation}
  \label{modu0}
  \qquad \ \part_t \cU_0 = \MM_\tau \Bigl \lbrack e^{- \tau \, \cL } \, \bigl( A(e^{\tau \, \cL } \, \cU_0,D_x)
  \, e^{\tau \, \cL } \, \cU_0 \bigr) + e^{- \tau \, \cL } f(e^{\tau \, \cL } \, \cU_0) \Bigr \rbrack \,,
\end{equation}
associated with the initial data
\begin{equation}
  \label{modu0ini0}
{\cU_0}_{|_{t=0}} = \cU^0_0 \,.
\end{equation}
In \eqref{modu0} the quasilinear term $A(\cdot,D_x)$ and the source term $f$ are given by
\begin{equation}
  \label{expresscA}
  A(\cU_0,D_x) \cU_0= \sum_{j=1}^3 A_j(\cU_0) \, \partial_{x_j} \cU_0\,, \quad A_j(\cU_0)=\cA_j(0,\cU_0)\, , \quad \quad f(\cU_0)=\cF (0,\cU_0)\,.
\end{equation}
The solution of \eqref{modu0} must also satisfy conditions coming from \eqref{DefG}, namely
\begin{equation}
  \label{Gausslaws}
  \cG \, \cU_0 =0\,.
\end{equation}
Let $\PG$ be the orthogonal projection onto $ {\rm Ker}\, \cG $, which may be identified through the conditions
\begin{equation}
  \label{characP}
  \PG^2 = \PG \, , \qquad \PG = \PG^* \, , \qquad \text{\rm Im} \, \PG = {\rm Ker}\, \cG \,.
\end{equation}
The profile $\cU_0$ can satisfy  \eqref{modu0ini0}-\eqref{Gausslaws} provided that $\cU_0^0 = \PG \cU^0_0$. Then, applying the projector $\PG$ to \eqref{modu0} and using Lemma \ref{compro} (in Subsection~\ref{ss:ACP}), it is a solution of \eqref{modu0}-\eqref{Gausslaws} if and only if $ \cU_0 = \PG \cU_0 $ satisfies the following Fast Limit Model
\begin{equation}
  \label{modueffective}
  \begin{array}{rl}
  \displaystyle  \part_t \cU_0 = \! \! \! & \displaystyle \MM_\tau \Bigl \lbrack e^{- \tau \PG \cL \PG} \PG \bigl( A(e^{\tau \PG \cL \PG} \PG \cU_0,D_x)
  \, e^{\tau \PG \cL \PG} \PG \cU_0 \bigr) \Bigr \rbrack \\
  \ & + \displaystyle \, \MM_\tau \Bigl \lbrack  e^{- \tau \PG \cL \PG} \PG f(e^{\tau \PG \cL \PG} \PG \cU_0) \Bigr \rbrack \, . 
  \end{array}
\end{equation}
The first line of the right hand side of \eqref{modueffective}
 will be referred to as the {{\it quasilinear} term; the second will be called the 
 {\it semilinear} (or source) term. The implementation of $ \PG $ allows to replace 
 \eqref{modu0}-\eqref{Gausslaws} by the sole equation \eqref{modueffective}.

\noindent The waves polarized in the kernel of $ \cL $ do not oscillate rapidly with respect to $ t $, because a large time derivative is not activated at the level of \eqref{TFEM_cU}. Such waves can be put apart by applying the orthogonal projection $ \PP_e $ onto $\text{Ker}\, \cL \cap  \text{Ker}\, \cG $.  Retain that
\begin{equation}
  \label{characPe}
  \PP_e^2 = \PP_e \, , \qquad \PP_e = \PP_e^* \, , \qquad \text{\rm Im} \, \PP_e = \text{Ker}\, \cL \cap {\rm Ker}\, \cG \,.
\end{equation}
When the initial data $\cU_0^0$ is prepared, that is when $\cU_0^0=\PP_e \, \cU_0^0$ , the equation \eqref{modueffective} reduces to the (effective) SL Model
\begin{equation}
  \label{modu0effective}
 \qquad  \partial_t (\PP_{\! e} \, \cU_0) =   \sum_{j=1}^3 \PP_{\! e} \, A_j(\PP_{\! e} \, \cU_0) \, \part_{x_j} \PP_{\! e} \, \cU_0 + \PP_{\! e} \, f(\PP_{\! e} \,\cU_0)\, , \quad \  \cU_0^0 = \PP_{\! e} \, \cU^0_0 \,.
\end{equation}
In \cite{BC25p} it was shown that \eqref{modu0effective} is the same system as the incompressible XMHD given by
\begin{equation}\label{putxmhd}
  \left \lbrace \begin{array}{l}
    \displaystyle \part_t u_0 + (u_0\cdot \nabla ) u_0 + \nabla p_0 +  \myunderbar{\rho}^{-1}B^*_0 \times (\nabla \times B_0) = 0 \,, \medskip \\
    \displaystyle \part_t B^*_0 + \nabla \times \bigl( B^*_0 \times (u_0 - (d_i/ \myunderbar{\rho}) \,\nabla \times B_0) \bigr) \smallskip\\
    \qquad \  + \, (d_e^2/\myunderbar{\rho})\ \nabla \times \bigl( (\nabla \times u_0)
    \times (\nabla \times B_0 ) \bigr) = 0 \,,
  \end{array} \right.
\end{equation}
where the unknowns $ (u_0,B^*_0) \in \RR^3 \times \RR^3 $  are divergence free, i.e.,
\begin{equation*}
  \label{divfreeini}
  \nabla \cdot u_0 = 0 , \qquad \nabla \cdot B^*_0 = 0 , \qquad \nabla \cdot B_0 = 0 \,.
\end{equation*}
The part $ u_0 $ is akin to the center of mass velocity, which is
\begin{equation*}\label{centermasselocity}
 u_0 := \frac{m_e \, v_{e0} + m_i \, v_{i0}}{m_e + m_i} \, .
\end{equation*}
The part $ B^*_0 $ is connected to the usual magnetic field $ B_0 $ through the constitutive relation
\begin{equation*}
  \label{lienBB*2}
  B_0 = \big(\Id - \underline{b}^{-2} \,\Delta\big)^{-1} B^*_0 \, , \qquad {\underline b} = \sqrt{\underline{b}_e^2+\underline{b}_i^2}= d_e^2/\myunderbar{\rho} \, .
\end{equation*}
Above, the dimensionless parameters  $\myunderbar{\rho}$,  $d_e$, and $d_i$ which stand respectively for the normalized total mass density, the normalized electron skin depth (resurgence of electronic inertial effects), and the normalized ion skin depth (resurgence of ionic inertial effects or Hall effects). These parameters are given in terms of the dimensionless density $\underline{n}$,  the dimensionless electronic mass $m_e$, the dimensionless ionic mass $m_i$, and the charge number $Z$ according to the formulas
\begin{equation*}
  \label{eqn:cteXMHD}
\qquad  \myunderbar{\rho}:=\underline{n} \, (m_e+ m_i) \, , \qquad d_e :=\sqrt{\updelta}\  \frac{m_i}{Z} \, , \qquad d_i := (1-\updelta) \, \frac{m_i}{Z} \,, \qquad \updelta := Zm_e/m_i \, .
\end{equation*}
Since the context is periodic in $ x $, the discussion is well-suited for a Fourier analysis. The solutions are real vector valued periodic functions which may be represented as Fourier series according to
\[
L^2 (\TT^3; \RR^{14}) \ni \cU(x) = \sum_{k \in \ZZ^3} \cU_k \ e^{{\rm i} \, k \cdot x} \, , \qquad \cU_k = \bar \cU_{-k}\in \CC^{14} \, , \qquad \TT := \RR /(2 \pi \ZZ) \, .
\]
The Fourier coefficients $ \cU_k $ are vectors of $ \CC^{14} $. The complex vector space $ \CC^{14} $ is endowed with the Hermitian inner product
$ \langle\cU,\cV \rangle = {}^{t} \bar \cU \, \cV = \bar\cU \cdot \cV $. When $ k $ is (implicitely) specified, we will drop the subsript $ k $, and $ \cU_k \equiv \cU $ may be represented according to
\begin{equation}\label{decompocUen}
\cU = {}^t (\varrho_e,\varrho_i,{}^t u_e,{}^t u_i,{}^t E,{}^t B)\in \CC \times \CC \times \CC^{3} \times \CC^{3} \times \CC^{3} \times \CC^{3} \, .
\end{equation}
The four last components of $ \cU $, namely $ u_e $, $ u_i $, $ E $ and $ B $, are vectors in $ \CC^3 $, which can be further broken up into
\[ u = \sum_{j=1}^3 \, u_j \ e^j \, , \qquad u_j := {}^t e^j  u \, , \qquad u \in \{u_e,u_i,E,B \} \, , \]
where $ \{ e^j \}_{j= 1,2,3} $ is the canonical basis of $\RR^3 $. Depending on the choice of $ k \in \ZZ^3 $ with $ k \not = 0 $, we will also write $ u $ in the form
\begin{equation}\label{deroulu}
 u = \sum_{j=1}^3 \, \mathbbm{u}_j \ \mathbbm{e}^j_k \, , \qquad \mathbbm{u}_j := {}^t \mathbbm{e}^j_k  u \, , \qquad u \in \{u_e,u_i,E,B \} \, ,
\end{equation}
where $ \{ \mathbbm{e}^j_k \}_{j= 1,2,3} $ is a right handed basis adjusted in such a way that $\mathbbm{e}^1 \equiv \mathbbm{e}^1_k:=k/|k|$. More precisely, the unit direction $ \mathbbm{e}^1_k $ is completed with $(\mathbbm{e}^2_k,\mathbbm{e}^3_k)$ so that\footnote{Make sure to distinguish between the vectors $  e^j $ and $ \mathbbm{e}^j_k $ (given that we will sometimes omit to mention $ k $ when dealing with $ \mathbbm{e}^j_k $). Also be careful that $ \mathbbm{e}^j_k $ is odd with respect to $ k $ for $ j = 1 $ or $ j= 2 $, and it is even for $ j = 3 $.}
\begin{equation}\label{proprme}
\mathbbm{e}^1_k \wedge \mathbbm{e}^2_k = \mathbbm{e}^3_k \, , \qquad \mathbbm{e}^1_{-k} = - \mathbbm{e}^1_k \, , \qquad \mathbbm{e}^2_{-k} = - \mathbbm{e}^2_k \, , \qquad \mathbbm{e}^3_{-k} = +  \mathbbm{e}^3_k \, .
\end{equation}
Locally, near any $ \tilde k \in \RR^3 \setminus \{ 0 \} $, the fields $ k \rightarrow \mathbbm{e}^j_k $ can be chosen of class $ C^\infty $. To this end, it suffices to select some (fixed) vector
$ \tilde {\mathbbm{e}} $ which is not colinear to $ \tilde k $, and then define $ \mathbbm{e}^2_k := \mathbbm{e}^1_k \times \tilde {\mathbbm{e}} / \mypar \mathbbm{e}^1_k \times \tilde {\mathbbm{e}} \mypar $, as well as $
\mathbbm{e}^3_k := \mathbbm{e}^1_k \times \mathbbm{e}^2_k $. It is clear that we have \eqref{proprme}. Note that the fields $ k \mapsto \mathbbm{e}^j_k $ cannot be extended smoothly on the whole of $ \RR^3 \setminus \{ 0 \} $ due to the \href{https://en.wikipedia.org/wiki/Hairy_ball_theorem}{hairy ball theorem}. Now, given  a vector $ A \in \CC^3 $ and $ k \in \ZZ^3 $ with $ k \not = 0 $, we can separate the components which are respectively parallel and orthogonal to the direction $ k $ according to
\[
A_\mypar := \bigl( {}^t \mathbbm{e}^1_k  A \bigr) \  \mathbbm{e}^1_k \, , \qquad A_\perp := \bigl( {}^t \mathbbm{e}^2_k  A \bigr) \ \mathbbm{e}^2_k + \bigl( {}^t \mathbbm{e}^3_k  A \bigr) \ \mathbbm{e}^3_k \, , \qquad A = A_\mypar + A_\perp \, .
\]
Recall that $ \cL $ and $ \cG $ are the two matrix valued differential operators defined by \eqref{DefL} and \eqref{DefG} with symbols $ \cL(i\xi) $ and $ \cG(i\xi)$. The actions of $ \cL $ and $ \cG $ on periodic functions can be expanded according to
\[
\cL \, \cU = \sum_{k \in \ZZ^3} \cL_k \, \cU_k \ e^{{\rm i} \, k \cdot x} \, , \qquad \cG \, \cU = \sum_{k \in \ZZ^3} \cG_k \, \cU_k \ e^{{\rm i} \, k \cdot x} \, .
\]
Above, $ \cL_k := \cL(i k ) $ and $ \cG_k := \cG(i k ) $ are $ 14 \times 14 $ and $ 2 \times 14 $ complex valued matrices given by
\begin{equation} \label{DefLk}
  {\cL}_k := \left( \begin{array}{cccccc}
    0 & 0 & -\, {\rm i}\, \underline{a}_e\, {}^tk & 0 & 0 & 0 \\
    0 &  0 & 0 &   - \,{\rm i} \, \underline{a}_i \, {}^tk & 0 & 0 \\
     - \,{\rm i}\, \underline{a}_ek & 0 & 0 & 0 &  -\,\underline{b}_e\,{\rm I}_3 & 0 \\
     0 &   -\, {\rm i} \, \underline{a}_i k & 0 & 0  & \ \ \,\underline{b}_i\,{\rm I}_3 & 0 \\
     0 & 0 & \underline{b}_e\,{\rm I}_3 & -\,\underline{b}_i\,{\rm I}_3 & 0  &  {\rm i}\,k\times\\
     0 & 0 & 0 & 0 &  -\, {\rm i}\, k\times & 0
  \end{array}  \right) ,
\end{equation}
where by $ \underline{a}_s $ and $ \underline{b}_s $ are defined by \eqref{defunderlineb}, and
\begin{equation}\label{DefGk}
\  {\cG}_{k}:=
  \left( \begin{array}{cccccc}
    0 \, & 0 & 0 & 0 & 0 &  {\rm i} \,{}^tk  \\
    \displaystyle \sqrt{\underline{n}/\underline{p}'_e} \, &  \displaystyle -\sqrt{\underline{n}/\underline{p}'_i}  & 0 & 0 &  {\rm i}\, {}^tk  & 0
  \end{array}  \right) =
  \left( \begin{array}{cccccc}
    0 \, & 0 & 0 & 0 & 0 &  {\rm i} \,{}^tk  \\
    \displaystyle \underline{b}_e/\underline{a}_e \, &  \displaystyle - \underline{b}_i / \underline{a}_i & 0 & 0 &  {\rm i}\, {}^tk  & 0
  \end{array}  \right) .
\end{equation}
Similarly, the actions of $ \PG $ and $ \PP_e $ may be specified on each Fourier coefficient. For instance
\begin{equation}
  \label{characPkk}
  \PG_k^2 = \PG_k \, , \qquad \PG_k = \PG_k^* \, , \qquad \text{\rm Im} \, \PG_k = {\rm Ker}\, \cG_k \,.
\end{equation}
Starting from these conventions, the rest of the paper is devoted to the proof of Theorem~\ref{thmain}.


\section{Spectral analysis of the penalized operator.}\label{sec:specdecompo}
Let us come back to  \eqref{modueffective}. In this evolution equation (with respect to $ t $ and $ x $), the right hand side is written in a rather abstract form which does not reveal the numerous oscillating terms which may disappear under the action of the (time) mean value operator $ \MM_\tau $. The non-zero contributions are those which emanate from resonances. To identify the structure of resonances, it is necessary to compute $ e^{\pm \tau  \PG \cL \PG} $. To this end, we need to exhibit the spectral decomposition of the operator $ \PG \cL \PG $. We begin in Subsection~\ref{ss:ACP} by stating a commutative property. We continue with Subsection \ref{subsec:cGk} where we give a complete description of $ \text{Ker} \, \cG_k $ and $ (\text{Ker} \, \cG_k)^\perp $. In Subsection \ref{subsec:genfor}, we provide preliminary information on $ \cL_k $. In Subsection \ref{subsec:genforPP}, we incorporate the role of $ \PG_k $. In Subsections \ref{subsubsec:compze} and \ref{subsubsec:compzeknot0}, we deal respectively with the cases $ k = 0 $ and $ k \not = 0 $.

\subsection{A commutative property}
\label{ss:ACP}
The matrix $ \cL_k $ is clearly skew-adjoint ($\cL_k^* = {}^t\bar{\cL}_k=-\cL_k$); its eigenvalues are therefore  purely imaginary, of the form $ {\rm i} \, \lambda $ with $ \lambda \in \RR $. Retain that $ \cL^* = - \cL $. 

\begin{lemma}[A commutative property]\label{compro}
We have $ \cL = \PG \cL= \cL \PG = \PG \cL \PG $.
\end{lemma}

\begin{proof} Using  \eqref{DefLk}, we find
\begin{equation}\label{cLkcU}
\cL_k \, \cU =
\left(
\begin{array}{c}
  -{\rm i}\underline{a}_e \, {}^t k u_e\\
  -{\rm i}\underline{a}_i \, {}^t k u_i\\
  -{\rm i}\underline{a}_e k \varrho_e- \underline{b}_e E \\-{\rm i}\underline{a}_i k \varrho_i+ \underline{b}_i E\\
  \underline{b}_e u_e - \underline{b}_i u_i + {\rm i}k\times B \\ -{\rm i}k\times E
\end{array}
\right) .
\end{equation}
Then, from \eqref{DefGk}, we deduce that
\[ \cG_k \cL_k \, \cU =
\left(
\begin{array}{c}
  {}^t k (k\times E) \\
  -{\rm i}\underline{b}_e \, {}^t k u_e + {\rm i}\underline{b}_i \, {}^t k u_i +
  {\rm i} k \cdot (\underline{b}_e u_e - \underline{b}_i u_i + {\rm i}k\times B)
\end{array}
\right) =
\left(
\begin{array}{c}
  0 \\
  0
\end{array}
\right) \,, \qquad \forall \, \cU \in \CC^{14} \,. \]
Since $ \PG_k $ is the orthogonal projector onto the kernel of $ \cG_k $, we have
\begin{equation}
  \label{eqnGL0}{\rm Im\, } \cL_k \subset  \text{Ker} \, \cG_k \, , \qquad \cL_k = \PG_k \cL_k \, , \qquad \cL_k^* = \cL_k^* \PG_k \, , \qquad \forall \, k \in \ZZ \, .
\end{equation}
Since $ \cL_k $ is skew-adjoint, from the last identity inside \eqref{eqnGL0}, we get that $ - \cL_k = -  \cL_k \PG_k $, and therefore $ \cL_k =  \PG_k \cL_k = \cL_k \PG_k = \PG_k \cL_k \PG_k $, which is equivalent to the expected result.
\end{proof}
Retain that
\begin{equation*}
\Id_{\CC^{14}} = \text{Ker}\, \cL_k \oplus {\rm Im\, } \cL_k = \text{Ker}\, \cG_k \oplus ( \text{Ker}\, \cG_k)^\perp\,, \quad {\rm Im\, } \cL_k \subset \text{Ker}\, \cG_k\,.
\end{equation*} 
In accordance with \cite{BC25p}, we then define the following orthogonal Projections (Proj)
\begin{eqnarray*}
&&  \PP:= \text{Proj}_{|_{ \text{Ker}\, \cL}}\,,  \quad  \QQ:= \text{Proj}_{|_{ {\rm Im\, } \cL}}\,,   \quad  \PP_e:= \text{Proj}_{|_{ \text{Ker}\, \cL \cap  \text{Ker}\, \cG }}\,, \\
&&  \PG:=  \text{Proj}_{|_{\text{Ker}\, \cG}}=\text{Proj}_{|_{ (\text{Ker}\, \cL \cap  \text{Ker}\, \cG) \oplus  {\rm Im\, } \cL}}= \PP_e + \QQ\,. 
\end{eqnarray*}  


\subsection{Structure of $ \text{Ker} \, \cG_k $ and $ (\text{Ker} \, \cG_k)^\perp $.}\label{subsec:cGk}
Recall that
\[ (\text{Ker} \, \cG_k)^\perp = \{ \, \cU \in \CC^{14} \ ; \ \bar \cU \cdot \cV = 0\,, \ \forall \,\cV \in \text{Ker} \, \cG_k \, \} \,. \]
In the sense of complex vector subspaces of $ \CC^{14}  $, we have $ \CC^{14} = \text{Ker} \, \cG_k \oplus (\text{Ker} \, \cG_k)^\perp $.


\subsubsection{The case $ k = 0 $.}\label{subsubsec:k=0}
We find that ${\rm dim\,}(\text{Ker}\, \cG_0)=13$ and $\text{Ker}\, \cG_0 = {\rm Vect}\{g_0^l\}_{l\in\{1,\ldots,13\}}$ with
\[ ^tg_0^1 :=\bigl(\underline{p}'_e +\underline{p}'_i \Bigr)^{-1/2}\ {}^t \big(\sqrt{\underline{p}_e'},\sqrt{\underline{p}_i'},0,0,0,0\big) = \Big(\frac{\underline{a}^2_e}{\underline{b}^2_e} + \frac{\underline{a}^2_i}{\underline{b}^2_i} \Big)^{-1/2}\
  {}^t\Big(\frac{\underline{a}_e}{ \underline{b}_e},\frac{\underline{a}_i}{ \underline{b}_i},0,0,0,0\Big)\, ,
\]
and
\[ \begin{array}{lll}
{}^tg_0^2:={}^t(0,0,{}^te^1,0,0,0)\,, \qquad &  {}^tg_0^3:={}^t(0,0,{}^te^2,0,0,0)\,, \qquad &  {}^tg_0^4:={}^t(0,0,{}^te^3,0,0,0)\,, \smallskip \\{}^tg_0^5:={}^t(0,0,0,{}^te^1,0,0)\,, &  {}^tg_0^6:={}^t(0,0,0,{}^te^2,0,0)\,, &  {}^tg_0^7:={}^t(0,0,0,{}^te^3,0,0)\,, \smallskip \\{}^tg_0^8:={}^t(0,0,0,0,{}^te^1,0)\,, &  {}^tg_0^9:={}^t(0,0,0,0,{}^te^2,0)\,, &{}^tg_0^{10}:={}^t(0,0,0,0,{}^te^3,0)\,, \smallskip \\  {}^tg_0^{11}:={}^t(0,0,0,0,0,{}^te^1) \, , &{}^tg_0^{12}:={}^t(0,0,0,0,0,{}^te^2) \, , &{}^tg_0^{13}:={}^t(0,0,0,0,0,{}^te^3) \,.
\end{array} \]
Furthermore
\[
(\text{Ker} \, \cG_0)^\perp = \text{Vect} \ \{ \mathfrak{g}_0^1 \} \, , \qquad
  \mathfrak{g}_0^1 := \displaystyle \Big(\frac{\underline{b}^2_e}{\underline{a}^2_e} + \frac{\underline{b}^2_i}{\underline{a}^2_i} \Big)^{-1/2}\
  {}^t\Big(\frac{\underline{b}_e}{ \underline{a}_e},-\frac{\underline{b}_i}{ \underline{a}_i},0,0,0,0\Big) \, .
\]


\subsubsection{The case $ k \not = 0 $.}\label{subsubsec:knot0}
For $ k \not = 0 $, the two lines of $ \cG_k $ are independent vectors. Thus, we must have ${\rm dim\,}(\text{Ker}\, \cG_k)=12$. More precisely, we find that $\text{Ker}\, \cG_k = {\rm Vect}\{g_k^l\}_{l\in\{1,\ldots,12\}}$ with
\[ \begin{array}{ll}
  {}^tg_k^1 := & \displaystyle \! \! \! \Big(1+\frac{\underline{b}_e^2}{\underline{a}_e^2} \, |k|^{-2}\Big)^{-1/2} \ {}^t\Big(1,0,0,0,+ {\rm i} \,  \frac{\underline{b}_e}{\underline{a}_e} \, |k|^{-1}\,{}^t\mathbbm{e}_k^1,0\Big)\,, \medskip \\
  {}^tg_k^2 := & \displaystyle \! \! \!  \Big(1+\frac{\underline{b}_i^2}{\underline{a}_i^2} \, |k|^{-2}\Big)^{-1/2} \ {}^t\Big(0,1,0,0,- {\rm i} \,  \frac{\underline{b}_i}{\underline{a}_i} \, |k|^{-1}\,{}^t\mathbbm{e}_k^1,0\Big)\,,
\end{array} \]
and
\[ \begin{array}{ll}
{}^tg_k^3:={}^t(0,0,0,0,{}^t\mathbbm{e}_k^2,0)\,,\qquad &
  {}^tg_k^4:={}^t(0,0,0,0,{}^t\mathbbm{e}_k^3,0)\,, \smallskip\\
  {}^tg_k^5:={}^t(0,0,0,0,0,{}^t\mathbbm{e}_k^2)\,, \qquad &
  {}^tg_k^6:={}^t(0,0,0,0,0,{}^t\mathbbm{e}_k^3)\, ,
    \end{array}
\]
as well as
\[ \begin{array}{lll}
{}^tg_k^7:={}^t(0,0,{}^te^1,0,0,0)\,,\qquad & {}^tg_k^8:={}^t(0,0,{}^te^2,0,0,0)\,,\qquad &
  {}^tg_k^9:={}^t(0,0,{}^te^3,0,0,0)\,, \smallskip \\
  {}^tg_k^{10}:={}^t(0,0,0,{}^te^1,0,0)\,,\qquad &
  {}^tg_k^{11}:={}^t(0,0,0,{}^te^2,0,0)\,,\qquad &
  {}^tg_k^{12}:={}^t(0,0,0,{}^te^3,0,0)\, .
  \end{array} \]
We find that $ (\text{Ker} \, \cG_k)^\perp = \text{Vect} \ \{\mathfrak{g}_k^1 , \mathfrak{g}_k^2 \} $ where $ \mathfrak{g}^1_k $ and $ \mathfrak{g}^2_k $ are the following orthonormal vectors
\begin{equation}\label{mathfrakg}
\begin{array}{l}
  {}^t \mathfrak{g}_k^1 := \displaystyle \Big(\frac{\underline{b}^2_e}{\underline{a}^2_e} + \frac{\underline{b}^2_i}{\underline{a}^2_i} + |k|^2 \Big)^{-1/2}\
  {}^t\Big(\frac{\underline{b}_e}{ \underline{a}_e},-\frac{\underline{b}_i}{ \underline{a}_i},0,0,- {\rm i} \, |k| \, {}^t\mathbbm{e}^1_k,0\Big) \, , \smallskip \\
  {}^t \mathfrak{g}_k^2 := {}^t (0,0,0,0,0,{}^t \mathbbm{e}^1_k ) \, .
\end{array}
\end{equation}


\subsection{General formula for $ \cL_k $.}\label{subsec:genfor} Consider the linear map $ J : \mathbb \CC^{14} \rightarrow \mathbb \CC^{14} $ given by
\begin{equation}\label{defideJ}
J:= {\rm diag\,}(-1,-1,\Id_3, \Id_3, -\Id_3,\Id_3) \, .
\end{equation}

\begin{lemma}[Properties of $ J $]\label{propertiesofJ} The matrix $ J $ is such that
\begin{equation}
  \label{symeigen}
  J=J^*={}^t J \, , \qquad J^2 = \Id_{\RR^{14}} \, , \qquad
 \cL_k \, J = - J \, \cL_k \, , \qquad \forall k \in \ZZ^3 \,.
\end{equation}
\end{lemma}

\begin{proof} The first two identities are obvious. The third one can be checked directly from the definition \eqref{defideJ} and \eqref{cLkcU}.
\end{proof}

Thus, the (discrete) spectrum of $ {\rm i} \cL_k $ is a subset of the real line which is symmetric with respect to zero. Note that $ {}^t (0,0,0,0,0,{}^t e^1) \in \text{Ker} \, \cL_0 $, whereas $ {}^t (0,0,0,0,0,{}^t \mathbbm{e}^1_k) \in \text{Ker} \, \cL_k $ for all $ k \in \ZZ^3\setminus \{0\} $. It follows that $ 0 $ is always an eigenvalue of $ \cL_k $. We denote by $ \pm {\rm  i} \lambda_k^\jmath $ with $ \lambda_k^0 = 0 $ and $  \lambda_k^{\jmath} \in \RR_+^*  $ for $ \jmath \in \{ 1, \cdots, N_k \} $ the distinct eigenvalues of the matrix $ \cL_k $, which may be listed in increasing order according to
\begin{equation}\label{ordered}
- \lambda_k^{N_k} < \cdots < - \lambda_k^1 < \lambda_k^0 = 0 < \lambda_k^1 < \cdots < \lambda_k^{N_k} \, , \qquad 1 \leq N_k < 7 \, .
\end{equation}
Since $ J $ is invertible, the multiplicity of $ {\rm i} \lambda_k^{\jmath} $ is the same as the one of $ - {\rm i} \lambda_k^{\jmath} $. Define
\[ \mathscr{K}_k^{\jmath \pm} := \text{Ker} \, (\cL_k \mp {\rm i} \lambda^\jmath_k \, \Id ) \, , \qquad
M^\jmath_k := \text{dim} \, \mathscr{K}_k^{\jmath +} = \text{dim} \, \mathscr{K}_k^{\jmath -}  \, , \qquad \forall \, \jmath \in \{ 0,\cdots , N_k \} \, . \]
With these conventions, in particular, we have
\[ \mathscr{K}_k^0 := \text{Ker} \, \cL_k = \mathscr{K}_k^{0 +} = \mathscr{K}_k^{0 -} \, , \qquad M^0_k = \text{dim} \, \mathscr{K}_k^0 \, , \]
and by construction, we can assert that
\begin{equation}
  \label{sumMM}
 0 < M^\jmath_k \, , \qquad M^0_k + 2 \, (M^1_k + \cdots + M_k^{N_k}) = 14 \, .
\end{equation}
Let $ \{ r_k^{\jmath,l} \}_l $ with $ l \in \{ 1, \cdots ,M_k^\jmath \} $ be some  orthonormal basis of the eigenspace $ \mathscr{K}^{\jmath+}_k $.
We have
\[
\mathscr{K}_k^{\jmath -} = J  \mathscr{K}_k^{\jmath +} = \lbrace \, \cU \in \CC^{14} \, ; \, \cL_k \, \cU = - {\rm i} \lambda_k^\jmath \, \cU \, \rbrace = \text{Vect} \ \{ J  r_k^{\jmath,l} \}_{l = 1, \cdots, M_k^\jmath} \, .
\]
Observe that $ J $ preserves the Hermitian inner product
\[ \langle J\,\cU , J\,\cV \rangle = \langle \cU , \cV \rangle \, , \qquad \forall \, (\cU,\cV) \in \CC^{14} \times \CC^{14} \, . \]
Thus, the family $ \{ J  r_k^{\jmath,l} \}_l $ with $ l \in \{ 1, \cdots ,M_k^\jmath \} $ is an orthonormal basis of $ \mathscr{K}_k^{\jmath -} $. By construction
\begin{equation}\label{decompocU}
\cU = \sum_{\jmath=1}^{N_k} \sum_{l=1}^{M_k^\jmath} \, \big( {}^t (J \bar r^{\jmath,l}_k) \, \cU \big) \ J r^{\jmath,l}_k +
\sum_{l=1}^{M_k^0} \, \big({}^t \bar r^{0,l}_k  \, \cU \big) \ r^{0,l}_k +  \sum_{\jmath=1}^{N_k} \sum_{l=1}^{M_k^\jmath} \, \big({}^t \bar r^{\jmath,l}_k\,   \cU\big) \ r^{\jmath,l}_k \, ,
\end{equation}
as well as
\begin{equation}\label{decompocLkU}
\cL_k \, \cU = -  \sum_{\jmath=1}^{N_k} \sum_{l=1}^{M_k^\jmath} \, {\rm i} \, \lambda_k^\jmath \  \big( {}^t (J \bar r^{\jmath,l}_k)\,  \cU \bigr) \ J  r^{\jmath,l}_k +  \sum_{j=1}^{N_k} \sum_{l=1}^{M_k^\jmath} \, {\rm i}\, \lambda_k^\jmath \ \big({}^t \bar r^{\jmath,l}_k \, \cU\big) \ r^{\jmath,l}_k \, .
\end{equation}


\subsection{General formula for $ \PG_k \cL_k \PG_k $.}\label{subsec:genforPP} The introduction of $ \PG_k $ modifies the previous discussion. From \eqref{eqnGL0} and \eqref{decompocLkU}, we know that
\begin{equation} \label{sumalge}
\text{Im} \, \cL_k = \bigoplus_{\jmath = 1}^{N_k} \, ( \mathscr{K}^{\jmath -}_k \oplus \mathscr{K}^{\jmath +}_k ) \subset \text{Ker} \, \cG_k \, .
\end{equation}
Since the symmetric matrices $ {\rm i} \cL_k $ and $ \PG_k $ commute, they are simultaneously diagonalizable. In fact, from \eqref{sumalge}, we have
\[ {\cL_k}_{\mid \mathscr{K}^{\jmath \pm}_k} \equiv \pm {\rm i} \, \lambda^\jmath_k \ \Id_{\mid \mathscr{K}^{\jmath \pm}_k} \, , \quad {\PG_k}_{\mid \mathscr{K}^{\jmath \pm}_k} \equiv \Id_{\mid \cK^{\jmath \pm}_k} \quad \Longrightarrow \quad {\PG_k \cL_k \PG_k}_{\mid \mathscr{K}^{\jmath \pm}_k} \equiv \pm {\rm i} \, \lambda^\jmath_k \, \PG_k \, , \quad \forall \, \jmath \not = 0 \, . \]
For $ \jmath \not = 0 $, the complex number $ \pm {\rm i} \, \lambda^\jmath_k $ is therefore an eigenvalue of $ \PG_k \cL_k \PG_k $ with multiplicity $ M^\jmath_k $. For $ \jmath \not = 0 $, all eigenvectors $ r_k^{\jmath,l} $ fall inside $ \text{Ker} \, \cG_k $.

Now, the case $ \jmath = 0 $  should be considered separately because $ \mathscr{K}^0_k $ is not necessarily included inside $ \text{Ker} \, \cG_k $.
Since the codimension of $ \text{Ker} \, \cG_k $, i.e. the dimension of $( \text{Ker} \, \cG_k)^\perp$, is at most two, we can say that
\[
M^0_k-2 \leq \mathring{M}^0_k := \text{dim} \, \bigl( (\text{Ker} \, \cG_k) \cap \mathscr{K}^{0}_k \bigr)\leq M^0_k
\]
In other words, we can adjust the $ r_k^{0,l} $ in such a way that
\begin{equation}
  \label{decomprjktilde}
  r_k^{0,l} \in  \text{Ker} \, \cG_k \, , \qquad \forall \, l \in \{ 1,\cdots,\mathring{M}^0_k \} \, ,
\end{equation}
and, when $ \mathring{M}^0_k < M^0_k $, in such a way that
\begin{equation}
  \label{decomprjkg}
  r_k^{0,l} \in (\text{Ker} \, \cG_k)^\perp \, , \qquad \forall \, l \in \{ \mathring{M}^0_k + 1, \cdots , M^0_k \} \, .
\end{equation}
In the light of the foregoing, we can assert that
\begin{equation}\label{decompocUPP}
\begin{array}{rl}
\displaystyle e^{\pm \tau \PG_k \cL_k \PG_k} \, \PG_k \cU_k = \! \! \! & \displaystyle + \sum_{\jmath=1}^{N_k} \sum_{l=1}^{M_k^\jmath} \, e^{\mp {\rm i} \tau \lambda^\jmath_k} \, \big( {}^t (J \bar r^{\jmath,l}_k)  \, \cU_k \big) \ J r^{\jmath,l}_k +
\sum_{l=1}^{\mathring{M}_k^0} \, \big({}^t \bar r^{0,l}_k\,  \cU_k\big) \ r^{0,l}_k \\
& \displaystyle +  \sum_{\jmath=1}^{N_k} \sum_{l=1}^{M_k^\jmath} \, e^{\pm {\rm i} \tau \lambda^\jmath_k} \,  \big({}^t \bar r^{\jmath,l}_k \,  \cU_k\big) \ r^{\jmath,l}_k \, .
\end{array}
\end{equation}
The aim in what follows is to make explicit
\eqref{decompocUPP}. To this end, we have to compute $ N_k $, $  M^\jmath_k $ and $ \mathring{M}^0_k $. We have also to determine the eigenvalues $ {\rm i} \lambda_k^{\jmath} $ and orthonormal basis $ \{ r^{\jmath,l}_k \}_l $ of the eigenspaces $ \mathscr{K}^{\jmath +}_k $ adjusted (for $ \jmath = 0 $) as indicated in \eqref{decomprjktilde} and \eqref{decomprjkg}. In Subsection \ref{subsubsec:compze}, we deal with $ k = 0 $. In Subsection \ref{subsubsec:compzeknot0}, we discuss separately the situation $ k \not = 0 $.


\subsection{Some orthonormal basis made of  eigenvectors of $ \cL_0 $}\label{subsubsec:compze}
The case $ k=0 $ is somewhat specific as can be clearly inferred by looking at Subsection \ref{subsec:cGk}.


\subsubsection{Description of the eigenspace $ \mathscr{K}_0^0 = \text{\rm Ker} \, \cL_0 $}\label{subsubsec:K00}
The matrix $\cL_0$ is just
\begin{equation*}
  {\cL}_0 = \left( \begin{array}{cccccc}
   0 & 0 & 0 & 0 & 0 & 0 \\
   0 &  0 & 0 &  0 & 0 & 0 \\
     0 & 0 & 0 & 0 &  -\,\underline{b}_e\,{\rm I}_3 & 0 \\
   0 &   0 & 0 & 0  & \ \ \,\underline{b}_i\,{\rm I}_3 & 0 \\
    0 & 0 & \underline{b}_e\,{\rm I}_3 & -\,\underline{b}_i\,{\rm I}_3 & 0  & 0\\
    0 & 0 & 0 & 0 &  0 & 0 
  \end{array}  \right) .
\end{equation*}
Let $\cU $ be a solution of $ {\cL}_{0} \,\cU =0$.  This is equivalent to
\begin{equation*}
\underline{b}_e E=0\,, \qquad \underline{b}_i E=0\,, \qquad \underline{b}_e u_e - \underline{b}_i u_i=0\,.
\end{equation*}
From this, we obtain $u_e=\underline{b}_i/\underline{b}_e u_i$ and $E=0$, while $\varrho_e$, $\varrho_i$, $ u_i $ and $B$ are arbitrary. Hence, we have $ M^0_0 = 8 $. With $ \underline{b} $ as in \eqref{defunderlineb}, define
\begin{equation}\label{cestr010}
\begin{array}{ll}
  {}^t r_0^{0,1} := \! \!\! & (\underline{p}'_e +\underline{p}'_i)^{-1/2} \ ^t \Bigl(\sqrt{\underline{p}'_e}, \sqrt{\underline{p}'_i},0,0,0,0 \Bigr) = - {}^t J \bar r^{0,1}_0 = - {}^t J r^{0,1}_0 \, , \smallskip \\
  {}^tr_0^{0,2} := \! \!\! & \underline{b}^{-1} \
  {}^t\Big(0,0,\underline{b}_i \, {}^t e^1, \underline{b}_e \, {}^t e^1,0,0\Big) = {}^t J \bar r^{0,2}_0 \, = {}^t J r^{0,2}_0 \,, \smallskip  \\
  {}^tr_0^{0,3} := \! \!\! & \underline{b}^{-1} \
  {}^t\Big(0,0,\underline{b}_i \, {}^t e^2,\underline{b}_e \, {}^t e^2,0,0\Big) = {}^t J \bar r^{0,3}_0 \, = {}^t J r^{0,3}_0 \,, \smallskip  \\
  {}^tr_0^{0,4} := \! \!\! & \underline{b}^{-1} \
  {}^t\Big(0,0,\underline{b}_i {}^t e^3,\underline{b}_e\,  {}^t e^3,0,0\Big) =  {}^t J \bar r^{0,4}_0 \, = {}^t J  r^{0,4}_0 \,, \smallskip \\
  {}^tr_0^{0,5} := \! \!\! & {}^t(0,0,0,0,0,{}^t e^1) =  {}^t J \bar r^{0,5}_0 \, = {}^t J r^{0,5}_0 \,, \smallskip \\
  {}^tr_0^{0,6} := \! \!\! & {}^t(0,0,0,0,0,{}^t e^2) =  {}^t J \bar r^{0,6}_0 \, = {}^t J r^{0,6}_0 \,, \smallskip \\
  {}^tr_0^{0,7}:= \! \!\! & {}^t(0,0,0,0,0,{}^t e^3) =  {}^t J \bar r^{0,7}_0 \, = {}^t J r^{0,7}_0 \,.
  \end{array}
\end{equation}
The vectors $\{ r^{0,l}_0\}_{ l= 1,\ldots, 7} $ are all in $ \text{Ker} \, \cL_0 \cap \text{Ker} \, \cG_0 $, so that $ 7 \leq \mathring{M}^0_0 $. But $ \mathfrak{g}^1_0 \in {\rm Ker}\, \cL_0 \cap (\text{Ker} \, \cG_0)^\perp $, which implies that $ \mathring{M}^0_0 = 7 $. To recover \eqref{decomprjktilde} and \eqref{decomprjkg}, it suffices to take $ r^{0,8}_0 = \mathfrak{g}^1_0 $.


\subsubsection{Description of the eigenspaces $ \mathscr{K}_0^{\jmath +} = \text{\rm Ker} \, (\cL_0 - {\rm i} \lambda^\jmath_0 \, \Id ) $ for $ \jmath \not = 0 $}\label{subsubsec:k=0jmathnot0}
Since we have found that $M_0^0 =8$, we must have ${\rm dim} \, ({\rm Im}\, \cL_0)=14-8 =6$.  Let $\cU \not = 0 $ be an eigenvector of the matrix $\cL_0$ with corresponding non-zero eigenvalue ${\rm i} \, \lambda\neq 0$. The condition $\cL_0\, \cU = {\rm i} \, \lambda\, \cU$ is the same as
\begin{equation*}
{\rm i} \, \lambda \left(\begin{array}{c}
\varrho_e \\ \varrho_i \\ u_e \\ u_i \\ E \\ B
\end{array}\right)=
\left(\begin{array}{c}
0 \\ 0 \\
- \underline{b}_e E \\
\ \, \underline{b}_iE \\ \underline{b}_eu_e -\underline{b}_i u_i \\ 0
\end{array}\right)\,.
\end{equation*}
This means  that $ \varrho_e = \varrho_i = 0 $ and $ B = 0 $. Combining the third, fourth and fifth equations, we get that
\begin{equation}\label{comb*}
- \lambda^2E=\underline{b}_e {\rm i}\lambda u_e -\underline{b}_i {\rm i}\lambda u_i=-(\underline{b}_e^2 +\underline{b}_i^2)E \, , \qquad E \not = 0 \, .
\end{equation}
We cannot have $ E=0 $. Otherwise, we find $ u_e = u_i = 0 $ so that $ \cU = 0 $, which is a contradiction. Since $ E \not = 0 $, from \eqref{comb*}, we deduce that $ \lambda = \pm \lambda_0^1 $ with $ 0 < \lambda_0^1 = \underline{b} $. We find $ N_0 = 1 $ and from \eqref{sumMM} that $ M_0^1 = 3 $. We now look at $ \mathscr{K}^{1+}_0 $. Let $\cU={}^t(\varrho_e,\varrho_i,{}^tu_e,{}^tu_i,{}^tE,{}^tB)$ be an element of $ \mathscr{K}^{1+}_0 $, then the equation $\cL_0\, \cU-{\rm i}\lambda_0^1 \,\cU=0$ is equivalent to the following set of equations
\[
\left\{
\begin{array}{lll}
  -{\rm i} \underline{b}  \varrho_e = 0\,, \quad & -{\rm i} \underline{b}  \varrho_i=0\,, \quad & -{\rm i} \underline{b}  B=0 \,,\\
  -{\rm i} \underline{b}  u_e -\underline{b}_e E=0 \,, \quad & -{\rm i}\underline{b} u_i +\underline{b}_iE=0\,, & \quad
  \underline{b}_eu_e -\underline{b}_iu_i - {\rm i}\underline{b} E=0\,,
\end{array}
\right.
\]
where the sixth equation is automatically verified if the fourth and fifth equations are satisfied. The resolution of these equations gives
\begin{equation*}
  \varrho_e =\varrho_i= 0\,, \qquad B=0\,, \qquad \underline{b} \,  u_e={\rm i} \, \underline{b}_e \,  E\,,
  \qquad \underline{b} \,  u_i=- {\rm i} \,  \underline{b}_i \, \, E\, .
\end{equation*}
It suffices to select the following orthonormal basis of $ \mathscr{K}_0^{1+} = \PG_0 \mathscr{K}_0^{1+} $
\begin{equation}\label{cestlasuite}
\begin{array}{ll}
   {}^t r_0^{1,1} := & \! \! \! 1/(\sqrt{2} \, \underline{b} ) \ {}^t\big(0,0,{\rm i}\, \underline{b}_e \, {}^t e^1, -{\rm i}\, \underline{b}_i \, {}^te^1, \underline{b} \, {}^t e^1,0\big) = - {}^t J \bar r^{1,1}_0 \,, \smallskip \\
   {}^tr_0^{1,2} := & \! \! \! 1/(\sqrt{2} \, \underline{b} ) \ {}^t\big(0,0,{\rm i}\, \underline{b}_e \,  {}^te^2, -{\rm i}\, \underline{b}_i \,  {}^te^2,\underline{b} \, {}^t e^2,0\big) = - {}^t J \bar r^{1,2}_0 \,, \smallskip \\
   {}^tr_0^{1,3} := & \! \! \!  1/(\sqrt{2} \, \underline{b} ) \  {}^t\big(0,0,{\rm i} \, \underline{b}_e \, {}^te^3, -{\rm i} \,\underline{b}_i \,  {}^te^3,\underline{b} \, {}^t e^3,0\big) = - {}^t J \bar r^{1,3}_0 \,.
\end{array}
\end{equation}


\subsection{Some orthonormal basis made of  eigenvectors of $ \cL_k $ when $ k \not = 0 $}\label{subsubsec:compzeknot0}
The more direct way to compute the eigenvalues and the eigenvectors is to investigate the
equation $\cL_k \, \cU = {\rm i} \lambda\, \cU$ with $ \cU $ adequately  adjusted. In Paragraph \ref{subsec:desck}, we describe the content of $ \mathscr{K}_k^0 $.
In Paragraph \ref{para1}, we exhibit a first eigenvalue $ {\rm i} \lambda $ with $ \lambda $ positive, and we remark that it is of multiplicity two. In Paragraph \ref{para2}, we extract two other eigenvalues $ {\rm i} \lambda $ with $ \lambda $ positive, each of multiplicity one. We check that this list of eigenvalues is exhaustive. Then, in Paragraph \ref{para3}, we examine according to the values of $ \underline{a}_e $ and $ \underline{a}_i $ how these three non-zero eigenvalues remain distinct and well-ordered uniformly with respect to large values of $ \vert k \vert $. In Paragraph \ref{para4}, we compute the longitudinal eigenvectors. 
In Paragraph \ref{summarytable} (to which the reader is subsequently referred), we
summarize the spectral results that have been obtained.


\subsubsection{Description of $ \mathscr{K}_k^0 $ when $ k \not = 0 $} \label{subsec:desck} Let $\cU={}^t(\varrho_e,\varrho_i,{}^tu_e,{}^tu_i,{}^tE,{}^tB)$ be a solution of $ {\cL}_{k} \,\cU =0$.  This  is equivalent to the following set of equations
\[\begin{array}{lll}
\mathbbm{e}^1\cdot u_e=0\,, \quad & \mathbbm{e}^1\cdot u_i=0\,, \quad & \mathbbm{e}^1\times E =0\,,\\
-{\rm i}\underline{a}_e\varrho_e |k|\mathbbm{e}^1-  \underline{b}_e E=0\,, \quad &
  -{\rm i}\underline{a}_i\varrho_i|k|\mathbbm{e}^1+  \underline{b}_i E=0\,, \quad
 & (\underline{b}_eu_e-\underline{b}_iu_i)+{\rm i}|k|\mathbbm{e}^1\times B=0\,.
\end{array}
\]
The first two equations are equivalent to set $ u_e=\alpha_e \, \mathbbm{e}^2 + \beta_e \, \mathbbm{e}^3 $ and $ u_i=\alpha_i \, \mathbbm{e}^2 + \beta_i \,  \mathbbm{e}^3 $ where $\alpha_s $ and $\beta_s$, with $s\in\{e,i\}$, are arbitrary complex numbers (recall that $ \mathbbm{e}^j \equiv \mathbbm{e}^j_k $). The fourth and the fifth equations can be combined to obtain
\begin{equation*}
  \varrho_e=- \frac{\underline{b}_e \, \underline{a}_i}{\underline{b}_i \,  \underline{a}_e} \varrho_i\,,\qquad
  E= {\rm i} \, \frac{|k|}{2} \, \bigg(\frac{\underline{a}_i}{\underline{b}_i}\varrho_i- \frac{\underline{a}_e}{\underline{b}_e}\varrho_e\bigg) \, \mathbbm{e}^1= {\rm i} |k| \frac{\underline{a}_i}{\underline{b}_i} \varrho_i \,  \mathbbm{e}^1\,,
\end{equation*}
while the third equation $\mathbbm{e}^1\times E=0$ is trivialy satisfied. Substituting the expressions of $u_e$ and $u_i$ in the sixth equation we extract
\begin{equation}
  \label{eqn_e1xR}
  -  \alpha_e\underline{b}_e \mathbbm{e}^3 +\beta_e\underline{b}_e  \mathbbm{e}^2  +\alpha_i\underline{b}_i  \mathbbm{e}^3 -\beta_i\underline{b}_i  \mathbbm{e}^2 + {\rm i}|k| \, B_\perp =0\,.
\end{equation}
From \eqref{eqn_e1xR} we deduce that there exists a real number $\gamma$ such that
\begin{equation*}
 B= -  {\rm i}\alpha_e\underline{b}_e |k|^{-1}\mathbbm{e}^3 +{\rm i}\beta_e\underline{b}_e |k|^{-1} \mathbbm{e}^2  + {\rm i}\alpha_i\underline{b}_i  |k|^{-1}\mathbbm{e}^3 -{\rm i}\beta_i\underline{b}_i |k|^{-1} \mathbbm{e}^2 +{\rm i} \gamma |k|^{-1}\mathbbm{e}^1\,.
\end{equation*}
There are six degrees of freedom corresponding to the choice of $ \varrho_i $, $ \alpha_e $, $ \alpha_i $, $ \beta_e $, $ \beta_i $ and $ \gamma $. This means that $ M_k^0=6 $. Define
\begin{equation}\label{cK0knot0}
\quad \begin{array}{rl}
  {}^t r_k^{0,1} := \! \! \! & (1/\underline{b})
  \ {}^t\big(0,0, \underline{b}_i \, {}^t \mathbbm{e}^2, \underline{b}_e \, {}^t \mathbbm{e}^2,0, 0\big) = - {}^t J \bar r^{0,1}_{-k} \, = {}^t J  r^{0,1}_{k} = -\bar r^{0,1}_{-k} \, , \medskip \\
  {}^t r_k^{0,2} := \! \! \! &  (1/\underline{b})
  \ {}^t\big(0,0, \underline{b}_i \, {}^t \mathbbm{e}^3, \underline{b}_e \, {}^t \mathbbm{e}^3,0, 0\big) = + {}^t J \bar r^{0,2}_{-k} \, = {}^t J  r^{0,2}_{k} = +\bar r^{0,2}_{-k} \,, \medskip\\
  {}^t r_k^{0,3} := \! \! \! &  (1/\underline{b}) \, (\underline{b}^2  + |k|^{2})^{-1/2}  \ {}^t\big(0,0,|k| \, \underline{b}_e \, {}^t \mathbbm{e}^2,-|k| \, \underline{b}_i \, {}^t \mathbbm{e}^2, 0, -{\rm i }\underline{b}^2 \,{}^t \mathbbm{e}^3\big) \smallskip \\
  = \! \! \! & - {}^t J \bar r^{0,3}_{-k} \, = {}^t J r^{0,3}_{k} = -\bar r^{0,3}_{-k} \,, \medskip\\
  {}^t r_k^{0,4} := \! \! \! &  (1/\underline{b}) \, (\underline{b}^2 + |k|^{2})^{-1/2}  \ {}^t\big(0,0,|k| \, \underline{b}_e \, {}^t \mathbbm{e}^3,-|k| \, \underline{b}_i \, {}^t \mathbbm{e}^3, 0, + {\rm i }\underline{b}^2 \,{}^t \mathbbm{e}^2 \big) \smallskip \\
  = \! \! \! & + {}^t J \bar r^{0,4}_{-k} \, = {}^t J r^{0,4}_{k} =+\bar r^{0,4}_{-k} \,.
\end{array}
\end{equation}
Retain that
\begin{equation}\label{cK0knot0bis}
 {}^t J r^{0,l}_{k} = r^{0,l}_{k} \, , \qquad
 {}^t J \bar r^{0,l}_{k} = \bar r^{0,l}_{k} \, , \qquad \forall l \in \{ 1,2,3,4\} \, , \qquad \forall k \in \ZZ^3 \setminus \{ 0 \} \, .
\end{equation}
We can check that $ \{ r^{0,l}_k \}_{l = 1,\cdots,4}$ is an orthonormal basis of $ {\rm Ker\,} \cL_k \cap {\rm Ker\,} \cG_k $. Furthermore, coming back to \eqref{mathfrakg}, we find that $ r^{0,5}_k := \mathfrak{g}^1_k $ and $ r^{0,6}_k := \mathfrak{g}^2_k $ are in $ {\rm Ker\,} \cL_k \cap ({\rm Ker\,} \cG_k)^\perp $, so that $ M_k^0 = 6 $ and $ \mathring{M}_k^0 = 4 $.


\subsubsection{Non-zero eigenvalues corresponding to eigenvectors pointing in the direction orthogonal to $ \mathbbm{e}^1$}\label{para1} In this paragraph, we consider {\it transverse waves} Recall the notation \eqref{deroulu}. We can test
the equation $\cL_k \, \cU = {\rm i} \lambda\, \cU$ with
\[ \cU={}^t(0,0, \mathbbm{u}_{e3} \, {}^t\mathbbm{e}^3, \mathbbm{u}_{i3} \, {}^t\mathbbm{e}^3, \mathbbm{E}_3 \, {}^t\mathbbm{e}^3,\mathbbm{B}_2 \, {}^t\mathbbm{e}^2) \, , \qquad (\mathbbm{u}_{e3},\mathbbm{u}_{i3},\mathbbm{E}_3,\mathbbm{B}_2) \in \CC^4
\, .
\]
This is equivalent to solve the following set of equations
\begin{equation}\label{eigene1}
{\rm i} \, \lambda \left(\begin{array}{c}
0 \\
0 \\
\mathbbm{u}_{e3} \, \mathbbm{e}^3 \\
\mathbbm{u}_{i3} \, \mathbbm{e}^3 \\
\mathbbm{E}_3 \, \mathbbm{e}^3 \\
\mathbbm{B}_2 \, \mathbbm{e}^2
\end{array}\right)=
\left(\begin{array}{c}
  0 \\
  0 \\
  - \mathbbm{E}_3 \, \underline{b}_e \, \mathbbm{e}^3 \\
  + \mathbbm{E}_3 \, \underline{b}_i \,  \mathbbm{e}^3 \\
  (\mathbbm{u}_{e3} \, \underline{b}_e -\mathbbm{u}_{i3} \, \underline{b}_i + {\rm i} \, \mathbbm{B}_2 \, |k|) \, \mathbbm{e}^3 \\
  {\rm i} \, \mathbbm{E}_3 \, |k| \, \mathbbm{e}^2
\end{array}\right)\,.
\end{equation}
We multiply the fifth equation by $ {\rm i} \, \lambda $ and we exploit the other relations to obtain
\begin{equation*}
- \lambda^2 \, \mathbbm{E}_3= \underline{b}_e \, ({\rm i} \, \lambda \, \mathbbm{u}_{e3}) -  \underline{b}_i ({\rm i} \, \lambda \, \mathbbm{u}_{i3}) - |k| \, (\lambda \, \mathbbm{B}_2) =-(\underline{b}^2 + |k|^2) \, \mathbbm{E}_3\, .
\end{equation*}
There is no non-zero solution $ \cU \not = 0 $ to
\eqref{eigene1} with $ \mathbbm{E}_3= 0 $. Thus, we must have (for some $ \jmath $)
\begin{equation}\label{lambdafisrt}
  \lambda_k^\jmath = \lambda_{\perp k} := \sqrt{\underline{b}^2+|k|^2}\,.
\end{equation}
Coming back to the equation $\cL_k \, \cU = {\rm i} \, \lambda \, \cU $,  we can express all components of $ \cU $ in terms of the coefficient $ \mathbbm{u}_{e3} $. This leads to the eigenvector associated to $\lambda_{\perp k} $, which is
\begin{equation}\label{cK0kperp1}
r^{\jmath,1}_k = r^{1}_{\perp k} := \frac{1}{\sqrt{2(\underline{b}^2+|k|^2)}}   \left(\begin{array}{c}
  0 \\ 0 \\ \ \ \, \underline{b}_e \,  \mathbbm{e}^3 \\ -\underline{b}_i \,  \mathbbm{e}^3\\
  - {\rm i}\sqrt{\underline{b}^2+|k|^2} \ \mathbbm{e}^3 \\ -{\rm i} \, |k| \, \mathbbm{e}^2\,
\end{array}\right) = + J \bar r^{\jmath,1}_{-k} \,.
\end{equation}
We can do the same with
\[ \cU={}^t(0,0, \mathbbm{u}_{e2} \, {}^t\mathbbm{e}^2, \mathbbm{u}_{i2} \, {}^t\mathbbm{e}^2, \mathbbm{E}_2 \, {}^t\mathbbm{e}^2,\mathbbm{B}_3 \, {}^t\mathbbm{e}^3) \, , \qquad (\mathbbm{u}_{e2},\mathbbm{u}_{i2},\mathbbm{E}_2,\mathbbm{B}_3) \in \CC^4
\, .
\]
Again, the equation $\cL_k \, \cU = {\rm i} \, \lambda\, \cU$ gives rise to \eqref{lambdafisrt}. But this time the eigenvector associated to $\lambda_{\perp k} $ is given by
\begin{equation}\label{cK0kperp2}
r^{\jmath,2}_k = r^{2}_{\perp k} := \frac{1}{\sqrt{2(\underline{b}^2+|k|^2)}}   \left(\begin{array}{c}
  0 \\ 0 \\ \ \ \, \underline{b}_e \,  \mathbbm{e}^2 \\ -\underline{b}_i \,  \mathbbm{e}^2 \\
  - {\rm i} \, \sqrt{\underline{b}^2+|k|^2} \ \mathbbm{e}^2 \\
  {\rm i} \, |k|\, \mathbbm{e}^3\,
\end{array}\right) = - J \bar r^{\jmath,2}_{-k} \,.
\end{equation}
It is clear that the two unit vectors $ r^{1}_{\perp k} $ and $ r^{2}_{\perp k}  $ are orthogonal, so that $ M^\jmath_k \geq 2 $.


\subsubsection{Non-zero eigenvalues corresponding to
eigenvectors which are pointing in the direction parallel to $ \mathbbm{e}^1$}\label{para2} We consider here {\it longitudinal waves}. More precisely, we look for eigenvectors $ \cU $ having the form
\[ \cU={}^t(\varrho_e,\varrho_i, \mathbbm{u}_{e1} \, {}^t\mathbbm{e}^1, \mathbbm{u}_{i1} \, {}^t\mathbbm{e}^1, \mathbbm{E}_1 \, {}^t\mathbbm{e}^1, 0) \, , \qquad (\varrho_e,\varrho_i, \mathbbm{u}_{e1} , \mathbbm{u}_{i1} , \mathbbm{E}_1 ) \in \CC^5 \, . \]
This means to consider the condition
\begin{equation}\label{debuteql}
{\rm i} \, \lambda \left(\begin{array}{c}
\varrho_e \\
\varrho_i \\
\mathbbm{u}_{e1} \, \mathbbm{e}^1 \\
\mathbbm{u}_{i1} \, \mathbbm{e}^1 \\
\mathbbm{E}_1 \, \mathbbm{e}^1 \\
0
\end{array}\right)=
\left(\begin{array}{c}
  -{\rm i} \, \mathbbm{u}_{e1} \, \underline{a}_e \, |k|\\ -{\rm i} \, \mathbbm{u}_{i1} \, \underline{a}_i \, |k| \\  (-{\rm i} \, \varrho_e \, \underline{a}_e |k|  - \mathbbm{E}_1 \, \underline{b}_e ) \mathbbm{e}^1 \\
  ( -{\rm i} \, \varrho_i \, \underline{a}_i \, |k|   + \mathbbm{E}_1 \,  \underline{b}_i) \, \mathbbm{e}^1  \\
  (\mathbbm{u}_{e1} \, \underline{b}_e -\mathbbm{u}_{i1} \, \underline{b}_i) \, \mathbbm{e}^1 \\
  0
\end{array}\right)\,.
\end{equation}
We multiply the third and fourth  equations by $ - {\rm i} \, \lambda $ and we exploit the other relations to obtain
\begin{align*}
& \lambda^2 \, \mathbbm{u}_{e1} = -(\lambda \, \varrho_e) \, \underline{a}_e \, |k| + {\rm i} \, (\lambda \, \mathbbm{E}_1) \, \underline{b}_e = ( \underline{b}_e^2 + \underline{a}_e^2 \, |k|^2) \, \mathbbm{u}_{e1} -  \underline{b}_e \, \underline{b}_i \,  \mathbbm{u}_{i1} \,, \\
&\lambda^2 \, \mathbbm{u}_{i1} = -(\lambda \, \varrho_i) \, \underline{a}_i \, |k| - {\rm i} \, (\lambda \, \mathbbm{E}_1) \, \underline{b}_i = ( \underline{b}_i^2 + \underline{a}_i^2 \, |k|^2) \, \mathbbm{u}_{i1} -  \underline{b}_e \, \underline{b}_i \,  \mathbbm{u}_{e1} \, .
\end{align*}
There is a solution with $ (\mathbbm{u}_{e1},\mathbbm{u}_{i1})\not = (0,0)$ if and only if
\begin{equation*}
  (\lambda^2 - \underline{c}_{ek}^2) \, (\lambda^2 - \underline{c}_{ik}^2) - \underline{b}_e^2 \, \underline{b}_i^2 = 0 \,, \qquad \underline{c}_{sk} := (\underline{b}_s^2 + \underline{a}_s^2 \, |k|^2)^{1/2} \,, \qquad s\in \{e,i\} \, .
\end{equation*}
With $ \Lambda := \lambda^2 $, this is the same as imposing
\begin{equation*}
  P(\Lambda) = 0 \, , \qquad P(\Lambda) := \Lambda^2 - (\underline{c}_{ek}^2+ \underline{c}_{ik}^2) \,  \Lambda + \underline{c}_{ek}^2 \, \underline{c}_{ik}^2 -\underline{b}_e^2 \, \underline{b}_i^2 \, .
\end{equation*}
Since $ \underline{c}_{ek} \, \underline{c}_{ik} > \underline{b}_{e} \, \underline{b}_{i} $, the two roots $ \Lambda^l $ and $ \Lambda^r $ are positive and distinct, say with $ \Lambda^l < \Lambda^r $ (where $ l $ is for left and $ r $ is for right). By this way, we can identify two  eigenvalues $ {\rm i} \lambda_{\mypar k}^l $ and $ {\rm i} \lambda_{\mypar k}^r $ satisfying $ 0 < \lambda_{\mypar k}^l < \lambda_{\mypar k}^r $, and given by
\begin{equation} \label{biglambdas}
 \begin{array}{l}
\displaystyle 0 < \Lambda^l := \frac{1}{2} \, \Bigl(\underline{c}_{ek}^2+\underline{c}_{ik}^2 - \sqrt{ (\underline{c}_{ek}^2-\underline{c}_{ik}^2)^2 + 4 \, \underline{b}_e^2 \, \underline{b}_i^2 } \Bigr) \, , \qquad \lambda_{\mypar k}^l:=\sqrt{\Lambda^l} \, , \medskip \\
\displaystyle 0 < \Lambda^r := \frac{1}{2} \, \Bigl(\underline{c}_{ek}^2+ \underline{c}_{ik}^2+ \sqrt{ (\underline{c}_{ek}^2-c_{ik}^2)^2 + 4 \, \underline{b}_e^2 \, \underline{b}_i^2 } \Bigr) \, , \qquad \lambda_{\mypar k}^r:= \sqrt{\Lambda^r}\, .
 \end{array}
\end{equation}


\subsubsection{Separation of the eigenvalues}\label{para3} Before proceeding further, the (absolute value of the) three eigenvalues must be sorted in ascending order.
This can be done for large values of $ \vert k \vert $ uniformly with respect to $ k $ as indicated below.

\begin{lemma}[Classification of eigenvalues for large values of $ \vert k \vert $]\label{Classification} There exists a positive lower bound $ m $ such that, for all $ \vert k \vert \geq m $, we have:
\begin{itemize}
 \item [i)] When $ \underline{a}_e \leq 1 $ and $ \underline{a}_i \leq 1 $ together with $ (\underline{a}_e,\underline{a}_i) \not = (1,1) $:
 \begin{equation}\label{alternat1}
  \lambda^1_k = \lambda_{\mypar k}^l < \lambda^2_k = \lambda_{\mypar k}^r < \lambda^3_k = \lambda_{\perp k} \, , \qquad (M^1_k ,M^2_k,M^3_k) = (1,1,2) \, .
 \end{equation}
 \item [ii)] When $ \underline{a}_e > 1 $ and  $ \underline{a}_i \leq 1 $, or when $ \underline{a}_e \leq 1 $ and $ \underline{a}_i > 1 $:
 \begin{equation}\label{alternat2}
  \lambda^1_k = \lambda_{\mypar k}^l < \lambda^2_k = \lambda_{\perp k}
    < \lambda^3_k = \lambda_{\mypar k}^r
  \, , \qquad (M^1_k ,M^2_k,M^3_k) = (1,2,1)\, .
 \end{equation}
 \item [iii)] When $ \underline{a}_e > 1 $ and $  \underline{a}_i > 1 $:
 \begin{equation}\label{alternat3}
  \lambda^1_k = \lambda_{\perp k}
  < \lambda^2_k =  \lambda_{\mypar k}^l < \lambda^3_k = \lambda_{\mypar k}^r \, , \qquad (M^1_k ,M^2_k,M^3_k) = (2,1,1)\, .
 \end{equation}
 \item [iv)] When $ \underline{a}_e = \underline{a}_i = 1 $, there remains only two eigenvalues ($ \lambda_{\mypar k}^r $ coincides with $\lambda_{\perp k} $)
 \begin{equation}\label{alternat4}
  \lambda^1_k = \lambda_{\mypar k}^l =|k|
  < \lambda^2_k = \lambda_{\mypar k}^r =\lambda_{\perp k} \, , \qquad (M^1_k ,M^2_k) = (1,3)\, .
 \end{equation}
\end{itemize}
\end{lemma}

In accordance with the above orderings,  when $ \vert k \vert \rightarrow +\infty $, we find that
\begin{equation}\label{equi+}
\lambda_{\mypar k}^l \underset{+\infty}{\sim} \min \, (\underline{a}_e ;\underline{a}_i) \ \vert k \vert \, , \qquad
\lambda_{\mypar k}^r \underset{+\infty}{\sim} \max \, (\underline{a}_e ;\underline{a}_i) \ \vert k \vert \, , \qquad \lambda_{\perp k} \underset{+\infty}{\sim} \vert k \vert \, .
\end{equation}
Resonances are more pronounced as an infinite number of frequencies $ k $ are involved. It is always the case here since $ \lambda^\jmath_k - \lambda^\jmath_k = \lambda^0_k = 0 $ for all $ k \in \ZZ^3 $. The interest of Lemma \ref{Classification} is to allow a uniform ordering of the waves when $ \vert k \vert $ is large enough. There is no crossing of eigenvalues as soon as $ \vert k \vert \geq m $, while this can happen for bounded values of $ \vert k \vert $.

\begin{proof} The last case iv), when $ \underline{a}_e = \underline{a}_i = 1 $, must be treated separately, with direct explicit computations. We focus below on the (three) other situations.\\
With $ \lambda_k^\jmath $ as in \eqref{lambdafisrt}, it suffices to compare $ (\lambda^\jmath_k)^2 $ with $ \Lambda_l $ and $ \Lambda_r $. To this end, compute
\[ \begin{array}{l}
    P\bigl((\lambda^\jmath_k)^2 \bigr) = |k|^2 \ \bigl \lbrack (1-\underline{a}^2_e) \, (1-\underline{a}^2_i) \, |k|^2 + (1-\underline{a}^2_e) \, \underline{b}^2_e + (1-\underline{a}^2_i) \, \underline{b}^2_i \bigr \rbrack \, , \medskip \\
    P'\bigl((\lambda^\jmath_k)^2 \bigr) = \underline{b}^2 + ( 2 - \underline{a}^2_e - \underline{a}^2_i ) \ \vert k \vert^2 \, .
   \end{array}
\]
The condition $ P\bigl((\lambda^\jmath_k)^2 \bigr) < 0 $ means that $ \lambda^\jmath_k $ lies between $ \lambda_{\mypar k}^l :=\sqrt{\Lambda^l} $ and $ \lambda_{\mypar k}^r :=\sqrt{\Lambda^r} $. This is sure to happen when $ (1-\underline{a}_e) (1-\underline{a}_i) < 0 $ and for all $|k|>m$, and a value of $m$ large enough. This leads to the case ii) and \eqref{alternat2}. When $ P\bigl((\lambda^\jmath_k)^2 \bigr) > 0 $, a criterion to decide whether $ \lambda^\jmath_k $ is on the left of $ \lambda_{\mypar k}^l $ or on the right of $ \lambda_{\mypar k}^r $ is the negative or positive sign of $ P'\bigl((\lambda^\jmath_k)^2 \bigr) $. This is consistent with \eqref{alternat1} and \eqref{alternat3}.\\
Since $ M^0_k = 6 $, from \eqref{sumMM}, we can  deduce that $ M^1_k + \cdots + M_k^{N_k} = 4 $. But we have seen that the multiplicity of the square root of $ \underline{b}^2+|k|^2 $ is greater than $ 2 $, whereas $ M^\jmath_k \geq 1 $. Thus, we must have $ N_k \leq 3 $. In the presence of three distinct eigenvalues such that $ \lambda^1_k  < \lambda^2_k < \lambda^3_k $, the multiplicity of $ {\rm i} \lambda_{\mypar k}^l $ and $ {\rm i}  \lambda_{\mypar k}^r $ is exactly $ 1 $ with the repartition  \eqref{alternat1},  \eqref{alternat2} and  \eqref{alternat3}.
\end{proof}

On the opposite side, it is interesting to investigate what happens for small values of $ \vert k \vert $. For $ \vert k \vert \sim 0  $, we find that
\[
\Lambda^l = \frac{\underline{b}^2_i \, \underline{a}^2_e + \underline{b}^2_e \, \underline{a}^2_i}{\underline{b}^2_e +  \underline{b}^2_i} \ \vert k \vert^2 + O ( \vert k \vert^4 ) < \Lambda^r = \underline{b}^2 + \frac{\underline{b}^2_e \, \underline{a}^2_e + \underline{b}^2_i \, \underline{a}^2_i}{\underline{b}^2_e +  \underline{b}^2_i} \ \vert k \vert^2 + O ( \vert k \vert^4 ) \, .
\]
Thus, the extension of the preceding formulas to the limiting case $ k = 0 $ gives rise to
\[
\lim_{\vert k \vert \rightarrow 0} \, \lambda_{\mypar k}^l = \lim_{\vert k \vert \rightarrow 0} \, \sqrt {\Lambda^l} = 0 \, , \quad \lim_{\vert k \vert \rightarrow 0} \, \lambda_{\mypar k}^r =\lim_{\vert k \vert \rightarrow 0} \, \sqrt {\Lambda^r}
= \underline{b} \, , \quad  \lim_{\vert k \vert \rightarrow 0} \,  \lambda_{\perp k} = \lim_{\vert k \vert \rightarrow 0} \, \sqrt{\underline{b}^2+|k|^2} = \underline{b} \, .
\]
Then $\lambda_{ k}^1=\lambda_{\mypar k}^l$, and $(\lambda_{ k}^2,\lambda_{ k}^3)=(\lambda_{\perp k}, \lambda_{\mypar k}^r)$ or $(\lambda_{ k}^2,\lambda_{ k}^3)= (\lambda_{\mypar k}^r, \lambda_{\perp k})$. Moreover, this means that the eigenvalues $ 0 $ and $ {\rm i} \underline{b} $ found for $ k=0 $ must be associated after perturbation (near $ \vert k \vert = 0 $) respectively with the couples  $ (0,{\rm i} \lambda_{\mypar k}^l ) $ $ \bigl($or $(0,{\rm i} \lambda_k^1)\, \bigr) $ and $ ({\rm i} \lambda_{\mypar k}^r, {\rm i} \lambda_{\perp k}) $ $ \bigl($or $({\rm i} \lambda_k^2,{\rm i} \lambda_k^3)$, or else  simply $ {\rm i} \lambda_k^2= {\rm i} \lambda_{\perp k} $ when $ {\rm i} \lambda_{\mypar k}^r = {\rm i} \lambda_{\perp k} \, \bigr) $.


\subsubsection{The two longitudinal eigenvectors}\label{para4}
To complete the discussion, we have to make explicit the eigenvectors $ r_{\mypar k}^l $ and $ r_{\mypar k}^r $ associated respectively with $ \lambda_{\mypar k}^l $ and $ \lambda_{\mypar k}^r $. To this end, we come back to \eqref{debuteql}. The relation \eqref{debuteql} can be used to  express all components of  $ \cU $ in terms of the coefficient $ \mathbbm{u}_{e1} $ only. Indeed, we have
\begin{align*}
  & \varrho_e =- \frac{\underline{a}_e \, |k|}{\lambda} \, \mathbbm{u}_{e1} \,,\qquad \quad \ \varrho_i = -\frac{\underline{a}_i \, |k|}{\lambda} \, \mathbbm{u}_{i1} =  \frac{\underline{a}_i \, |k|}{\lambda} \,  \frac{\lambda^2 - \underline{c}_{ek}^2}{\underline{b}_e \, \underline{b}_i} \, \mathbbm{u}_{e1} \,,\\
  & \mathbbm{u}_{i1} =- \frac{\lambda^2 - \underline{c}_{ek}^2}{\underline{b}_e \, \underline{b}_i} \,  \mathbbm{u}_{e1} \,, \quad \quad
  \mathbbm{E}_1 = - {\rm i} \, \frac{\underline{b}_e}{\lambda} \, \mathbbm{u}_{e1} + {\rm i} \, \frac{\underline{b}_i}{\lambda} \, \mathbbm{u}_{i1} = - {\rm i} \, \frac{(\underline{b}_e^2 + \lambda^2 - \underline{c}_{ek}^2)}{\lambda \,  \underline{b}_e} \, \mathbbm{u}_{e1} \, .
\end{align*}
With $ \mathbbm{u}_{e1} = \lambda \, \underline{b}_e \, \underline{b}_i $, this becomes
\[ \varrho_e =- \underline{a}_e \, \underline{b}_e \, \underline{b}_i \, |k| \,,\quad \ \varrho_i = \underline{a}_i \, (\lambda^2 - \underline{c}_{ek}^2) \, |k| \, , \quad \ \mathbbm{u}_{i1} = (\underline{c}_{ek}^2 -\lambda^2) \, \lambda \,, \quad \ \mathbbm{E}_1 = - {\rm i} \, \underline{b}_i \, (\underline{b}_e^2 + \lambda^2 - \underline{c}_{ek}^2) \, . \]
We introduce the notation
\begin{align*}
 & \underline{d}^l_{k}:= (\lambda_{\mypar k}^l)^2 - \underline{c}_{ek}^2 = \Lambda^l- \underline{c}_{ek}^2 = \frac 12 \,  (\underline{c}_{ik}^2- \underline{c}_{ek}^2) - \frac 12 \, \sqrt{ (\underline{c}_{ek}^2- \underline{c}_{ik}^2)^2 + 4 \, \underline{b}_e^2 \, \underline{b}_i^2 }\,, \\
 & \underline{d}_{k}^r:= (\lambda_{\mypar k}^r)^2 - \underline{c}_{ek}^2 = \Lambda^r- \underline{c}_{ek}^2 = \frac 12 \,  (\underline{c}_{ik}^2- \underline{c}_{ek}^2) + \frac 12 \, \sqrt{ (\underline{c}_{ek}^2- \underline{c}_{ik}^2)^2 + 4 \, \underline{b}_e^2 \, \underline{b}_i^2 }\, .
\end{align*}  
With these conventions, we have $ \underline{d}^l_{k} < \underline{d}^r_{k}$, and we obtain
\begin{equation}\label{rparallellk}
\begin{array}{l}
\displaystyle r_{\mypar k}^l := c^l_k \left(\begin{array}{c}
    - \underline{a}_e \, \underline{b}_e \, \underline{b}_i \, |k| \smallskip \\
    + \underline{a}_i \, \underline{d}_{k}^l \, |k| \smallskip\\
   + \underline{b}_e \, \underline{b}_i \, \lambda_{\mypar k}^l \, \mathbbm{e}^1 \smallskip\\
- \underline{d}_{k}^l \, \lambda_{\mypar k}^l \,  \mathbbm{e}^1 \smallskip\\
- {\rm i} \, \underline{b}_i \, (\underline{b}_e^2 + \underline{d}_{k}^l) \,  \mathbbm{e}^1 \smallskip\\
  0 \end{array} \right) = - J \bar r_{\mypar -k}^l \, , \smallskip \\
\displaystyle r_{\mypar k}^r := c^r_k \left(\begin{array}{c}
    - \underline{a}_e \, \underline{b}_e \, \underline{b}_i \, |k| \smallskip \\
    + \underline{a}_i \, \underline{d}_{k}^r \, |k| \smallskip\\
   + \underline{b}_e \, \underline{b}_i \, \lambda_{\mypar k}^r \, \mathbbm{e}^1 \smallskip\\
- \underline{d}_{k}^r \, \lambda_{\mypar k}^r \,  \mathbbm{e}^1 \smallskip\\
- {\rm i} \, \underline{b}_i \, (\underline{b}_e^2 + \underline{d}_{k}^r) \,  \mathbbm{e}^1 \smallskip\\
  0 \end{array} \right) = - J \bar r_{\mypar -k}^r \, , 
  \end{array} 
\end{equation}
where $ c^l_k $ and $ c^r_k $ are normalizing constants. By exploiting the relations
\[ \underline{d}_{k}^l \, \underline{d}_{k}^r = - \underline{b}_e^2 \, \underline{b}_i^2 \, , \qquad \underline{d}_{k}^l + \underline{d}_{k}^r = \underline{b}_i^2 - \underline{b}_e^2 + (\underline{a}_i^2 - \underline{a}_e^2) \, |k|^2 \, , \]
we can directly check that $  r_{\mypar k}^l $ and $ r_{\mypar k}^r $ are orthogonal.


\subsubsection{Summary table}\label{summarytable}
When exploiting the formulas \eqref{decompocU} and \eqref{decompocLkU}, the following catalogue, where $ \lambda_{\mypar k}^l $ and $ \lambda_{\mypar k}^r $ can be computed through \eqref{biglambdas}, may be useful.

\begin{figure}[h]
{\renewcommand{\arraystretch}{1.5}
\centerline{\begin{tabular}{|c|c|c|c|c|c|c|c|}
  \hline
  \ & $ M^0_k $ &
   $ \mathring{M}^0_k $ & $ N_k $ & non-zero eigenvalues & $ M^1_k $ & $ M^2_k $ & $ M^3_k $ \\
  \hline
  $ k=0 $ & 8 & 7 & 1 & $ \lambda_0^1 = \underline{b} $ & 3 & $ \times $ & $ \times $ \\
  \hline
$ \displaystyle \begin{array}{c}
k \not = 0 \ \text{and} \\
(\underline{a}_e , \underline{a}_i) = (1,1) \smallskip
\end{array} $
   & 6 & 4 & 2 & $ \lambda_{\mypar k}^l =|k|  <  \lambda_{\mypar k}^r =\lambda_{\perp k} $ &  1 & 3 & $ \times $ \\
  \hline
  $ \displaystyle \begin{array}{c}
k \not = 0 \ \text{and} \\
(\underline{a}_e , \underline{a}_i) \not = (1,1) \smallskip
\end{array} $ & 6 & 4 & 3 &
$ \displaystyle \begin{array}{c}
\lambda_{\mypar k}^l < \lambda_{\mypar k}^r \, , \\
  \lambda_{\perp k} = \sqrt{\underline{b}^2+|k|^2}
\end{array} $ & 1 \text{or} 2 & 1 \text{or} 2 & 1 \text{or} 2 \\
  \hline
\end{tabular}}}
\caption{Classification of the eigenvalues}
\label{figure1}
\end{figure}

Observe that, by extension of the dispersion relations $ \lambda^\star_{* k} $ to the value $ k=0 $, we obtain $ \lambda_{\mypar 0}^l = 0 $, whereas $ \lambda_{\mypar 0}^r =\lambda_{\perp 0} = \underline{b} > 0 $.
The transition from $ k = 0 $ to $ k \not = 0 $ is therefore as follows. The two non-zero eigenvalues $ \pm {\rm i} \lambda_{\mypar k}^l$ appear which makes $ M_k^0 $ decrease from $ 8 $ to $ 6 $. Along the same lines, the two eigenvalues  $\pm {\rm i} \lambda_k^\jmath$ which are closest to zero for $|k|\sim 0$ (that is $\pm {\rm i} \lambda_k^1$) appear which makes   $ \mathring{M}^0_k $ decrease  from $ 7 $ to $ 4 $ due to the supplementary condition induced by $ \cG_k $ when passing from $ k=0 $ to $ k \not = 0 $.

Note that all eigenvalues $ \lambda^\jmath_k $ depend only on $ \vert k \vert $. In particular, we have $ \lambda^\jmath_{-k} = \lambda^\jmath_k $. Exploiting successively the relations $ \cL_k = \bar \cL_{-k} $ and \eqref{symeigen}, we get that
\[ \cL_{-k} \, r^{\jmath,l}_{-k} = {\rm i} \, \lambda^{\jmath}_{-k} \, r^{\jmath,l}_{-k} \quad \Longrightarrow \quad \cL_{k} \, \bar r^{\jmath,l}_{-k} = - {\rm i} \, \lambda^{\jmath}_{k} \, \bar r^{\jmath,l}_{-k} \quad \Longrightarrow \quad \cL_{k} \, J \, \bar r^{\jmath,l}_{-k} = {\rm i} \, \lambda^{\jmath}_{k} \, J \, \bar r^{\jmath,l}_{-k} \, . \]
This means that $ J \, \bar r^{\jmath,l}_{-k} $ must be an eigenvector of $ \cL_k $ associated with the eigenvalue $ {\rm i} \,\lambda^{\jmath}_{k} $. In fact, as we have done, it is possible to  select the eigenvectors in such a way that
\begin{equation}\label{impaexp}
\qquad J \, \bar r^{\jmath,l}_{k} = \pm r^{\jmath,l}_{-k} \quad \Longrightarrow \quad e^{- {\rm i} \tau \lambda^\jmath_k} \, \big( {}^t (J \bar r^{\jmath,l}_k)  \, \cU_k \big) \ J r^{\jmath,l}_k \, e^{{\rm i} k \cdot x} = \overline{e^{{\rm i} \tau \lambda^\jmath_{-k}} \,  ( {}^t \bar r^{\jmath,l}_{-k}  \, \cU_{-k} ) \  r^{\jmath,l}_{-k} \, e^{- {\rm i} k \cdot x} } \, .
\end{equation}
From now on, in order to differentiate between $ \lambda_{\mypar k}^r $ and $ \lambda_{\perp k} $, we will assume that $  (\underline{a}_e,\underline{a}_i) \not = (1,1) $. A drawback\footnote{The implementation of \eqref{decompocUPP} is not amenable to concise computations concerning  resonances.} of \eqref{decompocUPP} is to deal separately with the case $ \jmath = 0 $, to show an integer $ N_k $ depending on $ k $, and to make appear some $ \mathring{M}^0_k \not = M^0_k $. This can be remedied by duplicating (artificially) the contribution obtained for $ \jmath = 0 $ and by repeating (three times) the eigenvalue $ \lambda^1_0 $ and the eigenvector $ r^{1,l}_0 $. With this in mind, we introduce the coefficients
\[ \tilde \delta^\jmath_k := \left\{ \begin{array}{lcccl}
1 & \text{if} & \jmath \not = 0 & \text{and} & k \not = 0 \, , \\
1/3 & \text{if} & \jmath \not = 0 & \text{and} & k = 0 \, , \\
1/2 & \text{if} & \jmath = 0 \ ,& \ & \ \ \\
\end{array} \right.
\]
and we adopt the following convention
\begin{equation}\label{conventiontilde}
\qquad \tilde M^\jmath_k := \left\{ \begin{array}{lcccl}
  M^\jmath_k & \text{if}  & \jmath \not = 0 \,  & \text{and} &  k \neq 0 \,,  \smallskip \\
  M_0^1     & \text{if}  & \jmath \not = 0 \,  & \text{and} &  k = 0\,,      \smallskip \\
 \mathring{M}^0_k & \text{if} & \jmath = 0\,. &   &
\end{array} \right.
\end{equation}
For $ \jmath \not = 0 $ and $  k = 0 $, the number $ M^\jmath_k $ was not defined before. This is why this special case is specified at the level of \eqref{conventiontilde}. In this continuity, we set
\begin{equation}\label{conventiontildebis}
\tilde M^j_0 := M^1_0 = 3 \, , \qquad \lambda^\jmath_0 := \lambda^1_0=\underline{b} \, , \qquad \forall \, \jmath \in \{1,\cdots, 3\} \, ,
\end{equation}
as well as
\[ r^{\jmath,l}_0 := r^{1,l}_0 \, , \qquad \forall \,l \in\{1,\cdots, 3 \} \, , \qquad \forall \, \jmath \in \{1,\cdots, 3\} \, . \]
Then, in place of \eqref{decompocUPP}, with $ J_- := J $ and $ J_+ := \Id $, we can exhibit the shorter  formulation
\[ e^{\tau \PG_k \cL_k \PG_k} \, \PG_k \cU_k = \sum_{\sigma \in \{-,+\}} \sum_{\jmath=0}^3 \sum_{l=1}^{\tilde M_k^\jmath} \, \tilde \delta_k^\jmath \ e^{\sigma {\rm i} \tau \lambda^\jmath_k} \, \bigl( {}^t (J_\sigma \bar r^{\jmath,l}_k) \, \cU_k \bigr) \ J_\sigma r^{\jmath,l}_k \, . \]
It follows that
\[ e^{\tau \, \PG \cL \PG} \, \PG \, \cU = \sum_{\sigma \in \{ -,+\}} \sum_{k \in \ZZ^3} \sum_{\jmath=0}^3 \sum_{l=1}^{\tilde M_k^\jmath} \, \tilde \delta_k^\jmath \ e^{\sigma {\rm i} \tau \lambda^\jmath_k} \, \bigl( {}^t (J_\sigma \bar r^{\jmath,l}_k) \, \cU_k \bigr) \ J_\sigma r^{\jmath,l}_k \ e^{{\rm i} k \cdot x} \, .  \]
The identity in the right hand side of \eqref{impaexp} together with the change of $ k $ into $ -k $ when $ \sigma = - $, allows to suppress the presence of $ J $. We find that
\begin{equation}\label{appliscat}
e^{\tau \, \PG \cL \PG} \, \PG \, \cU = 2 \, \Re \Bigl( \sum_{k \in \ZZ^3} \cY_k (\tau) \, \cU_k \, e^{{\rm i} k \cdot x} \Bigr) \, ,
\end{equation}
where the linear maps $ \cY_k \equiv \cY_k (\tau) $, which can be seen as complex valued matrices of size $ 14 \times 14 $, are given by
\begin{equation}\label{formcY}
\qquad \cY_k (\tau) \, \cU_k := \sum_{\jmath=0}^3  \, \tilde \delta_k^\jmath \ e^{{\rm i} \tau \lambda^\jmath_k} \ \tilde \Pi^{\jmath}_k \, \cU_k \, , \qquad \tilde \Pi^\jmath_k \, \cU_k := \sum_{l=1}^{\tilde M_k^\jmath}  ( {}^t  \bar r^{\jmath,l}_k \, \cU_k ) \ r^{\jmath,l}_k \, .
\end{equation}
We can recognize here the orthogonal projector $ \tilde  \Pi^{\jmath}_k $ onto the subspace $ \text{Ker} \, (\cL_k - \lambda^\jmath_k \, \Id) \cap \text{Ker} \, \cG_k $.
As expected, we recover here that the function $ e^{\tau \, \PG \cL \PG} \, \PG \, \cU $ is real valued. Given $ n \in \NN^* $, we denote by $ \cI_- z = \bar z $ the complex conjugate of $ z \in \CC^n $, and we just set $ \cI_+ z = z$. By this way, we obtain a relatively simple formula allowing to describe the action of the operator $ e^{\tau \, \PG \cL \PG} \, \PG  $, namely
\begin{equation}\label{appliscatbis}
e^{\tau \, \PG \cL \PG} \, \PG \, \cU
= \sum_{k \in \ZZ^3} \sum_{\sigma \in \{ -,+\}} \cI_\sigma \bigl( \cY_{\sigma k} (\tau) \, \cU_{\sigma k}  \bigr) \, e^{{\rm i} k \cdot x} \, .
\end{equation}


\section{Characterization of resonances.}\label{sec:resonances}
Resonances are detected when looking at the right hand side of \eqref{modueffective}. 
By construction, the matrices $ \cA_j(\cU_0) $ are linear with respect to $ \cU_0 $, while the semilinear term $ f (\cU_0) $ is quadratic with respect to $ \cU_0 $. The spatial derivatives $ \part_{x_j} $ 
commute with the action of $ e^{\tau \, \PG \cL \PG}  $. The matter is therefore to consider expressions like 
\[ \MM_\tau \bigl \lbrack e^{- \tau \, \PG \cL \PG} \PG \cB (e^{\tau \PG \cL \PG} \PG \cU, e^{\tau \PG \cL \PG} \PG \tilde \cU ) \bigr \rbrack \, , \]
where $ \cB : \RR^{14} \longrightarrow \RR^{14} $ is a bilinear map. Exploiting \eqref{appliscatbis}, we can compute
\[ \cB (e^{\tau \PG \cL \PG} \PG \cU, e^{\tau \PG \cL \PG} \PG \tilde \cU ) = \sum_{\sigma',\sigma''\in \{ -,+\}} \sum_{n \in \ZZ^3} f_{n,\sigma',\sigma''} (\tau) \, e^{{\rm i} n\cdot x} , \]
where
\begin{equation}\label{formqn}
f_{n,\sigma',\sigma''} (\tau) := \sum_{k + m = n} \cB \Bigl( \cI_{\sigma'} \bigl(\cY_{\sigma' k}(\tau) \, \cU_{\sigma' k} \bigr) , \cI_{\sigma''} \bigl( \cY_{\sigma'' m} (\tau) \, \tilde \cU_{\sigma'' m} \bigr) \Bigr) \, .
\end{equation}
Applying again  \eqref{appliscatbis}, we obtain
\begin{equation}\label{forintt}
\begin{array}{l}
\MM_\tau \bigl \lbrack e^{- \tau \, \PG \cL \PG} \PG \cB (e^{\tau \PG \cL \PG} \PG \cU, e^{\tau \PG \cL \PG} \PG \tilde \cU ) \bigr \rbrack \smallskip \\
\qquad \qquad \displaystyle =
\sum_{n \in \ZZ^3}
\sum_{\text{\tiny $ \displaystyle \begin{array}{r}
\sigma,\sigma',\sigma''\\
       \in \{ -,+ \}
      \end{array} $}} \MM_\tau \bigl \lbrack  \cI_\sigma \bigl( \cY_{\sigma n} (-\tau) \, f_{\sigma n,\sigma',\sigma''} (\tau) \bigr) \bigr \rbrack \, e^{{\rm i} n \cdot x} \, .
\end{array}
\end{equation}
Using \eqref{formcY}, we get
\[ \MM_\tau \bigl \lbrack  \cI_\sigma \bigl( \cY_{\sigma n} (-\tau) \, f_{\sigma n,\sigma',\sigma''} (\tau) \bigr) \bigr \rbrack = \sum_{\jmath=0}^3 \tilde \delta_n^\jmath \ \MM_\tau \bigl \lbrack e^{- \sigma {\rm i} \tau \lambda^\jmath_n} \ \cI_\sigma ( \tilde \Pi^\jmath_{\sigma n} \, f_{\sigma n,\sigma',\sigma''} ) (\tau) \bigr \rbrack  \, . \]
The role of $ \MM_\tau $ is to eliminate the presence of $ \tau $. In practice, replacing $ f_{\sigma n,\sigma',\sigma''} $ as indicated at the level of  \eqref{formqn} as well as $ \cY_{\sigma' k} \, \cU_{\sigma' k} $ and $ \cY_{\sigma'' m} \, \tilde \cU_{\sigma'' m} $ as specified in \eqref{formcY}, we have to compute under the constraint $ k + m = \sigma n $ expressions like
\[ \MM_\tau \bigl \lbrack e^{{\rm i} \sigma \tau (\sigma'  \lambda^{\jmath'}_k + \sigma'' \lambda^{\jmath''}_m - \lambda^\jmath_n)}\bigr \rbrack = \left \{ \begin{array}{lcl}
0 & \text{if} &   \sigma'  \lambda^{\jmath'}_k + \sigma'' \lambda^{\jmath''}_m - \lambda^\jmath_n \not = 0 \, , \smallskip \\
1 & \text{if} &   \sigma'  \lambda^{\jmath'}_k + \sigma'' \lambda^{\jmath''}_m -   \lambda^\jmath_n = 0 \, .
   \end{array} \right.
\]
This remark motivates a sorting of resonances according to
\[ \cR^{\jmath,\jmath',\jmath''}_{\sigma,\sigma',\sigma''} (n) := \bigl \lbrace \, (k,m)\in \ZZ^3 \times \ZZ^3 \, ; \, k + m = \sigma n \ \text{and} \ \sigma'  \lambda^{\jmath'}_k + \sigma'' \lambda^{\jmath''}_m - \lambda^\jmath_n = 0 \, \bigr \rbrace \, . \]
Taking this into account, there remains (the subscript $ n $ below is for the $ n^{th} $ Fourier coefficient)
\[ \begin{array}{l}
\MM_\tau \bigl \lbrack e^{- \tau \, \PG \cL \PG} \PG \cB (e^{\tau \PG \cL \PG} \PG \cU, e^{\tau \PG \cL \PG} \PG \tilde \cU ) \bigr \rbrack_n \smallskip \\
\qquad \displaystyle =
\sum_{\text{\tiny $ \displaystyle \begin{array}{r}
\sigma,\sigma',\sigma''\\
       \in \{ -,+ \}
      \end{array} $}} \sum_{\text{\tiny $ \displaystyle \begin{array}{r}
\jmath,\jmath',\jmath''\\
       =0
      \end{array} $}}^3
      \sum_{\text{\tiny $ \displaystyle \begin{array}{l}
(k,m)\in \\
 \cR^{\jmath,\jmath',\jmath''}_{\sigma,\sigma',\sigma''} (n)
      \end{array} $}} \tilde \delta_n^\jmath \, \tilde \delta_k^{\jmath'} \, \tilde \delta_m^{\jmath''} \ \cI_\sigma \, \bigl \lbrack \tilde \Pi^\jmath_{\sigma n} \, \cB \bigl(\cI_{\sigma'} (\tilde
\Pi^{\jmath'}_{\sigma' k} \cU_{\sigma' k} ) , \cI_{\sigma''} (\tilde \Pi^{\jmath''}_{\sigma'' m} \tilde \cU_{\sigma'' m} ) \bigr) \bigr \rbrack \, 
\end{array} .\]
The purpose is to determine if FLM can differ from SLM. To this end, our strategy is to focus on the spatial mean value of $ \cU $ which is
\begin{equation} \label{spamean} \langle \cU \rangle (t) := \int_{\TT^3} \cU(t,x) \ dx = \cU_0 (t) \, , 
\end{equation}
and to compare it with the expression $ \langle \PP_{\! e} \, \cU_0 \rangle $ inherited from the SL Model \eqref{modu0effective}. This furnishes a first entry point into the impact of resonances. In this way, the discussion comes down to the study of the zero mode
 $ n = 0 $. From now on, we fix $ n = 0 $. This means that $ m=-k $ and that we only need to look at\footnote{From now on, we can now omit to mention that $ n = 0 $ and that  $ m = -k $.}
\begin{equation} \label{defredresonances}
\quad R^{\jmath,\jmath',\jmath''}_{\sigma',\sigma''} := \bigl \lbrace \, k \in \ZZ^3 \, ; \, \sigma'  \lambda^{\jmath'}_k + \sigma''  \lambda^{\jmath''}_k - \lambda^\jmath_0 = 0 \, \bigr \rbrace \, . 
\end{equation}
With this convention, remark that
\[ \cR^{\jmath,\jmath',\jmath''}_{\sigma,\sigma',\sigma''} (0) = \bigl \lbrace \, (k,-k) \, ; \, k \in R^{\jmath,\jmath',\jmath''}_{\sigma',\sigma''} \, \bigr \rbrace \, .  \]
The aim of this subsection is to describe the content of $ R^{\jmath,\jmath',\jmath''}_{\sigma',\sigma''} $.
It is easy to see that
\[ \begin{array}{ll}
R^{\tilde \jmath,0,0}_{\sigma',\sigma''} =
R^{0,\tilde \jmath,0}_{\sigma',\sigma''} =
R^{0,0,\tilde \jmath}_{\sigma',\sigma''} = \emptyset \, , \qquad & \forall \, \tilde \jmath \in \{ 1,2,3 \} \, .
   \end{array}
\]
Thus, either $ (\jmath,\jmath',\jmath'') = (0,0,0) $ or at least two indices among $ \jmath $, $ \jmath' $ and $ \jmath'' $ must be non-zero.
In Paragraph \ref{jpour000}, we first  examine what happens when $ (\jmath,\jmath',\jmath'') = (0,0,0) $.
In Paragraph \ref{jpour000k}, we put apart the case $ k = 0 $ as well as $ (\jmath,\jmath',\jmath'') \not = (0,0,0) $. In Paragraph \ref{jpour0}, we investigate the configuration $ \jmath = 0 $ with $ \jmath' \not = 0 $, $\jmath'' \not = 0 $ and $ k \not = 0 $. In Paragraph \ref{jpournot0},
we examine the complementary case $ \jmath \not = 0 $ and $ k \not = 0 $. All these cases are distinct and, at the end, the discussion will be exhaustive.


\subsection{Case 1 when $ (\jmath,\jmath',\jmath'')= (0,0,0) $}\label{jpour000} For $ \jmath = \jmath' =\jmath'' = 0 $, the three eigenvalues $ \lambda^0_k $, $ \lambda^0_m $ and $ \lambda^0_0 $ are just zero. Time oscillations are not implied. This corresponds to the prepared configuration. Observe that $ R^{0,0,0}_{\sigma',\sigma''} = \ZZ^3 $.


\subsection{Case 2 when $ k = 0 $ and $ (\jmath,\jmath',\jmath'') \not = (0,0,0) $}\label{jpour000k} As can be seen in Figure \ref{figure1}, the selection of $ k = 0 $ needs special treatment. With the convention \eqref{conventiontildebis}, we have to look at the condition
\[ \sigma' \lambda^{\jmath'}_0 + \sigma'' \lambda^{\jmath''}_0 - \lambda^\jmath_0 = 0 \, , \qquad \lambda^\jmath_0 = \left \lbrace \begin{array}{lcl}
     0 & \text{if} & \jmath = 0  \, , \\
     \underline{b} & \text{if} & \jmath \not = 0 \, . \end{array} \right.
 \]
It follows that\footnote{The description of resonances is very sensitive to the enumeration of eigenvalues. In the absence of \eqref{conventiontildebis}, the discussion would be different with only one indice $ \jmath = 1 $ implied.}
\[ \begin{array}{llll}
0 \in R^{0,\jmath',\jmath''}_{\sigma',\sigma''} & \text{when} \ \jmath' \not = 0 \ , \! \! & \jmath'' \not = 0 \ \, \text{and} \! \! & \sigma' \sigma'' = - \, , \\
0 \in R^{\jmath,0,\jmath''}_{\sigma',\sigma''} & \text{when} \ \jmath \not = 0 \ , \! \! & \jmath'' \not = 0 \ \, \text{and} \! \! & \sigma'' = + \, , \\
0 \in R^{\jmath,\jmath',0}_{\sigma',\sigma''} & \text{when} \ \jmath \not = 0 \ , \! \! & \jmath' \not = 0 \ \ \text{and} \! \! & \sigma' = + \, . \\
   \end{array} \]
In all other situations (for instance when $ \jmath \not = 0 $, $ \jmath' \not = 0 $ and $ \jmath'' \not = 0 $), the origin $ 0 \equiv (0,0,0) $ of $ \ZZ^3 $ does not belong to $ R^{\jmath,\jmath',\jmath''}_{\sigma',\sigma''} $.


\subsection{Case 3 when $ \jmath = 0 $ with $ \jmath' \not = 0 $, $ \jmath'' \not = 0 $ and $ k \not = 0 $}\label{jpour0}
Here, we must consider
\[ R^{0,\jmath',\jmath''}_{\sigma',\sigma''} \setminus \{ 0 \} = \bigl \lbrace 0 \not = k \in \ZZ^3 \, ; \, \sigma'  \, \lambda^{\jmath'}_k + \sigma'' \, \lambda^{\jmath''}_k = 0 \bigr \rbrace = R^{0,\jmath',\jmath''}_{-\sigma',-\sigma''} \setminus \{ 0 \}  \, . \]
Since all implied eigenvalues are positive, assuming that $ \sigma' \,\sigma'' = + $, we have
\begin{equation}\label{eq:const0}
 R^{0,\jmath',\jmath''}_{\sigma',\sigma''} \setminus \{ 0 \}  = \emptyset
        \quad \text{when} \ \sigma' \,\sigma'' = + \ .
\end{equation}
Otherwise, when $ \sigma' \,\sigma'' = - $, we have to deal with
\[ R^{0,\jmath',\jmath''}_{\sigma',\sigma''} \setminus \{ 0 \} = \bigl \lbrace \, 0 \not = k \in \ZZ^3 \, ; \, \lambda^{\jmath'}_k = \lambda^{\jmath''}_k \, \bigr \rbrace \, .\]
For $ \jmath' \not = \jmath'' $ and $ k \not = 0 $, the two  eigenvalues $ \lambda^{\jmath'}_k $ and $ \lambda^{\jmath''}_k $ are distinct so that
\begin{equation*}\label{eq:const1}
\qquad R^{0,\jmath',\jmath''}_{\sigma',\sigma''} \setminus \{ 0 \} = \emptyset
        \quad \text{when} \ \sigma' \,\sigma'' = - \ \text{as well as} \  \jmath' \not = 0 \, , \jmath'' \not = 0 \ \text{and} \ \jmath' \not = \jmath'' \, .
\end{equation*}
On the contrary, when $ \jmath' = \jmath'' $, the separation of the eigenvalues disappears, and we obtain
\begin{equation}\label{eq:const2}
\qquad R^{0,\jmath',\jmath'}_{\sigma',\sigma''} \setminus \{ 0 \} = \ZZ^3 \setminus \{ 0 \} \quad \text{when} \ \sigma' \,\sigma'' = - \ \text{and} \  \jmath' \not = 0 \, .
\end{equation}
This latest configuration is of special interest because it implies an infinite number of resonances, which is likely to bring an important contribution.


\subsection{Case 4 when $ \jmath \not= 0 $ and $ k \not = 0 $}\label{jpournot0} Here, we have to deal with
\begin{equation}\label{Case44}
R^{\jmath,\jmath',\jmath''}_{\sigma',\sigma''} \setminus \{ 0 \} := \bigl \lbrace \, 0 \not = k \in \ZZ^3 \, ; \, \sigma' \,  \lambda^{\jmath'}_k + \sigma'' \, \lambda^{\jmath''}_k = \underline{b} \, \bigr \rbrace \, .
\end{equation}  

In Paragraph \ref{jpournot0fen}, we start with general considerations. Then, in Paragraph \ref{jpournot0sp}, we impose special assumptions in order to complement the analysis.


\subsubsection{The general situation}\label{jpournot0fen} We can assert that the number of resonances is finite, as a corollary of the following statement.

\begin{lemma}[About the content of $ R^{\jmath,\jmath',\jmath''}_{\sigma',\sigma''} \setminus \{ 0 \}  $ when $ \jmath \not = 0 $]\label{comproperdrebis} Assume that $ \underline{a}_e \not = 1 $, $  \underline{a}_i \not = 1 $ and $ \jmath \not = 0 $. Then, the set
$ R^{\jmath,\jmath',\jmath''}_{\sigma',\sigma''} \setminus \{ 0 \}  \subset \ZZ^3 $ is  bounded.
\end{lemma}

In fact, according to the choice of $ (\jmath,\jmath',\jmath'') $ or $ (\sigma',\sigma'')$, the proof given below presents a more refined description of $ R^{\jmath,\jmath',\jmath''}_{\sigma',\sigma''} \setminus \{ 0 \} $.

\begin{proof}
It is clear that $ R^{\jmath,\jmath',\jmath''}_{-,-} = \emptyset $. Exploiting \eqref{equi+}, it is easy to deduce that $ R^{\jmath,\jmath',\jmath''}_{+,+} $ is bounded. There remains to work with $ \sigma' \, \sigma'' = -1 $. When $ \jmath' = \jmath'' $, since $ 0 \not = \underline{b} $, we find that $ R^{\jmath,\jmath',\jmath'}_{\sigma', \sigma''} = \emptyset $. When $ \jmath' \not = \jmath'' $, coming back to \eqref{equi+}, we can assert that
\[ \vert \sigma' \,  \lambda^{\jmath'}_k + \sigma'' \, \lambda^{\jmath''}_k \vert \underset{\vert k \vert \rightarrow +\infty}{\sim} c \ \vert k \vert \, ,  \]
with
\[ c \in \bigl \lbrace \, \vert  \min \, (\underline{a}_e ;\underline{a}_i) - 1 \vert \, ; \, \vert \max \, (\underline{a}_e ;\underline{a}_i) - 1 \vert \, ; \, \max \, (\underline{a}_e ;\underline{a}_i) - \min \, (\underline{a}_e ;\underline{a}_i) \, \bigr \rbrace \, . \]
When $ \underline{a}_e \not = \underline{a}_i $, we find that $ c \in \RR_+^* $ and $ c \, \vert k \vert > \underline{b} $ as soon as $ \vert k \vert $ is large enough. When $ \underline{a}_e = \underline{a}_i \not = 1 $, the same argument holds except when $ \lambda^{\jmath'}_k = \lambda_{\mypar k}^l $ and $ \lambda^{\jmath''}_k = \lambda_{\mypar k}^r $ or when $ \lambda^{\jmath'}_k = \lambda_{\mypar k}^r $ and $ \lambda^{\jmath''}_k = \lambda_{\mypar k}^l $. In this latter situation, we have to test the relation
\[ \lambda_{\mypar k}^r - \lambda_{\mypar k}^l = \sqrt {\underline{b}^2 + \underline{a}_e^2 \, \vert k \vert^2 } - \underline{a}_e \, \vert k \vert = \underline{b} \, , \]
which leads to a contradiction when $ \vert k \vert \not = 0 $.
\end{proof}

Whether $ R $ is empty or not 
(when $ \jmath \not= 0 $ and $ k \not = 0 $) 
depends on the values of the parameters $ \underline{a}_s $ and $ \underline{b}_s $
managing $ \cL $. The formulas leading to $ \lambda_k^\jmath $ are too complicated in order to get an exhaustive description of $ R $. At this stage, we can only guess that resonances are a rare event and that, generically, we should have $ R = \emptyset $ (as illustrated below).


\subsubsection{A special situation}\label{jpournot0sp} We consider here data (pressures $ p_s $, density $ \underline{n} $ and masses $ m_s $) which are adjusted in such a way that $ \underline{a}_e = \underline{a}_i $.

\begin{assumption}\label{theascoin} There exists some $ \underline{a} \in \RR_+^* $ such that $ \underline{a} = \underline{a}_e = \underline{a}_i < 1  $.
\end{assumption}

\noindent The case $ \underline{a}>1 $ is completely similar. For this reason, it is not necessary to develop it. The advantage of the condition $ \underline{a}_e = \underline{a}_i $ is to allow more explicit computations of the eigenvalues. As a matter of fact, we find
\begin{equation}\label{rangemt} 
\lambda^l_{\mypar k} = \underline{a} \ \vert k \vert < \lambda^r_{\mypar k} = \sqrt{\underline{b}^2 + \underline{a}^2 \, \vert k \vert^2} < \lambda_{\perp k} = \sqrt{\underline{b}^2 + \vert k \vert^2} \, . 
\end{equation}
From there, knowing that $ \jmath \not= 0 $ and $ k \not = 0 $, we can categorize the content of $ R $ according to the different choices of $ \jmath' $, $ \jmath'' $, $ \sigma' $ and $ \sigma'' $. We first investigate what happens when $ \jmath' = 0 $. The role of $ \jmath' $ and $ \jmath'' $ being symmetric, the discussion when $ \jmath'' = 0 $ is the same. For $ \jmath' = 0 $, we have $ \lambda^{\jmath'}_k = 0 $.
Coming back to \eqref{Case44}, we have to impose $ \sigma'' \, \lambda^{\jmath''}_k = \underline{b} $, so  that $ \sigma'' = + $. In view of \eqref{rangemt}, the only way to see a resonance is to select $ \jmath'' $ in such a way that 
$ \lambda^{\jmath''}_k = \lambda^l_{\mypar k} $. The index $ \jmath'' $ must be associated with the parallel mode $ \mypar $, and 
\[ R_{\sigma',-}^{\jmath,0,\jmath''} = \emptyset \, , \qquad R_{\sigma',+}^{\jmath,0,\jmath''} = \bigl \lbrace \, 0 \not = k \, ; \, \vert k \vert = \underline{b} / \underline{a} \, \bigr \} \, , \qquad \jmath \not = 0 \, . \]
By Legendre's three-square theorem, this is satisfied if and only if
\begin{equation}\label{threesquare}
\exists \, (\gamma,\delta) \in \NN^2 \, ; \qquad (\underline{b}/\underline{a})^2 = k_1^2 + k_2^2 + k_3 ^2 = 4^\gamma \ (8 \, \delta + 7 ) \, . 
\end{equation}
Moving on to the case of $ \jmath' \not = 0 $ and $ \jmath'' \not = 0 $, we must look at 
 $ \sigma' \,  \lambda^{\jmath'}_k + \sigma'' \, \lambda^{\jmath''}_k = \underline{b}  $. This relation is not modified by exchanging $ (\jmath',\sigma') $ and $ (\jmath'',\sigma'') $.
Thus, the case $ (\sigma', \sigma'') = (- , +) $ can be deduced from $ (\sigma', \sigma'') = (+ ,-) $. We distinguish between the remaining choices for the signs  $ \sigma' $ and $ \sigma'' $.

\smallskip

\noindent $ \bullet $ When $ \sigma' = \sigma'' = - $, we find $ R^{\jmath,\jmath',\jmath''}_{\sigma',\sigma''} = \emptyset $.

\smallskip

\noindent $ \bullet $ When $ \sigma' = \sigma'' = + $, since $ \underline{b} < \lambda^r_{\mypar k} < \lambda_{\perp k}  $ when $ k \not = 0 $, the only way to exhibit resonances is to take $  \lambda^{\jmath'}_k=\lambda^{\jmath''}_k = \lambda^l_{\mypar k} $, and then 
\[ R^{\jmath,\jmath',\jmath''}_{\sigma',\sigma''} = \bigl \lbrace \, 0 \not = k \in \ZZ^3 \, ; \, \vert k \vert = \underline{b} / (2 \, \underline{a}) \, \bigr \} \, . \]
Again, the above condition on $ k $ can be  characterized by Legendre's three-square theorem.

\smallskip

\noindent $ \bullet $ When $ \sigma' = + $ and $ \sigma'' = - $, a basic calculation indicates that we cannot have $ \lambda^r_{\mypar k} - \lambda^l_{\mypar k}  = \underline{b} $. Only two options remain.  
Either $ \lambda^{\jmath'}_k=\lambda_{\perp k}$ and $\lambda^{\jmath''}_k = \lambda^l_{\mypar k} $, so that 
\[ R^{\jmath,\jmath',\jmath''}_{+,-} = \bigl \lbrace \, 0 \not = k \in \ZZ^3 \, ; \, |k|=2 \, \underline{a} \, \underline{b}/(1-\underline{a}^2) \, \bigr \} \, . \] 
Or $ \lambda^{\jmath'}_k=\lambda_{\perp k}$ and $\lambda^{\jmath''}_k = \lambda^r_{\mypar k} $, in which case
\[ R^{\jmath,\jmath',\jmath''}_{+,-} = \bigl \lbrace \, k \in \ZZ^3 \setminus \{ 0 \} \, ; \, |k|= \underline{b} \,  g_\pm (\underline{a}) \, \bigr \} \, , \qquad g_\pm (\underline{a}) := \frac{\bigl( 1 +\underline{a}^2 \pm 2 \, \sqrt{1- \underline{a}^2+ \underline{a}^4} \bigr)^{1/2}}{1-\underline{a}^2} \, . \] 
The preceding discussion can be summarized as follows.

\begin{lemma}\label{condabsre} [Conditional absence of resonances] Under Assumption \ref{theascoin}, when
\begin{equation}\label{atmostden} \frac{|k|}{\underline{b}} \not = \frac{1}{\underline{a}} \, , \qquad \frac{|k|}{\underline{b}} \not = \frac{1}{2 \, \underline{a}} \, , \qquad \frac{|k|}{\underline{b}} \not = \frac{2 \, \underline{a}}{1 - \underline{a}^2} \, , \qquad  \frac{|k|}{\underline{b}} \not = g_-(\underline{a}) \, , \qquad \frac{|k|}{\underline{b}} \not = g_+ (\underline{a}) \, ,
\end{equation}
as soon as $ \jmath \not= 0 $, we can assert that $ R^{\jmath,\jmath',\jmath''}_{\sigma',\sigma''} \setminus \{ 0 \}  = \emptyset $.
\end{lemma}

Given $ k \not = 0 $ and $ \underline{b} \in \RR_+^* $, the number of $ \underline{a} \in ]0,1[ $ such that \eqref{atmostden} is not satisfied is finite. Thus, given $ \underline{b} $, the cardinal of values $ \underline{a} \in ]0,1[ $ such that \eqref{atmostden} is not verified for all $ k \in \ZZ^3 $ is at most countable. This confirms that, for generic values of $ (\underline{a},\underline{b}) $, 
there is no resonance in the setting of Paragraph \ref{jpournot0sp}.


\section{Transparencies related to the source term}\label{transparenciesst}
The Euler--Maxwell Two-Fluid (EMTF) system is characterized by a strong coupling between the matter (electrons and ions) and the electromagnetic fields (through the Lorentz force and the electric current). This interplay covers semilinear and quasilinear facets which can be studied separately. The focus here is on the semilinear part. When performing the asymptotic analysis, the source term gives rise to the following quadratic map
\begin{equation} \label{defdefcU}
f (\cU) := {}^t \Bigl( 0,0, - \frac{{}^t (u_e \times B)}{m_e}, \frac{{}^t (u_i \times B)}{m_i},\frac{\varrho_e \,  {}^t u_e}{\sqrt{m_e \, \underline{p}'_e}} - \frac{\varrho_i \, {}^t  u_i}{\sqrt{m_i \, \underline{p}'_i}}, 0
\Bigr) .
\end{equation}
Through the polarization identity,
this can be alternatively expressed as
\begin{equation}\label{polarizationid}
\qquad f(\cU) = \cB (\cU,\cU) \, , \qquad \cB (\cU,\tilde \cU) := \bigl( f(\cU+\tilde \cU)-f(\cU- \tilde \cU) \bigr) /4  = \cB (\tilde \cU,\cU) \, .
\end{equation}
In \eqref{polarizationid}, given two entries $ \cU = {}^t (\varrho_e,\varrho_i,{}^t u_e,{}^t u_i,{}^t E,{}^t B) $ and $ \tilde \cU = {}^t (\tilde \varrho_e, \tilde \varrho_i,{}^t \tilde u_e,{}^t \tilde u_i,{}^t \tilde E,{}^t \tilde B) $, the bilinear symmetric map $ \cB $ is given by
\begin{equation} \label{defdefcUbi} \begin{array}{rl}
\cB (\cU,\tilde \cU) \! \! \! & \displaystyle = {}^t \Bigl( 0,0, - \frac{{}^t (u_e \times \tilde B + \tilde u_e \times B )}{m_e}, \frac{{}^t (u_i \times \tilde B + \tilde u_i \times B)}{m_i}, \smallskip \\
\ & \displaystyle \qquad \qquad \qquad \qquad \qquad \frac{\varrho_e \,  {}^t \tilde u_e + \tilde \varrho_e \,  {}^t u_e}{\sqrt{m_e \, \underline{p}'_e}} - \frac{\varrho_i \,  {}^t \tilde u_i + \tilde \varrho_i \,  {}^t u_i }{\sqrt{m_i \, \underline{p}'_i}}, 0
\Bigr) .
   \end{array}
   \end{equation}
Coming back to the preliminary discussion in  Section \ref{sec:resonances}, we  concentrate on the zero mode $ n = 0 $. With $ \cB $ as in \eqref{defdefcUbi}, this means to look at
\begin{equation}\label{evolutionU0}
\begin{array}{rl}
 \displaystyle \part_t \, \cU_0 = \! \! \!
\sum_{
\sigma',\sigma''
       \in \{ -,+ \}} \ \sum_{
\jmath,\jmath',\jmath''
       =0}^3 & \displaystyle  \sum_{
k \in R^{\jmath,\jmath',\jmath''}_{\sigma',\sigma''}} \! \! \! 2 \ \tilde \delta_0^\jmath \, \tilde \delta_k^{\jmath'} \, \tilde \delta_k^{\jmath''} \\
 \ & \displaystyle \qquad \times \Re \, \bigl \lbrack \tilde \Pi^\jmath_0 \, \cB \bigl(\cI_{\sigma'} (\tilde
\Pi^{\jmath'}_{\sigma' k} \cU_{\sigma' k} ) , \cI_{\sigma''} (\tilde \Pi^{\jmath''}_{- \sigma'' k} \cU_{- \sigma'' k} ) \bigr) \bigr \rbrack \, .
\end{array}
\end{equation}
In accordance with the structure of Section \ref{sec:resonances}, the right hand side may be decomposed into
\[ \part_t \, \cU_0 =  \cF_1 + \cF_2 + \cF_3 + \cF_4 \, , \]
where the subscript $ a $ inside $ \cF_a $ means that, in the above sum, the selection of $ (\jmath,\jmath',\jmath'',\sigma',\sigma'') $ as in ``Case $ a $'' of Subsection  \ref{sec:resonances}.$a $ is retained. Below, we discuss separately the contents of the $ \cF_a $.


\subsection{Case 1}\label{apportcase1} When $ (\jmath,\jmath',\jmath'')= (0,0,0) $, we have to deal with
\[ \cF_1 = \frac{1}{4} \, \sum_{
k\in \ZZ^3} \Re \Bigl \lbrack \tilde \Pi^0_{0} \,
\cB \Bigl( \sum_{\sigma'} \cI_{\sigma'} (\tilde
\Pi^{0}_{\sigma' k} \,  \cU_{\sigma' k} ) ,\sum_{\sigma''} \cI_{\sigma''} (\tilde \Pi^{0}_{-\sigma'' k} \,  \cU_{-\sigma'' k} ) \Bigr) \Bigr \rbrack \, . \]
In view of \eqref{cestr010}, we have $ r^{0,l}_0 = \bar r^{0,l}_0 $. Coming back to \eqref{cK0knot0} and \eqref{cK0knot0bis}, we can check that $ r^{0,l}_{-k} = \pm \bar r^{0,l}_k $ when $ k \not = 0 $. Briefly
\[ \cI_- (\tilde
\Pi^{0}_{-k} \, \cU_{-k} ) = \sum_{l=1}^4  ( {}^t r^{0,l}_{-k} \, \bar \cU_{-k} ) \ \bar r^{0,l}_{-k} = \tilde
\Pi^{0}_{k} \, \cU_{k} \, . \]
It follows that
\[ \cF_1 = \sum_{
k\in \ZZ^3} \Re \bigl \lbrack \tilde \Pi^0_{0} \, \cB (\tilde
\Pi^0_{k} \, \cU_{k} , \cI (\tilde \Pi^0_{k} \, \cU_{ k} ) ) \bigr \rbrack = \sum_{
k\in \ZZ^3} \tilde \Pi^0_{0} \, \bigl \lbrack f \bigl( \Re ( \tilde
\Pi^0_{k} \, \cU_{k} ) \bigr) + f \bigl( \Im ( \tilde
\Pi^0_{k} \, \cU_{k} ) \bigr) \bigr \rbrack \, . \]

\begin{lemma}[Polarization of the source term]\label{polarization} Given $ \cU $ as in \eqref{decompocUen}, define $ \cW := -{\underline b}_e u_e + {\underline b}_i u_i $. Then
\begin{equation}\label{tildepi000}
\qquad \tilde \Pi^0_{0} \, f(\cU) = \frac{1}{\sqrt{m_e \, m_i}} \ \frac{1}{{\underline b}^2} \ {}^t \bigl( 0, 0, {\underline b}_i \,  {}^t (\cW \times B) , {\underline b}_e \,  {}^t (\cW \times B) , 0, 0 \bigr) \, .
\end{equation}
\end{lemma}

\begin{proof} In view of \eqref{cestr010} and \eqref{defdefcU}, we have
\[ \tilde \Pi^0_{0} \, f(\cU) = \sum_{l=1}^7 \bigl(  {}^t \bar r^{0,l}_0 \, f(\cU)\bigr) \ r^{0,l}_0 = \sum_{l=2}^4 \bigl(  {}^t \bar r^{0,l}_0 \, f(\cU)\bigr) \ r^{0,l}_0 \, , \]
which corresponds to the orthogonal projection of the two velocity components $ (u_e,u_i) $ onto the direction $ (\underline{b}_i,\underline{b}_e) $. It suffices to compute
\[ \frac{1}{{\underline b}} \  \Bigl( - \frac{{\underline b}_i}{m_e} \, u_e + \frac{{\underline b}_e}{m_i} \, u_i \Bigr) = \frac{1}{\sqrt{m_e \, m_i}} \ \frac{1}{{\underline b}} \ (- {\underline b}_e \, u_e + {\underline b}_i \, u_i ) \, .  \]
\end{proof}

Taking into account \eqref{cestr010}, more precisely the relation $u_e=\underline{b}_i u_i/\underline{b}_e$ satisfied by $ r^{0,l}_0 $ for $ l = 3 $ and $ l = 4 $, the velocity $ \cW $ associated with $ \tilde \Pi^0_{0} \, \cU_0 $ is just $ 0 $. Thus, as a corollary of Lemma \ref{polarization}, we have
\[ \tilde \Pi^0_{0} \, \bigl \lbrack f \bigl( \Re ( \tilde
\Pi^0_0 \, \cU_0 ) \bigr) \bigr \rbrack = 0 \, , \qquad \tilde \Pi^0_{0} \, \bigl \lbrack f \bigl( \Im ( \tilde
\Pi^0_0 \, \cU_0 ) \bigr) \bigr \rbrack = 0 \, . \]
There remains to examine what happens when $ k \not = 0 $, that is to compute
\[ \cF_1 = \sum_{0 \not =
k\in \ZZ^3} \tilde \Pi^0_{0} \, \bigl \lbrack f \bigl( \Re ( \tilde
\Pi^0_{k} \, \cU_{k} ) \bigr) + f \bigl( \Re ( \tilde
\Pi^0_{k} \, {\rm i} \, \cU_{k} ) \bigr) \bigr \rbrack \, . \]

\begin{lemma}[Transparency of the part of the source term leading to the slow limit model] \label{transpa0}
  We have $\cF_1 = 0 $.
\end{lemma}

\begin{proof} Recall that
\[ \tilde
\Pi^0_{k} \, \cU_{k} = \sum_{k=1}^4 ( {}^t \bar r^{0,l}_k \, \cU_k) \ r^{0,l}_k = \alpha_k \, r^{0,1}_k + \beta_k \, r^{0,2}_k + \gamma_k \, r^{0,3}_k + \delta_k \, r^{0,4}_k \, . \]
When computing the velocity $ \cW_k $ and the magnetic field $ B_k $ issued from $ \tilde
\Pi^0_{k} \, \cU_{k} $, in view of \eqref{cK0knot0}, the two vectors $ r^{0,1}_k $ and $ r^{0,2}_k $ do not contribute. There remains
\[ \Re \cW_k = - \frac{\underline{b} \,  |k|}{ (\underline{b}^2 + |k|^{2})^{1/2}} \ (\Re \gamma_k \, \mathbbm{e}^2 + \Re \delta_k \, \mathbbm{e}^3)
\, , \qquad \Re B_k =  \frac{\underline{b}}{ (\underline{b}^2 + |k|^{2})^{1/2}} \ (\Im \gamma_k \, \mathbbm{e}^3 - \Im \delta_k \, \mathbbm{e}^2) \, . \]
It follows that
\[ \Re \cW_k \times \Re B_k = - \frac{\underline{b}^2 \,  |k|^2}{\underline{b}^2 + |k|^{2}} \ (\Re \gamma_k \, \Im \gamma_k + \Re \delta_k \, \Im \delta_k ) \ k \, . \]
Changing $ \cU_k $ into $ {\rm i} \, \cU_k $ is the same as replacing $ \gamma_k $ and $ \delta_k $ respectively by $ {\rm i} \, \gamma_k $ and $ {\rm i} \, \delta_k $. Since
\begin{equation}\label{symanepaso}
\Re ({\rm i} \, \gamma_k ) = - \Im \gamma_k \, , \qquad \Im ({\rm i} \, \gamma_k ) = \Re \gamma_k \, , \qquad \Re ({\rm i} \, \delta_k ) = - \Im \delta_k \, , \qquad \Im ({\rm i} \, \delta_k ) = \Re \delta_k \, ,
\end{equation}
we find that
\begin{equation}\label{symanepasoconc} \Re ({\rm i} \, \gamma_k) \, \Im ({\rm i} \, \gamma_k) = - \Re \gamma_k \, \Im \gamma_k \, , \qquad \Re ({\rm i} \, \delta_k) \, \Im ({\rm i} \, \delta_k) = - \Re \delta_k \, \Im \delta_k \, ,
 \end{equation}
which furnishes the expected cancellation.
\end{proof}

\begin{remark} The prepared configuration is fully studied in  \cite{BC25p}. It leads to the modulation equation (2.13) inside  \cite{BC25p}, or see \eqref{putxmhd}. Since
\[ \int_{\TT^3} (\nabla \times B ) \times B \ dx = 0 \quad \text{when} \quad \nabla \cdot B = 0 \, , \]
from \cite{BC25p}-(2.26) or just \eqref{putxmhd}, we can deduce that
the source term of SLM does not contribute to the time evolution of $ \cU_0 $. This remark is coherent with Lemma \ref{transpa0}. Indeed, since $ \cF_1 = 0 $, the term $ \cF_1 $ is as expected absent from \eqref{evolutionU0}. \hfill $ \circ $
 \end{remark}


\subsection{Case 2}\label{apportcase2} The part $ \cF_2 $ collects all (non-prepared) resonances implying the frequency $ k = 0 $, that is
\[ \begin{array}{rl}
\cF_2 = & \displaystyle \! \! \! \! \!
\sum_{
\sigma'
       \in \{ -,+ \}} \ \sum_{
\jmath',\jmath''
       =1}^3 2 \ \tilde \delta_0^0 \, \tilde \delta_0^{\jmath'} \, \tilde \delta_0^{\jmath''} \
 \Re \, \bigl \lbrack \tilde \Pi^0_0 \, \cB \bigl(\cI_{\sigma'} (\tilde
\Pi^{\jmath'}_0 \cU_0 ) , \cI \,  \cI_{\sigma'} (\tilde \Pi^{\jmath''}_0 \cU_0 ) \bigr) \bigr \rbrack \\
\ + & \displaystyle \! \! \! \! \!
\sum_{
\sigma'
       \in \{ -,+ \}} \ \sum_{
\jmath,\jmath''
       =1}^3 2 \ \tilde \delta_0^\jmath \, \tilde \delta_0^0 \, \tilde \delta_0^{\jmath''} \
 \Re \, \bigl \lbrack \tilde \Pi^\jmath_0 \, \cB \bigl(\cI_{\sigma'} (\tilde
\Pi^0_0 \, \cU_0 ) , \tilde \Pi^{\jmath''}_0 \cU_0 ) \bigr) \bigr \rbrack \\
\ + & \displaystyle \! \! \! \! \!
\sum_{
\sigma''
       \in \{ -,+ \}} \ \sum_{
\jmath,\jmath'
       =1}^3 2 \ \tilde \delta_0^\jmath \, \tilde \delta_0^{\jmath'} \, \tilde \delta_0^0 \
 \Re \, \bigl \lbrack \tilde \Pi^\jmath_0 \, \cB \bigl(\tilde
\Pi^{\jmath'}_0 \, \cU_0 , \cI_{\sigma''} ( \tilde \Pi^0_0 \, \cU_0 ) \bigr) \bigr \rbrack \, .
\end{array}\]
When $ \cU $ (and therefore $ \cU_0 $) is not prepared, the oscillations coming from $ \lambda^1_0 = \underline{b} $ may combine to produce $ \cF_2 $. It is clear that $ \cF_2 $
collects all such interactions and that it is only a function of $ \cU_0 $, while $ \cF_3 $ and $ \cF_4 $ do not depend on $ \cU_0 $ (but on the $ \cU_k $ with $ k \not = 0 $).

It is imporant to deal with $ \cF_2 $ separately, and to determine whether $ \cF_2 $ disappears or not. Indeed, if $ \cF_2 \not = 0 $, by adjusting all $ \cU_k $ with $ k \not = 0 $ such that $ {\cU_k}_{\mid t=0} = 0 $ and $ {\cU_0}_{\mid t=0} $ in such a way that $ \cF_2 \not = 0 $, we can obtain that $ \part_t \cU_0 \not = 0 $. Thus, $ \cU_0 $ does not remain constant, in contrast with the SLM situation. In other words, showing that $ \cF_2 \not = 0 $ is already a way to prove that FLM differs from SLM. Now, since $ \cU_0 $ is real valued, we find that
\[ \cF_2 = 2 \ \tilde \Pi^0_0 \, \bigl \lbrack f \bigl( \Re (\tilde
\Pi^1_0 \, \cU_0) \bigr) + f \bigl( \Im (\tilde
\Pi^1_0 \, \cU_0) \bigr)  \bigr \rbrack + \, 4 \
 \Re \, \tilde \Pi^1_0 \, \bigl \lbrack  \cB (\tilde \Pi^0_0 \, \cU_0 , \tilde
\Pi^1_0 \, \cU_0 )
\bigr \rbrack \, . \]
In view of \eqref{cestlasuite}, the density and magnetic components of $ \tilde
\Pi^1_0 \, \cU_0 $ are just zero. More precisely, with $ \cW = -{\underline b}_e u_e + {\underline b}_i u_i $ as before, we find that
\begin{equation}\label{pi10}
\tilde \Pi^1_0 \, \cU = \frac{1}{2 \, \underline{b}^2} \ {}^t \bigl( 0, 0 , -  \underline{b}_e \, {}^t (\cW - {\rm i} \, \underline{b} \, E ) , \underline{b}_i \, {}^t (\cW - {\rm i} \, \underline{b} \, E ), {\rm i} \,  \underline{b} \, {}^t (\cW - {\rm i} \, \underline{b} \, E ), 0 \bigr) \, .
\end{equation}
Applying Lemma \ref{polarization}, there remains $ \cF_2 = 4 \
 \Re \, \tilde \Pi^1_0 \, \bigl \lbrack  \cB (\tilde \Pi^0_0 \, \cU_0 , \tilde
\Pi^1_0 \, \cU_0 )
\bigr \rbrack $. Next compute
\[ \tilde \Pi^0_0 \, \cU = \frac{1}{\underline{b}^2} \ {}^t \bigl( c \, \underline{b}^2 \, \sqrt{\underline{p}'_e}, c \,  \underline{b}^2 \,  \sqrt{\underline{p}'_i} , \underline{b}_i \, {}^t (\underline{b}_i \, u_e + \underline{b}_e \, u_i ) , \underline{b}_e \, {}^t (\underline{b}_i \, u_e + \underline{b}_e \, u_i ) , 0 ,  \underline{b}^2 \,  {}^t B \bigr) \, ,
\]
where
\[ c \equiv c(\varrho_e,\varrho_i) := \bigl( \sqrt{\underline{p}'_e} \, \varrho_e + \sqrt{\underline{p}'_i} \, \varrho_i \bigr) / (\underline{p}'_e + \underline{p}'_i) \, .  \]
Coming back to \eqref{defdefcUbi}, with $ \cW_0 = -{\underline b}_e u_{e0} + {\underline b}_i u_{i0} $, we can extract
\[ \begin{array}{rl}
\cB (\tilde \Pi^0_0 \, \cU_0 , \tilde
\Pi^1_0 \, \cU_0 ) \! \! \! & \displaystyle = \frac{1}{2 \, \underline{b}^2} \ {}^t \Bigl( 0, 0 , \frac{\underline{b}_e}{m_e} \, {}^t \bigl( (\cW_0 - {\rm i} \, \underline{b} \, E_0 ) \times B_0 \bigr) , \frac{\underline{b}_i}{m_i} \, {}^t \bigl( (\cW_0 - {\rm i} \, \underline{b} \, E_0 ) \times B_0 \bigr) , \\
\ & \displaystyle \qquad \qquad \qquad \qquad \qquad \qquad \quad - c \, \bigl((\underline{b}_e/ \sqrt m_e) + ( \underline{b}_i/ \sqrt m_i) \bigr) \,
{}^t (\cW_0 - {\rm i} \, \underline{b} \, E_0 ), 0 \Bigr) \, . \end{array} \]
At the end, we obtain
$ \cF_2 = {}^t (0,0, {}^t \cF_2^{u_e}, {}^t \cF_2^{u_i},{}^t \cF_2^E, 0) $ with
\[ \begin{array}{l}\displaystyle \cF_2^{u_e} := - (\underline{b}_e / 4  \, \underline{b}^4) \ \Bigl( \bigl( - (\underline{b}_e^2/ m_e) + ( \underline{b}_i^2/ m_i) \bigr) \ \cW_0 \times B_0 +\  \underline{b}^2 \bigl((\underline{b}_e/ \sqrt m_e) + ( \underline{b}_i/ \sqrt m_i) \bigr) \ c \,  E_0 \Bigr) \, , \smallskip \\
\displaystyle \cF_2^{u_i} := + (\underline{b}_i / 4  \, \underline{b}^4) \ \Bigl( \bigl( - (\underline{b}_e^2/ m_e) + ( \underline{b}_i^2/ m_i) \bigr) \ \cW_0 \times B_0 + \ \underline{b}^2 \bigl((\underline{b}_e/ \sqrt m_e) + ( \underline{b}_i/ \sqrt m_i) \bigr)  \ c \,  E_0 \Bigr) \,, \smallskip \\
\displaystyle \cF_2^{E} := + (1 / 4  \, \underline{b}^2) \ \Bigl( \bigl( - (\underline{b}_e^2/ m_e) + ( \underline{b}_i^2/ m_i) \bigr) \ E_0 \times B_0 -  \bigl((\underline{b}_e/ \sqrt m_e) + ( \underline{b}_i/ \sqrt m_i) \bigr) \, c(\varrho_{e0},\varrho_{i0}) \ \cW_0 \Bigr) \, .
   \end{array}
\]
Let us consider the equation
$ \part_t \cU_0 = \cF_2 (\cU_0) $. The density components $ \varrho_{s0} $ remain unchanged, while the time evolution of $ B_0 $ is not modified\footnote{It can be  altered by the other source terms $ \cF_* $ and the quasilinear terms.} by $ \cF_2 $. We can assert that
\[ \varrho_{s0} (t) = \varrho_{s00} := {\varrho_{s0}}_{\mid t=0} \, , \qquad c\bigl( \varrho_{e0} (t) , \varrho_{i0} (t) \bigr) = c_0 := c ( \varrho_{e00} , \varrho_{i00} ) \qquad \forall t \in \RR \, . \]
From there, we can deduce that
\begin{equation} \label{electriccoupling}
\left \{ \begin{array}{l}
             \part_t \cW_0 = - \alpha \ \cW_0 \times B_{00} + \beta \ E_0 \, , \vspace{ 2pt} \\
             \part_t E_0 \, = - \alpha \ E_0 \times B_{00} - (\beta /\underline{b}^2) \ \cW_0 \, ,
            \end{array}
\right.
\end{equation}
where the coefficients $\alpha$ and  $\beta$ are given by
\begin{equation}
  \label{coefabg}
  \alpha=\frac{\underline{n}}{4\underline{b}^2}\bigg(\frac{1}{m_e^2}- \frac{1}{m_i^2}\bigg)\,, \quad
  \beta=\frac{1}{4}\bigg(\frac{\underline{b}_e}{ \sqrt m_e} + \frac{ \underline{b}_i}{ \sqrt m_i} \bigg) \, c_0 \,.
\end{equation}
Observe that  $ \alpha \in \RR_+^* $ (because $m_e < m_i$), while the sign of $ \beta \in \RR$ depends on the signs of the chosen initial densities $ \varrho_{e00} $ and $ \varrho_{i00}$. 
We see that
\[  \vert \cW_0 (t) \vert^2 /\underline{b}^2+ \ \vert E_0 (t) \vert^2 =  \vert \cW_0 (0) \vert^2 /\underline{b}^2 +  \vert E_0 (0) \vert^2 \, , \qquad \forall t \in \RR_+ \, . \]
The above nonnegative quantity plays the part of an energy. Both $ \cW_0 $ and $ E_0 $ remain bounded. But exchanges (rotations) between $ \cW_0 $ and $ E_0 $ do occur. 
Assume that 
\[ E_{00} := {E_0}_{\mid t=0} = 0 \, , \qquad \cW_{00} := {\cW_0}_{\mid t=0} \not = 0 \, , \qquad c_0 \not = 0 \ (\Rightarrow \beta \not = 0 ) \, . \]
It follows that $ \part_t E_0 \not = 0 $ so that $ E_0 (t)\not = 0 $ for small enough values of $ t \in \RR_+ $.
In the unprepared situation, a defect of neutrality (embodied in the condition $ c_0 \not = 0 $) and the consideration of two-fluid effects (represented by $ \cW_0 \not = 0 $) is correlated with the presence of a non-trivial electric component $ E_0 \not = 0 $. This already indicates that FLM cannot coincide with SLM. It is also clear that, due to resonances, the kinetic energy of charged particles can be converted into electrical energy (and conversely).


\subsection{Case 3}\label{apportcase3} Here, we have to deal with
\[ \begin{array}{rl}
\cF_3 := \! \! \! & \displaystyle
\sum_{
\sigma' \in \{ -,+ \}} \ \sum_{\jmath'=1}^3 \ \sum_{
0 \not = k} 2 \ \tilde \delta_0^0 \, \tilde \delta_k^{\jmath'} \, \tilde \delta_k^{\jmath'} \ \Re \, \bigl \lbrack \tilde \Pi^0_0 \, \cB \bigl(\cI_{\sigma'} (\tilde
\Pi^{\jmath'}_{\sigma' k} \cU_{\sigma' k} ) , \cI \,  \cI_{\sigma'} (\tilde \Pi^{\jmath'}_{\sigma' k} \cU_{\sigma' k} ) \bigr) \bigr \rbrack \, . \\
= \! \! \! & \displaystyle
\sum_{
\sigma' \in \{ -,+ \}} \ \sum_{\jmath'=1}^3 \ \sum_{
0 \not = k} \tilde \Pi^0_0 \, \bigl \lbrack f \bigl( \Re (\tilde
\Pi^{\jmath'}_{\sigma' k} \cU_{\sigma' k} ) \bigr )
+ f \bigl( \Im (\tilde
\Pi^{\jmath'}_{\sigma' k} \cU_{\sigma' k} ) \bigr ) \bigr \rbrack \, .
\end{array} \]
Changing $ k $ into $ -k $ allows to absorb $ \sigma'= - $. There remains
\[ \cF_3 = 2 \sum_{\jmath'=1}^3 \ \sum_{
0 \not = k} \tilde \Pi^0_0 \, \bigl \lbrack f \bigl( \Re (\tilde
\Pi^{\jmath'}_{k} \cU_{k} ) \bigr )
+ f \bigl( \Im (\tilde
\Pi^{\jmath'}_{k} \cU_{k} ) \bigr ) \bigr \rbrack = 2 \sum_{\jmath'=1}^3 \ \sum_{
0 \not = k} \tilde \Pi^0_0 \, \bigl \lbrack f \bigl( \Re (\tilde
\Pi^{\jmath'}_{k} \cU_{k} ) \bigr )
+ f \bigl( \Re (\tilde
\Pi^{\jmath'}_{k} {\rm i}\, \cU_{k} ) \bigr ) \bigr \rbrack \, .\]
There is an important difference between the status of $ \cF_1 $, $ \cF_2 $ and $ \cF_4 $ on the one hand, and $ \cF_3 $ on the other hand. The amplitude of $ \cF_j $ with $ j \in \{ 1,2,4\} $ is bounded by the $ L^2 $-norm of the function $ \cU $, while the size of $ \cF_3 $ may be  not\footnote{In the presence of an $ L^2 $-estimate on $ \cU $ but without enough $ H^s $-regularity (with $ s > 0 $ large enough), the sum defining $ \cF_3 $ may diverge. The expression $ \cF_3 $ is sensitive to what happens at high frequencies.}. When dealing with solutions having a low level of regularity, the impact of $ \cF_3 $ may be uncontrolled.
But the existence of non trivial (non prepared) resonances leading to $ \cF_3 $ does not preclude the elimination of $ \cF_3 $ due to diverse cancellations. This is the phenomenon of {\it transparency} in nonlinear geometric optics, which is verified concerning $ \cF_3 $ as stated below.

\begin{lemma}[Transparency of the part of the source term involving unbounded frequencies]\label{transpa3}
  We have $\cF_3 = 0 $.
\end{lemma}

\begin{proof} The choice of $ \jmath' \not = 0 $ corresponds either to the selection of the symbol $ \perp $ or $ \mypar $. We start by looking at
\[ \tilde
\Pi_{\perp k} \, \cU_{k} := \gamma_k \ r^1_{\perp k} + \delta_k \ r^2_{\perp k} \, , \qquad \gamma_k := {}^t \bar r^1_{\perp k} \, \cU_k \, , \qquad \delta_k := {}^t \bar r^2_{\perp k} \, \cU_k \, . \]
Coming back to \eqref{cK0kperp1} and  \eqref{cK0kperp2}, we find that
\[ \begin{array}{rr}
 \Re (\tilde
\Pi_{\perp k} \, \cU_{k} ) \! \! \! & \displaystyle = \frac{1}{\sqrt{2(\underline{b}^2+|k|^2)}} \ {}^t
\bigl( 0, 0,  \underline{b}_e \, {}^t (\Re \gamma_k \, \mathbbm{e}^3 + \Re \delta_k \, \mathbbm{e}^2 ), -  \underline{b}_i \, {}^t (\Re \gamma_k \, \mathbbm{e}^3 + \Re \delta_k \, \mathbbm{e}^2 ), \quad \\
& \qquad \qquad \sqrt{\underline{b}^2+|k|^2} \ {}^t (\Im \gamma_k \, \mathbbm{e}^3 + \Im \delta_k \, \mathbbm{e}^2 ) , \vert k \vert \ {}^t (\Im \gamma_k \, \mathbbm{e}^2 - \Im \delta_k \, \mathbbm{e}^3 ) \bigr) \, .
\end{array} \]
Applying Lemma \ref{polarization}, we have to compute the corresponding expression $ \cW \times B $, which is
\[ \begin{array}{rl}
\cW \times B =    -  \underline{b}^2 \ \vert k \vert \, (\Re \gamma_k \, \mathbbm{e}^3 + \Re \delta_k \, \mathbbm{e}^2 ) \times (\Im \gamma_k \, \mathbbm{e}^2 - \Im \delta_k \, \mathbbm{e}^3 ) = \underline{b}^2 \ \vert k \vert \, (\Re \gamma_k \,  \Im \gamma_k  + \Re \delta_k \, \Im \delta_k ) \, \mathbbm{e}^1 \, .
   \end{array} \]
Changing $ \cU_k $ into $ {\rm i} \, \cU_k $ is the same as replacing $ \gamma_k $ and $ \delta_k $ respectively by $ {\rm i} \, \gamma_k $ and $ {\rm i} \, \delta_k $. We still have \eqref{symanepaso} and \eqref{symanepasoconc} so that
\[ \tilde \Pi^0_0 \, \bigl \lbrack f \bigl( \Re (\tilde
\Pi_{\perp k} \, \cU_{k} ) \bigr )
+ f \bigl( \Re (\tilde
\Pi_{\perp k} \, {\rm i}\, \cU_{k} ) \bigr ) = 0 \, . \]
We turn next to the study of $ \tilde \Pi^\star_{\mypar k} \, \cU_{k} := \gamma_k \ r^\star_{\mypar k} $ where $ \gamma_k := {}^t \bar r^\star_{\mypar k} \, \cU_k $ with $ \star \in \{ l,r \} $. The content of $ r^\star_{\mypar k} $ is prescribed at the end of Paragraph \ref{para4}. Since the magnetic component of $ r^\star_{\mypar k} $ is just $ 0 $, from Lemma \ref{polarization} again, we can assert that
\[ \tilde \Pi^0_0 \, \bigl \lbrack f \bigl( \Re (\tilde
\Pi^\star_{\mypar k} \, \cU_{k} ) \bigr )
+ f \bigl( \Re (\tilde
\Pi^\star_{\mypar k} \, {\rm i}\, \cU_{k} ) \bigr ) = 0 \, , \qquad \forall \star \in \{ l,r \} \, . \]
\end{proof}


\subsection{Case 4}\label{apportcase3bis}
By construction, we have 
\begin{equation}\label{c'estF4} 
\cF_4 = \sum_{
\sigma',\sigma''
       \in \{ -,+ \}} \ \sum_{\jmath',\jmath''
       =0}^3 \ \sum_{\bigl \{ k \not = 0 \, ; \, k \in
 R^{1,\jmath',\jmath''}_{\sigma',\sigma''} \bigr \} } \cF_{4,\sigma',\sigma''}^{\jmath',\jmath''} (k) \, ,  
\end{equation}
where
\[ \cF_{4,\sigma',\sigma''}^{\jmath',\jmath''} (k) := 2 \ \tilde \delta_k^{\jmath'} \, \tilde \delta_k^{\jmath''} \ \Re \, \bigl \lbrack \tilde \Pi^1_0 \, \cB \bigl(\cI_{\sigma'} (\tilde
\Pi^{\jmath'}_{\sigma' k} \cU_{\sigma' k} ) , \cI_{\sigma''} (\tilde \Pi^{\jmath''}_{- \sigma'' k} \cU_{- \sigma'' k} ) \bigr) \bigr \rbrack \, . \]
As explained in Subsection \ref{jpournot0} the set $ R^{1,\jmath',\jmath''}_{\sigma',\sigma''} $
is in general empty, so that $ \cF_4 =  0 $.
Still, for well-chosen values of $ \underline{a} $ and $ \underline{b} $, the sum inside $ \cF_4 $ may contain effective contributions. For the sake of completeness, we want here to examine a special situation leading to $ \cF_4 \not =  0 $. To clarify this point, we work under Assumption \ref{theascoin} and we impose \eqref{threesquare}, taking for granted that $ \underline{a} \not = 1 / \sqrt 3 $ as well as $ \underline{a}^{-1} \not = 2 g_\pm ( \underline{a} ) $ (see Lemma~\ref{condabsre} with $|k|=\underline{b}/(2\underline{a})$). Recall from Paragraph \ref{jpournot0sp} that this corresponds to the choice of $ (\jmath',\sigma'') = (0,+) $ - or $ (\jmath'',\sigma') = (0,+) $ - while $ \jmath'' $ - or $ \jmath' $ - must be adjusted in such a way that $ \lambda^{\jmath''}_k = \lambda_{\mypar k}^l $ - or $ \lambda^{\jmath'}_k = \lambda_{\mypar k}^l $. Hence, the selection of the parallel mode $ \mypar $. In this way, we can be sure that the last sum inside \eqref{c'estF4} is reduced to
\[ R_\mypar := \bigl \lbrace \, 0 \not = k \in \ZZ^3 \, ; \, \vert k \vert = \underline{b} / \underline{a} \, \bigr \} \, , \qquad \operatorname{card} R_\mypar >0 \, . \]
There remains to consider
\[ \begin{array}{rl} 
\cF_4 = \! \! \! & \displaystyle \sum_{k \in R_\mypar} \bigl \lbrace \Re \, \bigl \lbrack \tilde \Pi^1_0 \, \cB \bigl(\cI_- (\tilde
\Pi^0_{-k} \, \cU_{-k})  , \tilde \Pi^l_{\mypar (- k)} \, \cU_{- k} \bigr) \bigr \rbrack + 
\Re \, \bigl \lbrack \tilde \Pi^1_0 \, \cB (\tilde
\Pi^0_{k} \, \cU_{k}  , \tilde \Pi^l_{\mypar (- k)} \, \cU_{-k} ) \bigr \rbrack \smallskip \\
\ & \qquad + \Re \, \bigl \lbrack \tilde \Pi^1_0 \, \cB \bigl(\tilde \Pi^l_{\mypar k} \, \cU_{ k} , \cI_- (\tilde \Pi^0_{k} \, \cU_{k})  \bigr) \bigr \rbrack + \Re \, \bigl \lbrack \tilde \Pi^1_0 \, \cB (\tilde \Pi^l_{\mypar k} \, \cU_{ k} , \tilde \Pi^0_{-k} \, \cU_{-k}) \bigr \rbrack \bigr \rbrace \, .
\end{array} \]
Since $ \cI_{-} (\tilde
\Pi^0_{- k} \, \cU_{- k} ) = \tilde \Pi^0_{k} \, \cU_{k} $ and because $ \cB $ is symmetric, this can be reformulated as 
\[ \cF_4 = 2 \sum_{k \in R_\mypar} \cF_{4k} \, , \qquad \cF_{4k} := \Re \, \bigl \lbrack \tilde \Pi^1_0 \, \cB (\tilde
\Pi^0_{k} \, \cU_{k}  , \tilde \Pi^l_{\mypar (- k)} \, \cU_{- k} ) \bigr \rbrack + 
\Re \, \bigl \lbrack \tilde \Pi^1_0 \, \cB (\tilde
\Pi^0_{-k} \, \cU_{-k}  , \tilde \Pi^l_{\mypar k} \, \cU_{k} ) \bigr \rbrack \, . \]
When defining $ \cF_{4k} $, care has been taken to put the opposite frequencies $ k $ and $ -k $ together in order to ensure that $ \cF_{4k} = \cF_{4(-k)} $. By this way, the contributions inside  $ \cF_{4k} $ coming  from $ +k $ and $ -k $ can (potentially) compensate one another. Let us now concentrate on $ \cF_{4k} $ with $ 0 \not = k \in \NN^3 $ as indicated in \eqref{threesquare}. The starting point is the spectral decomposition \eqref{decompocU}, where we can forget the terms which are eliminated under the projections $ \tilde
\Pi^0_{k} $ and $ \tilde \Pi^l_{\mypar (\pm k)} $, that is\footnote{We take care below to clearly distinguish between the counting index ``$ \ell $'' and the symbol``$ l $''  for left.}
\[ \cU_k = \beta^l_k \ J r^l_{\mypar k} +
\sum_{\ell=1}^4 \, \alpha^\ell_k \ r^{0,\ell}_k +  \gamma^l_k \ r^l_{\mypar k} \, . \]
To guarantee that $ \bar \cU_{-k} = \cU_k $, in view of \eqref{cK0knot0} and \eqref{rparallellk}, we must impose 
\[ \alpha^\ell_{-k} = - \bar \alpha^\ell_k , \quad \forall \ell \in \{1,3 \} \, , \qquad \alpha^\ell_{-k} = \bar \alpha^\ell_k , \quad \forall \ell \in \{2 ,4 \} \, , \qquad \beta^l_{-k} = - \bar \gamma^l_k \, , \qquad \gamma^l_{-k} = - \bar \beta^l_k \, . \]
By construction, we have
\[ \tilde
\Pi^0_{k} \, \cU_{k} = \sum_{\ell=1}^4 \, \alpha^\ell_k \ r^{0,\ell}_k \, , \qquad \tilde \Pi^l_{\mypar (- k)} \, \cU_{- k} = \gamma^l_{-k} \ r^l_{\mypar (-k)} \, . \]
Let us consider
\[ \cB (\tilde
\Pi^0_{k} \, \cU_{k}  , \tilde \Pi^l_{\mypar (- k)} \, \cU_{- k} ) = \sum_{\ell=1}^4 \alpha^\ell_k \ \gamma^l_{-k} \ \cB(r^{0,\ell}_k , r^l_{\mypar (-k)} ) \, . \]
In our context, since $ \underline{d}^l_k = - \underline{b}_e^2 $, we deal with
\begin{equation}
  \label{eqn:newrpl}
\begin{array}{rl}  
{}^t r^l_{\mypar k} \! \! \! & = c^l_k \ \underline{a} \ \underline{b}_e \ \vert k \vert \  (-\underline{b}_i,-\underline{b}_e,\underline{b}_i \,  {}^t \mathbbm{e}^1 , \underline{b}_e \,  {}^t \mathbbm{e}^1  , 0,0) \smallskip \\
\ & = (1/ \sqrt 2 \, \underline{b}) \ {}^t (-\underline{b}_i,-\underline{b}_e,\underline{b}_i \,  {}^t \mathbbm{e}^1,\underline{b}_e \, {}^t  \mathbbm{e}^1, 0,0)  =  {}^t \bar r^l_{\mypar k} \, .
\end{array}
\end{equation}
From \eqref{defdefcUbi}, knowing that the density components of $ \cU $ are zero (as in $ \tilde
\Pi^0_{k} \, \cU_{k} $) and that the magnetic component of $ \tilde \cU $ is zero (as in $ \tilde
\Pi^{\jmath''}_{-k} \, \cU_{-k} = \tilde
\Pi^l_{\mypar (-k)} \, \cU_{-k} $), we find that
\begin{equation}
  \label{eqn:BUtU}
\cB (\cU,\tilde \cU) = {}^t \Bigl( 0,0, - \frac{{}^t (\tilde u_e \times B )}{m_e}, \frac{{}^t (\tilde u_i \times B)}{m_i},\frac{\tilde \varrho_e \,  {}^t u_e}{\sqrt{m_e \, \underline{p}'_e}} - \frac{\tilde \varrho_i \,  {}^t u_i }{\sqrt{m_i \, \underline{p}'_i}}, 0
\Bigr) .
\end{equation}
With $ r^{0,1}_k $ as in \eqref{cK0knot0}, we get that
\[ \cB ( r^{0,1}_k , r^l_{\mypar (-k)} ) = {}^t (0,0,0,0 , \cB^E , 0 ) \]
with 
\[ \cB^E := \frac{1}{\sqrt 2 \, \underline{b}^2}  \ \left( \frac{\underline{b}_e^2}{\sqrt {m_i \, \underline{p}'_i}} - \frac{\underline{b}_i^2}{\sqrt {m_e \, \underline{p}'_e}}  \right) \ \mathbbm{e}^2 = \frac{1}{\sqrt 2 \, \underline{b}^2 \, \underline{a}}  \ \left( \frac{\underline{b}_e^2}{m_i} - \frac{\underline{b}_i^2}{m_e}  \right) \ \mathbbm{e}^2 = 0 \, . \]
The case of $ r^{0,2}_k $ is similar. For the contribution associated with $ r^{0,3}_k$, we first recall from the choice of $ {\underline a}_e = {\underline a}_i = {\underline a} $, that we have $ \sqrt{\underline{p}_s' \, m_s} = \underline a \, m_s $. Then, using \eqref{eqn:newrpl} and \eqref{eqn:BUtU}, it suffices to look at
\[ \begin{array}{l}
  \cB( r^{0,3}_k , r^l_{\mypar (-k)} ) = {}^t (\cB^{\varrho_e}, \cB^{\varrho_i},  {}^t \cB^{u_e},   {}^t \cB^{u_i},  {}^t  \cB^{E},  {}^t  \cB^{B}) \\  
 \qquad \qquad \displaystyle = {}^t \bigg (0,0, - \frac{{\rm i} \, \underline{b}_i \,  {}^t\mathbbm{e}^2}{m_e \, \sqrt{ 2\underline{b}^2 + 2| k |^2}}, 
    \frac{{\rm i} \, \underline{b}_e \, {}^t\mathbbm{e}^2}{m_i \, \sqrt{2  \underline{b}^2 + 2  \vert k \vert^2}} ,
   - \bigg( \frac{1}{m_e} + \frac{1}{m_i}\bigg)\frac{\underline{b}_e \, \underline{b}_i \, |k|\, {}^t\mathbbm{e}^2}{\underline{a}\underline{b}^2\sqrt{2  \underline{b}^2 + 2  \vert k \vert^2}},0
      \bigg) .
\end{array} \]  
For the contribution associated with $r^{0,4}_k$, we obtain the same  sort of expression, except that $ \mathbbm{e}^2$ must be everywhere replaced by $ \mathbbm{e}^3$. Remark that the above expression is odd with respect to $ k $ - and it is even when $ \ell = 4 $ - so that
\[ \begin{array}{rl} 
 \cB (\tilde \Pi^0_{k} \, \cU_{k}  , \tilde  \Pi^l_{\mypar (- k)} \, \cU_{- k} ) + \cB  (\tilde \Pi^0_{-k} \, \cU_{-k}  , \tilde  \Pi^l_{\mypar k} \, \cU_{k} ) = \! \! \! &  + (\alpha^3_k \, \gamma^l_{-k} - 
 \alpha^3_{-k} \, \gamma^l_{k} ) \ \cB( 
r^{0,3}_k , r^l_{\mypar (-k)} )  \\
 \ & +  (\alpha^4_k \, \gamma^l_{-k} + 
  \alpha^4_{-k} \, \gamma^l_{k} ) \ \cB( 
 r^{0,4}_k , r^l_{\mypar (-k)} ) \, , 
\end{array} \]
which is even as expected. There is no further cancellation coming from $ \tilde \Pi^1_0 $. As a matter of fact, coming back to \eqref{pi10}, it suffices to check that the expression $ \cW - {\rm i} \, \underline{b} \, E $ involved in the present context is non-trivial. Knowing that $ \vert k \vert =
\underline{b}/\underline{a} $, this means to look at
\begin{equation}
  \label{eqn:nzcF4}
-\underline{b}_e \, \cB^{u_e} + \underline{b}_i \, \cB^{u_i} - {\rm i} \, \underline{b} \, \cB^E =  \bigg( \frac{1}{m_e} + \frac{1}{m_i}\bigg)\bigg( 1 + \frac{1}{\underline a^2}\bigg) 
\frac{{\rm i} \, \underline{b}_e \, \underline{b}_i}{\sqrt{2  \underline{b}^2 + 2  \vert k \vert^2}} \ \mathbbm{e}^2 \, .
\end{equation}
A similar calculation holds concerning $ \ell = 4 $, yielding this time a vector pointing in the direction of $ \mathbbm{e}^3 $ which therefore cannot compensate the above result. Since $ \alpha^3_k \, \gamma^l_{-k} - \alpha^3_{-k} \, \gamma^l_{k} $ can be adjsuted as wanted,   an effective contribution can be produced in this way. This means that the presence of a coupling inside \eqref{evolutionU0} of $ \cU_0 $  with the modes $ \cU_k $ can become fully apparent. But again, this is an  exceptional situation, which generically does not occur.


\section{Transparencies related to the quasilinear term}\label{transQL}
   
Here, we focus on the quasilinear term of \eqref{modueffective}, that is, we consider
\begin{equation}
  \label{modueffectiveQL}
 \part_t \cU = \MM_\tau \bigl \lbrack e^{- \tau \, \PG \cL \PG} \, \PG Q(e^{\tau \, \PG \cL \PG} \, \PG \cU,e^{\tau \, \PG \cL \PG} \, \PG \cU) \bigr \rbrack \,,
\end{equation}
where
\begin{equation}
  \label{defQ}
  Q(\cU,\tilde\cU)= - A(\cU,D_x)\tilde\cU\,, \quad A(\cU,D_x)\tilde\cU= \sum_{p=1}^3 A_p(\cU)\partial_{x_p} \tilde\cU\,,
\end{equation}  
and where, using \eqref{retainforbaras}, we have implemented
\begin{equation}
  \label{defA}  
  A_p(\cU)=\left( \begin{array}{cccccc}
\displaystyle \frac{u_e^p}{\sqrt{\underline{n} \, m_e}} & 0 & \displaystyle \frac{a_e'(0)}{\sqrt{\underline{n} \, m_e}} \, \varrho_e {}^te^p & 0 & 0 & 0 \\
   0 & \displaystyle \frac{u_i^p}{\sqrt{\underline{n}\, m_i}} & 0 & \displaystyle \frac{a_i'(0)}{\sqrt{\underline{n} \, m_i}} \, \varrho_i \, {}^te^p  & 0 & 0 \\
  \displaystyle  \frac{a_e'(0)}{\sqrt{\underline{n} \, m_e}} \, \varrho_e \,  e^p  & 0 & \displaystyle \frac{u_e^p}{\sqrt{\underline{n} \, m_e}} \, e^p\otimes e^p & 0 & 0 & 0 \\
   0 & \displaystyle  \frac{a_i'(0)}{\sqrt{\underline{n} \, m_i}} \, \varrho_i \, {}^te^p & 0 & \displaystyle \frac{u_i^p}{\sqrt{\underline{n} \, m_i}} \, e^p\otimes e^p  & 0 & 0 \\
    0 & 0 & 0 & 0 &  0 & 0\\
    0 & 0 & 0 & 0 &  0 & 0
  \end{array}  \right). 
\end{equation}
Observe that the matrices $ A_p $ are linear functions of the unknown $\cU:={} ^t(\varrho_e,\varrho_i,{}^tu_e,{}^tu_i,{}^tE,{}^tB)$. From \eqref{appliscatbis}-\eqref{defQ}-\eqref{defA}, we obtain
\[ \begin{array}{l}
 - Q(e^{\tau \, \PG \cL \PG} \, \PG \,  \cU,e^{\tau \, \PG \cL \PG} \, \PG \, \cU) \\ \qquad = \displaystyle \sum_{p=1}^3\ \sum_{k,m\,\in\ZZ^3} e^{{\rm i}(k+m)\cdot x} \sum_{\sigma',\sigma'' \in \pm}
  {\rm i} \, m_p \, A_p\big(\cI_{\sigma'}(\cY_{\sigma' k}(\tau) \, \cU_{\sigma'k})\big) \ \cI_{\sigma''}(\cY_{\sigma'' m}(\tau) \, \cU_{\sigma''m}) \\
\qquad \displaystyle = -\sum_{\sigma',\sigma''\in \pm }\  \sum_{n\in\ZZ^3} Q_{n,\sigma',\sigma''}(\tau)e^{{\rm i}n\cdot x}\,, 
\end{array} \]
where, exploiting \eqref{formcY}, we have isolated the Fourier coefficient
\begin{equation}
\label{defqnspspp}
\begin{array}{rl}
Q_{n,\sigma',\sigma''}(\tau)= \! \! \! & \displaystyle - \sum_{p=1}^3 \ \sum_{k+m=n}\  \sum_{\jmath',\jmath''=0}^3 \\
\ & {\rm i} \, m_p \, \tilde \delta_k^{\jmath'} \, \tilde \delta_m^{\jmath''} \, 
A_p\big(\cI_{\sigma'}(\tilde\Pi_{\sigma'k}^{\jmath'}\cU_{\sigma'k})\big) \ \cI_{\sigma''}(\tilde\Pi_{\sigma''m}^{\jmath''}\cU_{\sigma''m})\
e^{{\rm i} \tau (\sigma'\lambda_k^{\jmath'}+\sigma''\lambda_m^{\jmath''})}\,. 
\end{array}
\end{equation}
From \eqref{formcY}-\eqref{appliscatbis}, we then obtain
\[ \begin{array}{rl}
  \partial_t \cU_n = \! \! \! & \displaystyle  \MM_\tau \bigg \lbrack \sum_{\sigma, \sigma',\sigma'' \in \pm } \ \sum_{\jmath=0}^3 \tilde \delta_n^{\jmath} 
 \, \cI_{\sigma}(\tilde\Pi_{\sigma n}^{\jmath}Q_{\sigma n,\sigma',\sigma''}(\tau)) \ e^{-{\rm i} \tau \sigma\lambda_n^{\jmath}} \bigg \rbrack \\
  = \! \! \! & \displaystyle - \sum_{\sigma, \sigma',\sigma''\in \pm} \ \sum_{\jmath, \jmath', \jmath''=0}^3\  \sum_{k+m=\sigma n}\  \sum_{p=1}^3 \\
& \displaystyle {\rm i} \, \sigma \, m_p \,
  \tilde \delta_n^{\jmath} \, \tilde \delta_k^{\jmath'} \, \tilde \delta_m^{\jmath''} \ 
  \cI_{\sigma}\Big \lbrack \tilde\Pi_{\sigma n}^{\jmath} \big( A_p(\cI_{\sigma'}(\tilde\Pi_{\sigma' k}^{\jmath'} \cU_{\sigma' k})) \big) \,
  \cI_{\sigma''}(\tilde\Pi_{\sigma'' m}^{\jmath''} \cU_{\sigma'' m}) \Big \rbrack \ 
\MM_\tau \big[ e^{{\rm i} \sigma \tau (\sigma'\lambda_k^{\jmath'}+\sigma''\lambda_m^{\jmath''}- \lambda_n^{\jmath} )}\big] \,.
\end{array} \]  
In the previous expression the non vanishing contributions are defined by the same resonance set as the one defined in the introduction of Section \ref{sec:resonances}, that is $\cR_{\sigma,\sigma',\sigma''}^{\jmath, \jmath', \jmath''}(n)$. Therefore, we get
\begin{equation}
  \label{eqnQL1n}
\begin{array}{rl} \partial_t \cU_n = \! \! \! & \displaystyle - 
 \sum_{\sigma, \sigma',\sigma''\in \pm} \ \sum_{\jmath, \jmath', \jmath''=0}^3 \   \sum_{(k,m)\in \cR_{\sigma,\sigma',\sigma''}^{\jmath, \jmath', \jmath''}(n)}\  \sum_{p=1}^3 \\
\displaystyle & \quad \displaystyle \tilde \delta_n^{\jmath} \, \tilde \delta_k^{\jmath'} \, \tilde \delta_m^{\jmath''} \, {\rm i} \, \sigma \,  m_p \
 \cI_{\sigma}\Big \lbrack \tilde\Pi_{\sigma n}^{\jmath} \big( A_p(\cI_{\sigma'}(\tilde\Pi_{\sigma' k}^{\jmath'} \cU_{\sigma' k})) \big) \ 
 \cI_{\sigma''}(\tilde\Pi_{\sigma'' m}^{\jmath''} \cU_{\sigma'' m})\Big \rbrack \,. 
 \end{array}
\end{equation}  
Expanding the term in the bracket of \eqref{eqnQL1n}, we find
\begin{multline}
\cI_{\sigma}\Big \lbrack \tilde\Pi_{\sigma n}^{\jmath} \big( A_p(\cI_{\sigma'}(\tilde\Pi_{\sigma' k}^{\jmath'} \cU_{\sigma' k})) \big)
\cI_{\sigma''}(\tilde\Pi_{\sigma'' m}^{\jmath''} \cU_{\sigma'' m})\Big \rbrack 
= \sum_{l=1}^{\tilde M_n^\jmath} \sum_{l'=1}^{\tilde M_k^{\jmath'}}\sum_{l''=1}^{\tilde M_m^{\jmath''}}  \cI_{\sigma}(r_{\sigma n}^{\jmath,l})  \\
\cI_{\sigma\sigma'}(\langle r_{\sigma' k}^{\jmath',l'},\cU_{\sigma' k}\rangle) \ 
\cI_{\sigma\sigma''}(\langle r_{\sigma'' m}^{\jmath'',l''},\cU_{\sigma''m}\rangle) \ 
\big\langle \cI_{\sigma}(r_{\sigma n}^{\jmath,l}),A_p\big(\cI_{\sigma\sigma'}(r_{\sigma' k}^{\jmath',l'})\big) \,  \cI_{\sigma\sigma''}(r_{\sigma'' m}^{\jmath'',l''})
\big\rangle\,.
\end{multline}
Substituting this expression into \eqref{eqnQL1n}, we finally obtain
\begin{multline}
  \label{eqnQL2n}
  \partial_t \cU_n = - \sum_{\sigma, \sigma',\sigma''\in \pm} \ \sum_{\jmath, \jmath', \jmath''=0}^3\
  \sum_{\underset{\cR_{\sigma,\sigma',\sigma''}^{\jmath, \jmath', \jmath''}(n)}{(k,m)\in}}\  \sum_{p=1}^3 \
  \sum_{l=1}^{\tilde M_n^\jmath} \sum_{l'=1}^{\tilde M_k^{\jmath'}}\sum_{l''=1}^{\tilde M_m^{\jmath''}} 
   \tilde \delta_n^{\jmath} \, \tilde \delta_k^{\jmath'} \, \tilde \delta_m^{\jmath''} \ {\rm i} \, \sigma \,  m_p \
   \cI_{\sigma}(r_{\sigma n}^{\jmath,l}) \\
   \cI_{\sigma\sigma'}(\langle r_{\sigma' k}^{\jmath',l'},\cU_{\sigma' k}\rangle)
  \ \cI_{\sigma\sigma''}(\langle r_{\sigma'' m}^{\jmath'',l''},\cU_{\sigma''m}\rangle)
  \ \big\langle \cI_{\sigma}(r_{\sigma n}^{\jmath,l}),A_p\big(\cI_{\sigma\sigma'}(r_{\sigma' k}^{\jmath',l'})\big) \, \cI_{\sigma\sigma''}(r_{\sigma'' m}^{\jmath'',l''})
   \big\rangle\,.
\end{multline}
For $n=0$, \eqref{eqnQL2n} becomes, with $ R_{\sigma',\sigma''}^{\jmath, \jmath', \jmath''}=\cR_{\sigma,\sigma',\sigma''}^{\jmath, \jmath', \jmath''}(0)$, 
\begin{align}
  \label{eqnQL20}   
  \partial_t \cU_0  = & \sum_{\sigma, \sigma',\sigma''\in \pm} \ \sum_{\jmath, \jmath', \jmath''=0}^3\
  \sum_{k \in 
  R_{\sigma',\sigma''}^{\jmath, \jmath', \jmath''}} \sum_{p=1}^3 \
  \sum_{l=1}^{\tilde M_0^\jmath} \sum_{l'=1}^{\tilde M_k^{\jmath'}}\sum_{l''=1}^{\tilde M_k^{\jmath''}} 
   \tilde \delta_0^{\jmath} \, \tilde \delta_k^{\jmath'} \, \tilde \delta_k^{\jmath''} \ {\rm i} \, \sigma \, k_p \ \cI_{\sigma}(r_{0}^{\jmath,l})  \\
   &\  \cI_{\sigma\sigma'}(\langle r_{\sigma' k}^{\jmath',l'},\cU_{\sigma' k}\rangle) \ 
     \cI_{\sigma\sigma''}(\langle r_{-\sigma'' k}^{\jmath'',l''},\cU_{-\sigma''k}\rangle) \ 
   \big\langle \cI_{\sigma}(r_{0}^{\jmath,l}),A_p\big(\cI_{\sigma\sigma'}(r_{\sigma' k}^{\jmath',l'})\big) \,  \cI_{\sigma\sigma''}(r_{-\sigma'' k}^{\jmath'',l''})
   \big\rangle \nonumber  \\
    = & \ \cQ_{1} \ +\  \cQ_{2} \ +\  \cQ_{3} \ +\ \cQ_4  \,, \nonumber
\end{align}
where the subscript $ a $ inside $ \cQ_a $ means that ``Case $ a $'' of Subsection \ref{sec:resonances}.$a $ is activated. Below, we discuss separately the contents of the $\cQ_a $.\\ 

Before going further we observe the following simple property that we will use repeatedly.

\begin{lemma}[Commutation properties associated with the velocity components of $ A_p(\cU) \, \tilde \cU  $] \label{lemQCPA}
  Given $ \cU:={} ^t(\varrho_e,\varrho_i,{}^tu_e,{}^tu_i,E,B)$ and
  $ \tilde\cU:={} ^t(\tilde \varrho_e,\tilde \varrho_i,{}^t\tilde u_e, {}^t \tilde u_i,\tilde E,\tilde B)$, we have
\begin{equation}
  \label{Aquasicommut}
  A_p(\cU) \, \tilde \cU =  \left(\begin{array}{c}
\displaystyle \frac{1}{\sqrt{\underline{n} \, m_e}} \, \bigl(\tilde \varrho_e \, u_e^p + {a}_e'(0)\, \varrho_e \, \tilde u_e^p \bigr) \smallskip \\
\displaystyle  \frac{1}{\sqrt{\underline{n} \, m_i}}  \,  \bigl(\tilde \varrho_i \, u_i^p + {a}_i'(0)\, \varrho_i \, \tilde u_i^p \bigr) \smallskip \\
\displaystyle  \frac{1}{\sqrt{\underline{n} \, m_e}} \, \bigl( a_e'(0) \, \varrho_e \, \tilde\varrho_e  + u_e^p \, \tilde u_e^p  \bigr) \ e^p \smallskip \\
\displaystyle \frac{1}{\sqrt{\underline{n} \, m_i}}  \,    \bigl( a_i'(0) \, \varrho_i \,  \tilde \varrho_i  + u_i^p \, \tilde u_i^p  \bigr) \  e^p \\
    0\\
    0
    \end{array}
  \right) \,.
\end{equation}
The expression $ A_p(\cU) \, \tilde \cU $ does not depend on $ (E,B) $ and $ (\tilde E,\tilde B) $. The third and fourth  components of the vectors $ A_p(\cU) \, \tilde \cU $ and $ A_p(\tilde \cU) \, \cU $ are the same. Moreover, if the two density components of both $ \cU $ and $ \tilde \cU $ are zero, then $ A_p(\cU) \, \tilde \cU = A_p(\tilde \cU) \, \cU  $.
\end{lemma}

\subsection{Case 1}\label{apportcase1Q}
The part $ \cQ_1$ is associated with the resonances set $R_{\sigma',\sigma''}^{0, 0, 0}$; it corresponds to the time evolution of the spatial mean value of the solutions to extended magnetohydrodynamics. It is thereby given by
\begin{align}
  \cQ_1 =  &\ \frac{1}{8} \sum_{\sigma, \sigma',\sigma''\in \pm} \ \sum_{k\in\ZZ^3} \ 
  \sum_{p=1}^3 \  \sum_{l=1}^{\tilde M_0^0} \sum_{l'=1}^{\tilde M_k^{0}}\sum_{l''=1}^{\tilde M_k^{0}} 
  {\rm i} \, \sigma \, k_p \ r_{0}^{0,l}  \label{eqn0Q1}\\
     &\  \cI_{\sigma\sigma'}(\langle r_{\sigma' k}^{0,l'},\cU_{\sigma' k}\rangle)
   \  \cI_{\sigma\sigma''}(\langle r_{-\sigma'' k}^{0,l''},\cU_{-\sigma''k}\rangle)
  \ \big\langle r_{0}^{0,l},A_p\big(\cI_{\sigma\sigma'}(r_{\sigma' k}^{0,l'})\big) \,  \cI_{\sigma\sigma''}(r_{-\sigma'' k}^{0,l''})
   \big\rangle\,. \nonumber
\end{align}

\begin{lemma}[Transparency of the part of the quasilinear term  coming from the slow limit model]\label{transpa0Q}
  We have $\cQ_1 = 0 $.
\end{lemma}

\begin{proof}
  From \eqref{eqn0Q1}, the contribution of the case $k=0$ is trivially zero. Now, we consider $k\in \ZZ_*^3=\ZZ^3\! \setminus \!\{0\}$ in \eqref{eqn0Q1}. Expanding \eqref{eqn0Q1}, for $\sigma'$ and $\sigma''$ in the set $\{-,\,+\}$, we obtain
\begin{align*}
  \cQ_1 =  &\ \frac{1}{8} \, \sum_{\sigma\in \pm} \ \sum_{k\in\ZZ^3_*} \ 
  \sum_{p=1}^3 \  \sum_{l=1}^{\tilde M_0^0} \sum_{l'=1}^{\tilde M_k^{0}}\sum_{l''=1}^{\tilde M_k^{0}} 
  {\rm i} \, \sigma \, k_p \ r_{0}^{0,l} \\
     &\  \cI_{\sigma} \big( \langle r_{ k}^{0,l'},\cU_{ k}\rangle \, \langle r_{- k}^{0,l''},\cU_{-k}\rangle \,
  \langle r_{0}^{0,l},A_p(r_{k}^{0,l'}) r_{- k}^{0,l''}\rangle +
  \langle r_{k}^{0,l'},\cU_{k}\rangle \, \langle \bar r_{k}^{0,l''},\bar \cU_{k}\rangle \, 
  \langle r_{0}^{0,l},A_p(r_{k}^{0,l'}) \bar r_{k}^{0,l''}\rangle  \\
  & \ + \, \langle \bar r_{ -k}^{0,l'},\bar \cU_{ -k}\rangle \, \langle r_{- k}^{0,l''},\cU_{-k}\rangle \, 
  \langle r_{0}^{0,l},A_p(\bar r_{-k}^{0,l'}) r_{-k}^{0,l''}\rangle +
  \langle \bar r^{0,l'}_{-k},\bar \cU_{-k}\rangle \, \langle \bar r_{k}^{0,l''},\bar \cU_{k}\rangle \, 
  \langle r_{0}^{0,l},A_p(\bar r^{0,l'}_{-k}) \bar r_{k}^{0,l''}\rangle 
  \big)\,. \nonumber
\end{align*}
Using properties \eqref{cK0knot0} for $\{r_k^{0,l}\}_{l=1,\ldots,4}$, we obtain
\begin{equation*}
  \cQ_1 =   \frac{1}{2} \, \sum_{\sigma\in \pm} \ \sum_{k\in\ZZ^3_*} \ 
  \sum_{p=1}^3 \  \sum_{l=1}^{\tilde M_0^0} \sum_{l'=1}^{\tilde M_k^{0}}\sum_{l''=1}^{\tilde M_k^{0}} 
  {\rm i} \, \sigma \, k_p \ r_{0}^{0,l} \, 
     \\ \cI_{\sigma} \big(\langle r_{ k}^{0,l'},\cU_{ k}\rangle \, \langle r_{- k}^{0,l''},\cU_{-k}\rangle \,
  \langle r_{0}^{0,l},A_p(r_{k}^{0,l'}) r_{- k}^{0,l''}\rangle\, \big).
\end{equation*}
By doubling this multiple sum (divided by two), in the second replica  we can swap the (mute) indices $l'$ and $l''$ (because they run through the same set $\tilde M_k^{0}$) and we can change $k $ into $ -k$ (which means to replace $ k_p $ by $ - k_p $) in order to obtain
\begin{align*}
  \cQ_1 =  &\ \frac{1}{4} \, \sum_{\sigma\in \pm} \ \sum_{k\in\ZZ^3_*} \ 
  \sum_{p=1}^3 \  \sum_{l=1}^{\tilde M_0^0} \sum_{l'=1}^{\tilde M_k^{0}}\sum_{l''=1}^{\tilde M_k^{0}} 
  {\rm i} \, \sigma \, k_p \ r_{0}^{0,l} \,   \\
     &\ \cI_{\sigma} \big(\langle r_{ k}^{0,l'},\cU_{ k}\rangle \, \langle r_{- k}^{0,l''},\cU_{-k}\rangle \,
  \langle r_{0}^{0,l},A_p(r_{k}^{0,l'}) r_{- k}^{0,l''} - A_p(r_{-k}^{0,l''}) r_{k}^{0,l'}\rangle\big)\,.\nonumber
\end{align*}
Remark that the first two (density) components of all  vectors $\{r_k^{0,l}\}_{l=1,\ldots,4}$ are zero. Using Lemma~\ref{lemQCPA}, it follows that $\cQ_1=0$.
\end{proof}

\subsection{Case 2}\label{apportcase2Q} The part $ \cQ_2 $ collects all (non-prepared) resonances implying the frequency $ k = 0 $. In view of formula \eqref{eqnQL20}, where every term of the sum is multiplied by $k_p$ for some $p$ in the set $\{1,2,3\}$, the corresponding contribution is zero and we have the following lemma.

\begin{lemma}[Transparency of the part of the quasilinear term coming from resonances implying the frequency $k=0$]\label{transpa2Q}
 We have $\cQ_2 = 0 $.
\end{lemma}

\subsection{Case 3}\label{apportcase3Q}
The part  $ \cQ_3$ is associated with the set of resonances $R_{\sigma',\sigma''}^{0, \jmath', \jmath''}$, which involves an infinite number of frequencies $k\in \ZZ_*^3$.

\begin{lemma}[Transparency of the part of the quasilinear term involving unbounded frequencies]\label{transpa3Q}
  We have $\cQ_3=0$.
\end{lemma}  

\begin{proof}
The set of resonances $R_{\sigma',\sigma''}^{0, \jmath', \jmath''}$ is the set \eqref{eq:const2} whose cardinality is countably infinite. Using the definition of the set \eqref{eq:const2} and equation \eqref{eqnQL20}, we have
\begin{align*}
  \cQ_{3} =  &\  \frac{1}{2} \, \sum_{\sigma, \sigma'\in \pm} \ \sum_{\jmath'=1}^3\
  \sum_{k\in \ZZ_*^3} \ \sum_{p=1}^3 \
  \sum_{l=1}^{\tilde M_0^0} \sum_{l'=1}^{\tilde M_k^{\jmath'}}\sum_{l''=1}^{\tilde M_k^{\jmath'}} 
    (\tilde \delta_k^{\jmath'})^2 \ {\rm i} \, \sigma \, k_p \ r_{0}^{0,l} \nonumber \\
   &\ \cI_{\sigma\sigma'}(\langle r_{\sigma' k}^{\jmath',l'},\cU_{\sigma' k}\rangle) \, 
     \cI_{-\sigma\sigma'}(\langle r_{\sigma' k}^{\jmath',l''},\cU_{\sigma'k}\rangle) \, 
   \big\langle r_{0}^{0,l},A_p\big(\cI_{\sigma\sigma'}(r_{\sigma' k}^{\jmath',l'})\big) \, \cI_{-\sigma\sigma'}(r_{\sigma' k}^{\jmath',l''})
   \big\rangle \nonumber  \\
   = &\  \frac{1}{2} \, \sum_{\sigma \in \pm} \ \sum_{\jmath'=1}^3\ \sum_{k\in \ZZ_*^3} \ \sum_{p=1}^3 \
   \sum_{l=1}^{\tilde M_0^0} \sum_{l'=1}^{\tilde M_k^{\jmath'}}\sum_{l''=1}^{\tilde M_k^{\jmath'}} 
    (\tilde \delta_k^{\jmath'})^2 \ {\rm i} \, \sigma \, k_p \ r_{0}^{0,l}\, \nonumber \\
   &\ \cI_{\sigma} \big (\langle r_{k}^{\jmath',l'},\cU_{ k}\rangle \, \langle \bar{r}_{k}^{\jmath',l''},\cU_{-k}\rangle \, 
   \langle r_{0}^{0,l}, A_p(r_{k} ^{\jmath',l'}) \, \bar{r}_{k}^{\jmath',l''} \rangle \\
&  \  + \langle \bar{r}_{-k}^{\jmath',l'},\cU_{ k}\rangle \, \langle {r}_{-k}^{\jmath',l''},\cU_{-k}\rangle \, 
   \langle r_{0}^{0,l}, A_p(\bar{r}_{-k}^{\jmath',l'}) \, {r}_{-k}^{\jmath',l''} \rangle \big)\,. \nonumber
\end{align*}
In the second member of this multiple sum,  we can swap the indices $l'$ and $l''$ (because they run through the same set $\tilde M_k^{\jmath'}$) and we can change $k $ into $ -k$ so as to obtain
\begin{align*}
  \cQ_{3} =  &\ \frac{1}{2} \, \sum_{\sigma \in \pm} \ \sum_{\jmath=1}^3\ \sum_{k\in \ZZ_*^3} \ \sum_{p=1}^3 \
   \sum_{l=1}^{\tilde M_0^0} \sum_{l'=1}^{\tilde M_k^{\jmath}}\sum_{l''=1}^{\tilde M_k^{\jmath}} 
    (\tilde \delta_k^{\jmath})^2 \ {\rm i} \, \sigma \, k_p \ r_{0}^{0,l}\,\nonumber \\
   & 
    \ \cI_{\sigma} \big ( \langle r_{k}^{\jmath,l'},\cU_{ k}\rangle \, \langle \bar{r}_{k}^{\jmath,l''},\cU_{-k}\rangle \, 
\big \langle r_0^{0,l},
A_p({r}_{k}^{\jmath,l'}) \, \bar{r}_{k}^{\jmath,l''} - A_p(\bar{r}_{k}^{\jmath,l''}) \, {r}_{k}^{\jmath,l'} \rangle \big)\,. \nonumber
\end{align*}
Let us define
\begin{equation}
  \updelta^l:= \langle r_{0}^{0,l}, w \rangle \, , \qquad w:= A_p({r}_{k}^{\jmath,l'}) \,  \bar{r}_{k}^{\jmath,l''}  - A_p(\bar{r}_{k}^{\jmath,l''}) \, {r}_{k}^{\jmath,l'} \,, \qquad l\in\{1,\ldots, \tilde M_0^0\}\,.
\end{equation}
Since the first eight components of $ {r}_{k}^{\jmath,l'} $ and  $ {r}_{k}^{\jmath,l''} $ are purely real and because \eqref{Aquasicommut} does not involve the electric and magnetic components of these vectors, we obtain 
\[ w:= A_p({r}_{k}^{\jmath,l'}) \,  {r}_{k}^{\jmath,l''}  - A_p( {r}_{k}^{\jmath,l''}) \, {r}_{k}^{\jmath,l'} \, . \]
Since the first two components of the vectors $r_{0}^{0,l}$, for $l\in\{2,\ldots, \tilde M_0^0\}$, are zero, then $\updelta^l=0$ for $l\in\{2,\ldots, \tilde M_0^0\}$. It remains to consider the study of $\updelta^1$, that we split in two cases. On the one hand, when the selection of $ \jmath $ corresponds to the perpendicular case, we have 
\[ \tilde M_k^{\jmath} = 2 \, , \qquad {r}_{k}^{\jmath,l'} \in \{ r_{\perp k}^1 ,r_{\perp k}^2 \} \, , \qquad {r}_{k}^{\jmath,l''} \in \{ r_{\perp k}^1 ,r_{\perp k}^2 \} \, . \]
Since the first two components of $r_{\perp k}^1$ and  $r_{\perp k}^2$ are zero, from  Lemma~\ref{lemQCPA}, we obtain $w=0$ so that $\updelta^1=0$.
On the other hand, when the selection of $ \jmath $ corresponds to the parallel case, we can assert that
\[ \tilde M_k^{\jmath} = 1 \, , \qquad {r}_{k}^{\jmath,l'}={r}_{k}^{\jmath,l''}=r_{\mypar k}^l \quad \text{or} \quad {r}_{k}^{\jmath,l'}={r}_{k}^{\jmath,l''}=r_{\mypar k}^r \, . \]
In both cases, again, $w=0$ so that $\updelta^1=0$. Finally, we have shown that $\cQ_3=0$.
\end{proof}


\subsection{Case 4}\label{apportcase4Q}
The part  $\cQ_4$ is associated with a set of resonances which is finite in number. This set comprises the set of resonances for which $\jmath \neq 0$, for which we state the following lemma.

\begin{lemma}[Transparency of the part of the quasilinear term involving a finite number of resonances]\label{transpa4Q}
  We have $\cQ_4=0$.
\end{lemma}  

\begin{proof}
  As seen in Paragraph~\ref{jpournot0fen}, we have to consider the two cases, $ \mathring{R}_{\sigma',-\sigma'}^{\jmath, \jmath', \jmath''} := R_{\sigma',-\sigma'}^{\jmath, \jmath', \jmath''}\setminus \{0 \}$ with $\jmath'\neq \jmath''$, and $\mathring{R}_{+,+}^{\jmath, \jmath', \jmath''} := R_{+,+}^{\jmath, \jmath', \jmath''}\setminus\{0 \}$. We start by the contribution $\cQ_{41}$ corresponding to $\mathring{R}_{\sigma',-\sigma'}^{\jmath, \jmath', \jmath'}$ with $\jmath'\neq \jmath''$. We then treat the contribution $\cQ_{42}$ which is associated with $\mathring{R}_{+,+}^{\jmath, \jmath', \jmath''}$.\\

\noindent {$\bullet$ Resonances issued from $\mathring{R}_{\sigma',-\sigma'}^{\jmath, \jmath', \jmath''}$ with $\jmath'\neq \jmath''$}.
Expanding the right-hand side of \eqref{eqnQL20} with $\sigma'=\pm$, we obtain
\begin{align}
  \label{eqncQ4-1}   
   \cQ_{41} = &\ \frac 13 \sum_{\sigma \in \pm} \ \sum_{ \jmath'\neq \jmath'' = 0}^3 \ \sum_{\jmath=1}^3  \ \Big\{ \nonumber \\
  &\ \sum_{\underset{\mathring{R}_{+,-}^{\jmath, \jmath', \jmath''}}{k\in}}\  \sum_{p=1}^3 \
  \sum_{l=1}^{\tilde M_0^\jmath} \sum_{l'=1}^{\tilde M_k^{\jmath'}}\sum_{l''=1}^{\tilde M_k^{\jmath''}} 
  \tilde \delta_k^{\jmath'} \, \tilde \delta_k^{\jmath''} \, {\rm i} \, \sigma \, k_p \ \cI_{\sigma}\big(
   r_{0}^{\jmath,l} \langle r_{ k}^{\jmath',l'},\cU_{ k}\rangle \, \langle \bar{r}_{ k}^{\jmath'',l''},\bar{\cU}_{k}\rangle \,
  \langle r_{0}^{\jmath,l}, A_p(r_{k}^{\jmath',l'}) \bar{r}_{k}^{\jmath'',l''}\rangle \big) \nonumber  \\
   &\  +  \sum_{\underset{\mathring{R}_{-,+}^{\jmath, \jmath', \jmath''}}{k\in}}\  \sum_{p=1}^3 \
  \sum_{l=1}^{\tilde M_0^\jmath} \sum_{l'=1}^{\tilde M_k^{\jmath'}}\sum_{l''=1}^{\tilde M_k^{\jmath''}} 
  \tilde \delta_k^{\jmath'} \tilde \delta_k^{\jmath''}\, {\rm i} \, \sigma \, k_p \ \cI_{\sigma}\big(
    r_{0}^{\jmath,l} \langle \bar{r}_{-k}^{\jmath',l'},\bar{\cU}_{-k}\rangle \, \langle {r}_{-k}^{\jmath'',l''},{\cU}_{-k}\rangle \, 
  \langle r_{0}^{\jmath,l}, A_p(\bar{r}_{-k}^{\jmath',l'}) {r}_{-k}^{\jmath'',l''}\rangle \big)
   \Big\}\,.\nonumber
\end{align}
We observe that $\mathring{R}_{+,-}^{\jmath, \jmath', \jmath''}=\mathring{R}_{-,+}^{\jmath, \jmath'', \jmath'}$. Thus, by changing $ (\jmath',l') $ into $ (\jmath'',l'') $ in the second member of the previous expression, after factorization, there remains 
\begin{align}
   \cQ_{41} = &\ \frac13 \ \sum_{\sigma \in \pm} \ \sum_{ \jmath'\neq \jmath'' = 0}^3 \ \sum_{\jmath=1}^3  \
  \sum_{k \in \mathring{R}_{+,-}^{\jmath, \jmath', \jmath''}}\  \sum_{p=1}^3 \
  \sum_{l=1}^{\tilde M_0^\jmath} \sum_{l'=1}^{\tilde M_k^{\jmath'}}\sum_{l''=1}^{\tilde M_k^{\jmath''}} 
   \tilde \delta_k^{\jmath'} \tilde \delta_k^{\jmath''}\, {\rm i} \, \sigma \, k_p \ \cI_{\sigma}\big (  r_{0}^{\jmath,l} \big \{ \nonumber\\
   &\ \langle r_{ k}^{\jmath',l'},\cU_{ k}\rangle \langle \bar{r}_{ k}^{\jmath'',l''},\bar{\cU}_{k}\rangle
   \langle r_{0}^{\jmath,l}, A_p(r_{k}^{\jmath',l'}) \bar{r}_{k}^{\jmath'',l''}\rangle
   +  \langle \bar{r}_{-k}^{\jmath'',l''},\bar{\cU}_{-k}\rangle \langle {r}_{-k}^{\jmath',l'},{\cU}_{-k}\rangle
  \langle r_{0}^{\jmath,l}, A_p(\bar{r}_{-k}^{\jmath'',l''}) {r}_{-k}^{\jmath',l'}\rangle   \big\} \big)\,. \nonumber
\end{align}
Making the change of variables $k $ into $ -k$ in second multiple sum, we obtain
\begin{align}
   \cQ_{41} = &\ \frac13 \, \sum_{\sigma \in \pm} \  \sum_{ \jmath'\neq \jmath'' = 0}^3 \ \sum_{\jmath=1}^3  \ 
  \sum_{k \in \mathring{R}_{+,-}^{\jmath, \jmath', \jmath''}}\  \sum_{p=1}^3 \
  \sum_{l=1}^{\tilde M_0^\jmath} \sum_{l'=1}^{\tilde M_k^{\jmath'}}\sum_{l''=1}^{\tilde M_k^{\jmath''}} 
  \tilde \delta_k^{\jmath'} \, \tilde \delta_k^{\jmath''}\, {\rm i} \, \sigma \, k_p   \nonumber\\
  &\ \cI_{\sigma}\big (  r_{0}^{\jmath,l} \,
  \langle r_{ k}^{\jmath',l'},\cU_{ k}\rangle \, \langle \bar{r}_{ k}^{\jmath'',l''},\bar{\cU}_{k}\rangle \, 
  \langle r_{0}^{\jmath,l}, A_p(r_{k}^{\jmath',l'}) \, \bar{r}_{k}^{\jmath'',l''}
  - A_p(\bar{r}_{k}^{\jmath'',l''}) \, {r}_{k}^{\jmath',l'}\rangle 
  \big)\,. \nonumber
\end{align}
Since the first two components of $ r_0^{\jmath,l}  \equiv r_0^{1,l} $ are zero for all $ (\jmath,l) \in \{1,2,3 \}^2 $, and because the first eight components of $ r_{k}^{\jmath',l'} $ and $ r_{k}^{\jmath'',l''} $ are purely real,  using Lemma~\ref{lemQCPA},
we obtain $ \cQ_{41} = 0 $.
\\
\\
\noindent {$\bullet$ Resonances issued from $\mathring{R}_{+,+}^{\jmath, \jmath', \jmath''}$}. We separate the case $\jmath'\neq\jmath''\neq 0$ yielding $ \cQ_{421} $ from the combination of the situations 
 $(\jmath'= 0,\jmath''\neq 0)$ and $(\jmath'\neq 0,\jmath''= 0)$ leading to 
$\cQ_{422}$ . Obviously we have $\cQ_{42}=\cQ_{421}+\cQ_{422}$. Let us start with 
$ \cQ_{421} $. On the one hand, recall that 
\[ \underline{b} < \lambda_{\perp k} = \sqrt{\underline{b}^2 + \vert k \vert^2} \, , \qquad \forall k \not = 0 \, . \] 
On the other hand, since 
\[ \underline{b}^2 < \Lambda^l + \Lambda^r = \underline{b}^2 + (\underline{a}_e^2 + \underline{a}_i^2) \ \vert k \vert^2 \, , \qquad \forall k \not = 0 \, ,\]
we can assert that 
\[ \underline{b} < \lambda^l_{\mypar k} + \lambda^r_{\mypar k} 
< 2 \, \lambda^r_{\mypar k} \, , \qquad \forall k \not = 0 \, . \]
Looking inside \eqref{Case44} at the resonance conditions defining $\mathring{R}_{+,+}^{\jmath, \jmath', \jmath''}$, we must take $\jmath'=\jmath''=\jmath_*$ where $\jmath_*$ is such that $\lambda_{k}^{\jmath_*}=\lambda_{\mypar k}^l$. Therefore $\cQ_{421} $ is built with
\begin{equation*}
  \label{eqncQ4-4}   
   \cQ_{421} = \frac 13 \sum_{\sigma \in \pm} \ \sum_{\jmath\neq0} \, \sum_{k \in \mathring{R}_{+,+}^{\jmath, \jmath_*, \jmath_*}} \ \sum_{p=1}^3 \
  \sum_{\ell=1}^{\tilde M_0^\jmath} \, {\rm i} \, \sigma \, k_p \ \cI_{\sigma}\big(
   r_{0}^{\jmath,\ell} \langle r_{ \mypar k}^{l},\cU_{ k}\rangle \, \langle {r}_{ \mypar (-k)}^{l},{\cU}_{-k}\rangle \, 
   \langle r_{0}^{\jmath,\ell}, A_p(r_{\mypar k}^{l}) \, {r}_{\mypar (-k)}^{l}\rangle \big)\,.
\end{equation*}
Symmetrizing this expression in $k$, we obtain
\begin{multline*}
  \label{eqncQ4-5}   
   \cQ_{421} = \frac 16 \sum_{\sigma \in \pm} \ \sum_{\jmath\neq0} \, \sum_{k \in \mathring{R}_{+,+}^{\jmath, \jmath_*, \jmath_*}}\ \sum_{p=1}^3 \
  \sum_{\ell=1}^{\tilde M_0^\jmath} \, {\rm i} \, \sigma \, k_p \, \\ 
  \cI_{\sigma}\big(
   r_{0}^{\jmath,\ell} \langle r_{ \mypar k}^{l},\cU_{ k}\rangle \, \langle {r}_{ \mypar (-k)}^{l},{\cU}_{-k}\rangle \, 
   \langle r_{0}^{\jmath,\ell}, A_p(r_{\mypar k}^{l}) \, {r}_{\mypar (-k)}^{l}- A_p(r_{\mypar (-k)}^{l}) \, {r}_{\mypar k}^{l}\rangle \big)\,.
\end{multline*}
Since the first two components of $ r_0^{\jmath,l}  \equiv r_0^{1,l} $ are zero for all $ (\jmath,l) \in \{1,2,3 \}^2 $, applying Lemma~\ref{lemQCPA}, 
we find $\cQ_{421}=0$.

\smallskip 

We now consider what happens when $(\jmath'= 0,\jmath''\neq 0)$ and $(\jmath'\neq 0,\jmath''= 0)$. Since $  \lambda_{\mypar k}^r > \underline{b} $, and $ \lambda_{\perp k} > \underline{b} $, from the resonance conditions of  $\mathring{R}_{+,+}^{\jmath, \jmath', 0}$, we infer that $\jmath'=\jmath_*$ where $\jmath_*$ is such that $\lambda_{k}^{\jmath_*}=\lambda_{\mypar k}^l$. Similarly, from the resonance conditions of  $\mathring{R}_{+,+}^{\jmath, 0, \jmath''}$, we infer that $\jmath''=\jmath_*$ where $\jmath_*$ is such that $\lambda_{k}^{\jmath_*}=\lambda_{\mypar k}^l$. Observe that  $\mathring{R}_{+,+}^{\jmath, \jmath_*, 0}=\mathring{R}_{+,+}^{\jmath, 0,\jmath_*}$. Therefore, the sum of the cases $(\jmath'= 0,\jmath''\neq 0)$, and $(\jmath'\neq 0,\jmath''= 0)$ can be recast as
\begin{align}
   \cQ_{422} = &\ \frac16 \sum_{\sigma \in \pm} \  \sum_{  \jmath\neq0} \
  \sum_{k \in \mathring{R}_{+,+}^{\jmath, \jmath_*, 0}} \  \sum_{p=1}^3 \
  \sum_{\ell=1}^{\tilde M_0^\jmath} \sum_{\ell'=1}^{\tilde M_k^{0}}
   \, {\rm i} \, \sigma \, k_p \ \cI_{\sigma}\big (  r_{0}^{\jmath,\ell} \big \{ \nonumber\\
   &\ \langle r_{\mypar k}^{l},\cU_{ k}\rangle \, \langle {r}_{ -k}^{0,\ell'},{\cU}_{-k}\rangle \, 
   \langle r_{0}^{\jmath,\ell}, A_p(r_{\mypar k}^{\ell}) \, {r}_{-k}^{0,\ell'}\rangle
   +  \langle {r}_{k}^{0,\ell'},{\cU}_{k}\rangle \, \langle {r}_{\mypar (-k)}^{l},{\cU}_{-k}\rangle \, 
  \langle r_{0}^{\jmath,\ell}, A_p({r}_{k}^{0,\ell'}) \, {r}_{\mypar (-k)}^{l}\rangle   \big\} \big)\,. \nonumber
\end{align}
Making the change of $k $ into $ -k$ in the  second multiple sum of the previous equation, we obtain
\begin{align}
  \cQ_{422} = &\ \frac16 \sum_{\sigma \in \pm} \  \sum_{  \jmath\neq0} \
  \sum_{k \in \mathring{R}_{+,+}^{\jmath, \jmath_*, 0}} \  \sum_{p=1}^3 \
  \sum_{\ell=1}^{\tilde M_0^\jmath} \sum_{\ell'=1}^{\tilde M_k^{0}}
   \, {\rm i} \, \sigma \, k_p \  \nonumber\\
  &\ \cI_{\sigma}\big (  r_{0}^{\jmath,\ell}
   \langle r_{\mypar k}^{l},\cU_{ k}\rangle \, \langle {r}_{ -k}^{0,\ell'},{\cU}_{-k}\rangle \,
  \langle r_{0}^{\jmath,\ell}, A_p(r_{\mypar k}^{l}) \, {r}_{-k}^{0,\ell'}-  A_p({r}_{-k}^{0,\ell'}) \, {r}_{\mypar k}^{l}\rangle
  \big)\,. \nonumber
\end{align}
Since the first two components of $r_0^{\jmath,\ell}$ are zero for all $\jmath\in\{1, \ldots,3 \}$ and for all $ \ell\in \{1,\ldots,\tilde M_0^\jmath\}$, using Lemma~\ref{lemQCPA}, we obtain $\cQ_{422}=0$, which concludes the proof.
\end{proof}


\end{document}